\documentclass[11pt]{amsart}
	
	\usepackage[utf8]{inputenc}
	\usepackage{amsmath}
	\usepackage{amsfonts}
	\usepackage{amssymb}
	\usepackage{amsthm}
	\usepackage{mathtools}
	\usepackage[margin=3.0cm]{geometry}
	\usepackage{tikz-cd}
	\usepackage{enumitem}
	\usepackage{hyperref}
	\usepackage{placeins}
	\usepackage[boxsize=1.25em, centerboxes]{ytableau}
\usetikzlibrary{shapes,arrows}
\usepackage{verbatim}
	
	\hypersetup{colorlinks = true, linkbordercolor = {white}, linkcolor = {blue}, citecolor = {blue}}
	\usetikzlibrary{matrix}
	
	\numberwithin{subsection}{section}
	\numberwithin{figure}{section}
	\numberwithin{table}{section}

	\usepackage{pgfplots}
	\usepackage{grffile}
	\pgfplotsset{compat=1.15}
	\usetikzlibrary{plotmarks}
	\usetikzlibrary{arrows.meta}
	\usepgfplotslibrary{patchplots, groupplots}

	\definecolor{mycolor1}{RGB}{57,106,177}

	\newdimen\tikzwidth
	\newdimen\tikzheight
	\pgfplotsset{
    compat=newest,
    scale only axis,
    width=\tikzwidth, height=\tikzheight,
    legend style={font=\footnotesize,},
    legend cell align={left},
    xticklabel style={font=\footnotesize,},
    yticklabel style={font=\footnotesize,},
    x tick label style={
        /pgf/number format/.cd,
        fixed,
        fixed zerofill,
        precision=4,
        /tikz/.cd
    }
    hugeData/.style={
	    every axis plot/.append style={each nth point=100, filter discard warning=false},
    }
}
	
	\pgfdeclarelayer{background}
\pgfsetlayers{background,main}
	
	\pgfmathsetmacro{\myxlow}{-2}
\pgfmathsetmacro{\myxhigh}{2}
\pgfmathsetmacro{\myiterations}{6}

\newtheoremstyle{1}
{6pt} 
{0pt} 
{\itshape} 
{} 
{\bfseries} 
{.} 
{.5em} 
{} 

\newtheoremstyle{2}
{6pt} 
{0pt} 
{} 
{} 
{\bfseries} 
{.} 
{.5em} 
{} 

\theoremstyle{1}

	\newtheorem{theorem}{Theorem}[section]
	\newtheorem*{theorem*}{Theorem}
	\newtheorem{lemma}[theorem]{Lemma}
	\newtheorem{proposition}[theorem]{Proposition}
	\newtheorem{conjecture}[theorem]{Conjecture}

	\newtheorem{corollary}[theorem]{Corollary}
	\newtheorem{definition}[theorem]{Definition}
	\newtheorem*{definition*}{Definition}
	\newtheorem*{proposition*}{Proposition}
	\newtheorem*{question*}{Question}

	\newtheorem{def/prop}[theorem]{Definition/Proposition}

\theoremstyle{2}
	\newtheorem{example}[theorem]{Example}
	\newtheorem{notation}[theorem]{Notation}
	\newtheorem{rmrk}[theorem]{Remark}
		
			\newtheorem*{rmrk*}{Remark}

\newtheorem{innercustomgeneric}{\customgenericname}
\providecommand{\customgenericname}{}
\newcommand{\newcustomtheorem}[2]{%
  \newenvironment{#1}[1]
  {%
   \renewcommand\customgenericname{#2}%
   \renewcommand\theinnercustomgeneric{##1}%
   \innercustomgeneric
  }
  {\endinnercustomgeneric}
}

\newcustomtheorem{customthm}{Theorem}

	\DeclareMathOperator{\Z}{\mathbb{Z}}
	
	\DeclareMathOperator{\Q}{\mathbb{Q}}
	\DeclareMathOperator{\N}{\mathbb{N}}
	\DeclareMathOperator{\F}{\mathbb{F}}
	\DeclareMathOperator{\C}{\mathbb{C}}

	\DeclareMathOperator{\FinVect}{FinVect}
	\DeclareMathOperator{\Hom}{Hom}
	\DeclareMathOperator{\Spec}{Spec}

	\DeclareMathOperator{\Ext}{Ext}
	\DeclareMathOperator{\Lie}{Lie}
	\DeclareMathOperator{\Sch}{Sch}
	\DeclareMathOperator{\Rep}{Rep}
	\DeclareMathOperator{\Ind}{Ind}
	\DeclareMathOperator{\Res}{Res}
	\DeclareMathOperator{\St}{St}
	\DeclareMathOperator{\GL}{GL}
	\DeclareMathOperator{\Sp}{Sp}

	\DeclareMathOperator{\ZipFlag}{-ZipFlag}
	\DeclareMathOperator{\Zip}{-Zip}
	\DeclareMathOperator{\Sh}{Sh}

	\DeclareMathOperator{\dR}{dR}
	\DeclareMathOperator{\std}{std}
	\DeclareMathOperator{\Sym}{Sym}
	\DeclareMathOperator{\Proj}{Proj}

	\DeclareMathOperator{\tor}{tor}

	\DeclareMathOperator{\Lc}{\mathcal{L}}
	\DeclareMathOperator{\Fc}{\mathcal{F}}
	\DeclareMathOperator{\E}{\mathcal{E}}
	\DeclareMathOperator{\Oc}{\mathcal{O}}
	\DeclareMathOperator{\Pb}{\mathbb{P}}
	\DeclareMathOperator{\W}{\mathcal{W}}
	\DeclareMathOperator{\Mod}{Mod}
	\DeclareMathOperator{\Loc}{Loc}
	\DeclareMathOperator{\pol}{pol}
	\DeclareMathOperator{\id}{id}
	\DeclareMathOperator{\ev}{ev}
	\DeclareMathOperator{\red}{red}
	\DeclareMathOperator{\Pol}{\normalfont \textsf{Pol}}
	\DeclareMathOperator{\End}{End}

	\DeclareMathSymbol{\shortminus}{\mathbin}{AMSa}{"39}
\newcommand{\DistTo}{\xrightarrow{
   \,\smash{\raisebox{-0.65ex}{\ensuremath{\scriptstyle\sim}}}\,}}

\title[Positivity, plethysm and hyperbolicity of Siegel varieties]{Positivity, plethysm and hyperbolicity of Siegel varieties in positive characteristic}

\author{Thibault Alexandre}

\begin{document}

\begin{abstract}
We study hyperbolicity properties of the moduli space of polarized abelian varieties (also known as the Siegel modular variety) in characteristic $p$. Our method uses the plethysm operation for Schur functors as a key ingredient and requires a new positivity notion for vector bundles in characteristic $p$ called $(\varphi,D)$-ampleness. Generalizing what was known for the Hodge line bundle, we also show that many automorphic vector bundles on the Siegel modular variety are $(\varphi,D)$-ample.
\end{abstract}

\maketitle

\vspace{-1cm}
\tableofcontents
\vspace{0.1cm}
\section{Introduction}
\subsection{General picture}
Let $p$ be a prime number. This article is concerned with the interplay in characteristic $p$ of the following three topics: positivity of vector bundles in algebraic geometry, the plethysm operation on Schur functors (and symmetric functions) and hyperbolicity properties of the moduli space of polarized abelian varieties (also known as the Siegel modular variety). It is well-known that positive vector bundles can be used to prove hyperbolicity results - a crucial new observation in our work is a link between the plethysm operation and hyperbolicity properties of Siegel varieties.
\vspace{0.2cm}

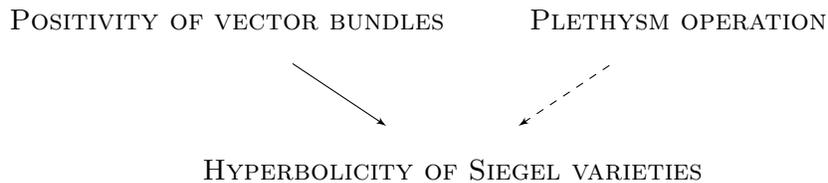
\tikzstyle{block} = [rectangle, minimum height=3em, minimum width=6em]
\tikzstyle{block_hl} = [rectangle, minimum height=3em, minimum width=6em]

\begin{center}
\begin{tikzpicture}[auto, node distance=2cm,>=latex']
   \node[block_hl] at (0, -2)   (a) {$\textsc{Hyperbolicity of Siegel varieties}$};
   \node[block] at (3, 0)   (b) {$\textsc{Plethysm operation}$};
   \node[block] at (-3,0)   (c) {$\textsc{Positivity of vector bundles}$};
    \draw[dashed,<-] (a)   -- (b) node[at start] {};
    \draw[<-] (a)   -- (c) node[at start] {};
\end{tikzpicture}
\end{center}
\vspace{0.3cm}
\subsection{History and motivation}
\subsubsection*{Hyperbolicity over a number field}
One of the most celebrated result in Diophantine geometry is the following.
\begin{theorem*}[{\cite{MR718935}}]
Consider a geometrically integral smooth projective curve $C$ over a number field $K/\Q$. The following three assertions are equivalent.
\begin{enumerate}
\item For all finite extension $F/K$, the set of $F$-rational points of $C$ is finite.
\item Every holomorphic map $\C \rightarrow C_{\C}^{\text{an}}$ is constant.
\item The canonical bundle $\omega_C$ is big, equivalently the genus $g$ of $C$ satisfies $g \geq 2$.
\end{enumerate}
\end{theorem*}
A curve satisfying these assertions is called hyperbolic and generalizing hyperbolicity to higher dimensional varieties is an open problem. One might be tempted to consider the following three definitions of hyperbolicity which are conjectured to be equivalent \cite[Conjecture 5.6/5.8]{MR828820}.
\begin{definition*}
Let $X$ denote a projective variety over a number field $K/\Q$.
\begin{enumerate}
\item $X$ is arithmetically hyperbolic if for all finite extension $F/K$, the set of $F$-rational points of $X$ is finite.
\item $X$ is Brody hyperbolic if every holomorphic map $\C \rightarrow X_{\C}^{\text{an}}$ is constant.
\item $X$ is algebraically hyperbolic\footnote{The terminology \emph{algebraically hyperbolic} is also used by Demailly and many others like Rousseau and Riedl for another notion of hyperbolicity.} if every integral subvariety $V$ of $X_{\C}$ is of general type, i.e. there exists a desingularization $\tilde{V} \rightarrow V$ such that $\omega_{\tilde{V}}$ is big.
\end{enumerate}
\end{definition*}
\vspace{0.2cm}
What can be said about the algebraic hyperbolicity of the moduli space $\mathcal{A}_{g,N} \rightarrow \Spec \Q$ of $g$-dimensional polarized abelian varieties with a full level $N$-structure? Since this moduli space is not proper, we should replace $(3)$ with the condition that all subvarieties are of \emph{log general type}. We recall that a variety $V$ is of log general type if there exists a proper desingularization $V \rightarrow \tilde{V}$ and a smooth projective variety $W$ together with an open embedding $\tilde{V} \subset W$ with $D := W \smallsetminus \tilde{V}$ a normal crossing divisor such that $\omega_{W}(D)$ is big. The moduli space $\mathcal{A}_{g,N}$ is known to be algebraically hyperbolic \cite{MR1803724} \cite{MR3859271}.
\subsubsection*{Algebraic hyperbolicity of $\mathcal{A}_{g,N}$ over a field of characteristic $p$}
Let's replace the base field $K$ by $k$, an algebraically closed field of characteristic $p$. Since desingularization techniques do not always exist in characteristic $p$, we will restrict ourselves to \emph{smooth} subvarieties. Assume that $p$ does not divide $N$ and consider a smooth projective toroidal compactification $\mathcal{A}^{\tor}_{g,N}$ of the Siegel modular variety over $k$ and write $D$ for its boundary as a normal crossing divisor. In this context, we will say that a smooth subvariety $\iota : V \hookrightarrow \mathcal{A}^{\tor}_{g,N}$ such that $\iota^{-1}D$ remains an effective Cartier divisor is said of log general type with respect to $D$ if the log canonical bundle $\omega_V (\iota^{-1}D)$ is big.
\begin{question*}[Characteristic $p$]
Is $\mathcal{A}^{\tor}_{g,N}$ algebraically hyperbolic over $k$? In other words, is every subvariety
\begin{equation*}
\iota : V \hookrightarrow \mathcal{A}^{\tor}_{g,N}
\end{equation*}
such that $\iota^{-1}D$ is well defined, of log general type with respect to $D$? 
\end{question*}
\vspace{0.2cm}
The answer to this question is in fact negative. In \cite{MR3618576}, Moret-Bailly constructs a non-isotrivial family $A \rightarrow \Pb^1$ of principally polarized supersingular abelian surfaces with a full level $N$-structure over the projective line over $\F_p$. This family yields a closed immersion $\Pb^1\rightarrow \mathcal{A}_{2,N}$ which contradicts the hyperbolicity of $\mathcal{A}^{\tor}_{2,N}$. The main objective of this article is to investigate the failure of hyperbolicity of Siegel varieties in positive characteristic.

\subsection{Our main result}
From now on, the letter $k$ will denote an algebraically closed field of characteristic $p$. To simplify our notation we will denote $\Sh$ the Siegel variety of genus $g$ over $k$ (instead of $\mathcal{A}_{g,N}$) and $D_{\red}$ the boundary of a smooth projective toroidal compactification $\Sh^{\tor}$. Motivated by the Green-Griffiths-Lang conjecture (\ref{GGL}), it is natural to expect that there is some exceptional locus $E \subset \Sh^{\tor}$ such that for any smooth subvariety $V$ not contained in the boundary, $V$ is of log general type if and only if $V \nsubseteq E$. Our main result about the hyperbolicity of the Siegel varieties is the following.
\begin{customthm}{1}[Corollary \ref{cor_hyperbolic}]
Assume that $p \geq g^2+3g+1$. Any subvariety $\iota : V \hookrightarrow \Sh^{\tor}$ of codimension $\leq g-1$ satisfying:
\begin{enumerate}
\item $V$ is smooth,
\item $\iota^{-1}D_{\red}$ is a normal crossing divisor,
\end{enumerate}
is of log general type with respect to $D$.
\end{customthm}
\vspace{0.2cm}
This indicates that the hypothetical exceptional locus $E \subset \Sh^{\tor}$ has a codimension strictly larger than $g-1$ and we believe it has exactly codimension $g$.
\begin{rmrk*}
\
\begin{enumerate}
\item The Theorem 1 is actually a corollary of the Theorem 4 which is stated at the end of the introduction.
\item For simplicity, we have restricted ourselves to smooth subvarieties but the Theorem 1 should also hold for non-smooth subvarieties if we use the definition of the logarithmic Kodaira dimension of a variety in positive characteristic which appears in \cite[p.46]{abramovich1994subvarieties} and \cite{luo1987kodaira, luo1988invariant} and we adapt our arguments using for example \cite[Lemma 5, p.46]{abramovich1994subvarieties} in the proof of our lemma \ref{lem_log_general}.
\item When $g = 1$ and $p = 2$ or $3$, there exist families of non-isotrivial elliptic curves over the multiplicative group $\mathbb{G}_m$. Specifically, consider the families $y^2 = x^3+x^2 - t$ in characteristic $3$ and $y^2 + xy = x^3 +t$ in characteristic $2$ where $t \in \mathbb{G}_m$. In both cases, the $j$-invariant is $j = 1/t$, so the curves are non-isotrivial. Since $\mathbb{G}_m$ is a smooth curve not of log general type, these examples show that the bound $p \geq g^2 + 3g + 1 = 5$ in the theorem 1 is sharp when $g = 1$. For $g > 1$, it is not known whether counterexamples exist when $p < g^2 + 3g + 1$.
\item The codimension assumption in the theorem 1 indicates that the Siegel modular variety $\Sh^{\tor}$ exhibits an intermediate form of pseudo-hyperbolicity, as suggested by the intermediate Lang conjectures (see \cite{MR828820}). These conjectures predict that varieties of general type should not contain ``large'' subvarieties that are not of general type. Our result aligns with this expectation by showing that any smooth subvariety of codimension $\leq g - 1$ intersecting the boundary normally is of log general type.
\end{enumerate}
\end{rmrk*}
\subsection{A new positivity notion for vector bundles in characteristic $p$}\label{sect_pos_intro}
In order to prove Theorem 1, we introduce and study a positivity notion for vector bundles which is weaker than ampleness but stronger that nefness and bigness. Assume that $X$ is a projective scheme over $k$ and $D$ is an effective Cartier divisor on $X$. Since this positivity notion involves  the relative Frobenius map $\varphi : X \rightarrow X^{(p)}$, we have decided to call it $(\varphi,D)$-ampleness. 
\begin{definition*}
A vector bundle $\E$ over $X$ is said to be $(\varphi,D)$-ample if there is an integer $r_0 \geq 1$ such that for all integers $r \geq r_0$, the vector bundle $\E^{(p^r)}(-D) := {(\varphi^{r})}^{*}{(\varphi^{r})}_{*}\E \otimes \Oc_X(-D)$ is ample.
\end{definition*}
\vspace{0.2cm}
Our main motivation comes from the fact that the Hodge line bundle $\omega := \det \Omega$ is not always ample on a toroidal compactification $\Sh^{\tor}$ but it is nef and big with exceptional locus\footnote{Following \cite{keel1999basepoint}, the exceptional locus of a nef line bundle $\Lc$ is the closure, with reduced structure, of the union of all subvarieties $V$ such that $\Lc_{|V}$ is not big.} contained in the boundary $D_{\red}$. In fact, we even know that $\omega$ is $(\varphi,D)$-ample for some effective Cartier divisor $D$ whose associated reduced divisor is the boundary $D_{\red}$. Compared to nefness and bigness for vector bundles, we show that $(\varphi,D)$-ampleness behaves well as it is stable under direct sum, extension\footnote{Under a regular hypothesis on $X$}, quotient, tensor product, tensor roots, pullback by finite morphim and it satisfies descent along finite surjective morphism. Inspired by a result of Mourougane {\cite[Théorème 1]{MR1614576}} over $\C$ about the ampleness of the adjoint bundle $\pi_*(\Lc \otimes \ \omega_{Y/X})$ where $\pi : Y \rightarrow X$ is a surjective morphism and $\Lc$ is an ample line bundle on $Y$, we prove similar results in characteristic $p$ when $\pi$ is a flag bundle. \\

More precisely, let $G$ denote a connected split reductive algebraic group over $k$. Fix a Borel pair $(B,T)$ of $G$ and write $\rho$ for the half-sum of positive roots of $G$. Let $E$ be a $G$-torsor over $X$ and $\pi : Y \rightarrow X$ the flag bundle that parametrizes $B$-reduction of $E$. Recall that we can associate a line bundle $\Lc_{\lambda}$ on $Y$ to each character $\lambda$ of $T$. We prove the following.
\begin{customthm}{2}[Theorem \ref{th1} and \ref{th2}]
If $\mathcal{L}_{2\lambda+2\rho}$ is ample (resp. $(\varphi,\pi^{-1}D)$-ample) on $Y$, then $\pi_*\Lc_{\lambda}$ is an ample (resp. $(\varphi,D)$-ample) vector bundle on $X$.
\end{customthm}
\vspace{0.2cm}
Note that since $\omega_{Y/X} = \Lc_{-2\rho}$, our result can be seen as a characteristic $p$ version of the result of Mourougane. 
\subsection{Positivity of automorphic vector bundles on the Siegel variety}
We explain a direct application of Theorem 2 to automorphic vector bundles defined over the Siegel variety. Recall that the Hodge bundle is a rank $g$ vector bundle over the Siegel variety which is defined as $\Omega = e^*\Omega^1_{A/\Sh}$ where $e$ is the neutral section of the universal abelian scheme $f : A \rightarrow \Sh$. Denote $\pi : Y \rightarrow \Sh$ the flag bundle which parametrizes complete filtration of $\Omega$. Recall that for every character $\lambda$ of the standard maximal torus of $\GL_g$, we have an associated line bundle $\Lc_{\lambda}$ on $Y$ and a costandard automorphic vector bundle $\nabla(\lambda)$ over the Siegel variety which is isomorphic to $\pi_{*}\Lc_{\lambda}$. All these objets can be extended to a toroidal compactification $\Sh^{\tor}$. Following the idea of \cite{StrohPrep}, we know by {\cite[Theorem 5.11]{alexandre2022vanishing}} that certain line bundles $\Lc_{\lambda}$ are $(\varphi,D)$-ample on $Y$ where $D$ is some fixed\footnote{It depends on the choice of a polarization function on the polyhedral cone decomposition of a toroidal compactification.} effective Cartier divisor whose associated reduced divisor is the boundary $D_{\red}$. \\

We denote $G$\footnote{It is not the same $G$ as in section \ref{sect_pos_intro}.} the symplectic group $\Sp_{2g}$ over $k$, $W$ the Weyl group of $G$, $P \subset G$ the parabolic that stabilizes the Hodge filtration on the first de Rham cohomology of $A \rightarrow \Sh$, $\Delta \subset \Phi^+ \subset \Phi$ the set of (simple, positive) roots of $G$, $I \subset \Delta$ the type of $P$, $L$ the Levi subgroup of $P$, $\Phi_L^{+} \subset \Phi_L$ the set of (positive) roots of $L$ and $\rho_L = 1/2\sum_{\alpha \in \Phi_L^{+}} \alpha$. The following result is a direct application of Theorem 2.
\begin{customthm}{3}[Theorem \ref{th_automorphic_bundle}]
Let $\lambda$ be a dominant character of $T$. If $\gamma := 2\lambda+2\rho_L$ is
\begin{enumerate}
\item orbitally $p$-close, i.e.
\begin{equation*}
\max_{\alpha \in \Phi, w \in W, \langle \gamma,\alpha^{\vee} \rangle \neq 0} | \frac{\langle \gamma,w\alpha^{\vee}\rangle}{\langle \gamma,\alpha^{\vee} \rangle }| \leq p-1
\end{equation*}
\item  and $\mathcal{Z}_{\emptyset}$-ample, i.e.
\begin{equation*}
\langle \gamma, \alpha^{\vee}\rangle >0 \text{ for all } \alpha \in I \text{ and } \langle \gamma, \alpha^{\vee}\rangle <0 \text{ for all } \alpha \in \Phi^+ \backslash \Phi^+_L
\end{equation*}
\end{enumerate}
then the automorphic vector bundle $\nabla(\lambda)$ is $(\varphi,D)$-ample on $\Sh^{\tor}$. 
\end{customthm}
\begin{rmrk*}
Positivity results for automorphic vector bundles were only known for line bundles which corresponds to the case where $\lambda$ is positive parallel, i.e. $\nabla(\lambda)$ is a positive power of the Hodge line bundle $\omega = \det \Omega = \nabla(\shortminus 1, \ldots, \shortminus 1)$.
\end{rmrk*}
\subsection{Schur functor and the plethysm operation}
Schur functors are certain endofunctors 
\begin{equation*}
S : \FinVect_{k} \rightarrow \FinVect_{k}
\end{equation*}
of the abelian category of finite dimensional $k$-vector spaces that generalize the constructions of exterior powers and symmetric powers of a vector space. Schur functors are indexed by integer partition or Young diagrams and they can be defined on the category of finite locally free modules over a scheme. We are interested in these functors because if $\lambda = (k_1 \geq \ldots \geq k_g \geq 0)$ is a $G$-dominant character, we can identify it with a Young diagram where the $i^{\text{th}}$-row has $k_i$ columns and we get an isomorphism 
\begin{equation*}
S_{\lambda}\Omega =  \nabla(-w_0\lambda)
\end{equation*}
where $w_0 \in W$ is the longest element of the Coxeter group $W$. The strategy to prove Theorem 1 is to show that the bundle $S_{\lambda}\Omega^1_{\Sh^{\tor}}(\log D_{\red})$ is $(\varphi,D)$-ample for specific choices of $\lambda$. Since $(\varphi,D)$-ampleness is stable by quotients, pullback by finite morphisms and $S_{\lambda}$ respects surjections, the $(\varphi,D)$-ampleness of $S_{\lambda}\Omega^1_{\Sh^{\tor}}(\log D_{\red})$ implies that the quotient 
\begin{equation*}
\iota^*S_{\lambda}\Omega^1_{\Sh^{\tor}}(\log D_{\red}) \twoheadrightarrow S_{\lambda}\Omega^1_V(\log \iota^{-1}D_{\red})
\end{equation*}
is also $(\varphi,\iota^{-1}D)$-ample for any smooth subvariety $\iota : V \hookrightarrow \Sh^{\tor}$ such that $\iota^{-1}D_{\red}$ is well defined as a normal crossing divisor. It follows from the general theory of Schur functors that the bundle $S_{\lambda}\Omega^1_V(\log \iota^{-1}D_{\red})$ is non-zero exactly when the dimension of $V$ is larger than the number of parts (also called the height $\text{ht}(\lambda)$) of $\lambda$. In this case, the determinant of $S_{\lambda}\Omega^1_V(\log \iota^{-1}D_{\red})$ is a tensor power of $\omega_{V}(\iota^{-1}D_{\red})$ and the $(\varphi,D)$-ampleness of $S_{\lambda}\Omega^1_{\Sh^{\tor}}(\log D_{\red})$ implies that $V$ is of log general type with respect to $D_{\red}$. By the Kodaira-Spencer isomorphism
\begin{equation*}
\rho_{\text{KS}} : \Sym^2 \Omega \DistTo \Omega^1_{\Sh^{\tor}}(\log D_{\red}),
\end{equation*}
we are reduced to study the composition of Schur functors $S_{\lambda} \circ \Sym^2$. \\

The correct category to study Schur functors in characteristic $p$ such as $S_{\lambda}$ is the category of strictly polynomial functors $\textsf{Pol}$ introduced by Friedlander and Suslin in \cite{MR1427618}. A strictly polynomial functor $T : \FinVect_{k} \rightarrow \FinVect_{k}$ over a field $k$ is polynomial in the sense that for any finite dimensional $k$-vector spaces $V,W$, the map
\begin{equation*}
T_{V,W} : \Hom_k(V,W) \rightarrow \Hom_k(T(V),T(W))
\end{equation*}
is a scheme morphism where we have enriched $\Hom_k(V,W)$ and $\Hom_k(T(V),T(W))$ with their natural scheme structure. Equipped with the classical tensor product $\otimes$, the category of strictly polynomial functors is a symmetric monoidal category whose Grothendieck group $K_0(\textsf{Pol})$ is the ring $\mathcal{R}$ of symmetric functions. A key feature of $\textsf{Pol}$ is that the functor composition $\circ$ defines a second (non-symmetric) monoidal structure on it. Recall that $\mathcal{R}$ possesses a natural basis $\left\{s_{\lambda}\right\}_{\lambda}$ indexed by the set of integer partition where each $s_{\lambda}$ is the class of the Schur functor $S_{\lambda}$. Over $\C$, it is well-known that $\textsf{Pol}$ is semi-simple ; in particular the composition of two Schur functors of partition $\lambda$ and $\mu$ can be splitted as a direct sum of Schur functors 
\begin{equation*}
S_{\lambda} \circ S_{\mu} = \bigoplus_{\eta} S_{\eta}^{\oplus c^{\eta}_{\lambda,\mu}}
\end{equation*}
where the coefficient $c^{\eta}_{\lambda,\mu}$ are given by the decomposition of $s_{\lambda}\circ s_{\mu}$ in the basis $\left\{s_{\lambda}\right\}_{\lambda}$ of $\mathcal{R}$. The problem of determining the coefficients $c^{\eta}_{\lambda,\mu}$ is known as plethysm. Over a field of characteristic $p$, semi-simplicity of $\textsf{Pol}$ fails but we may ask if the composition $S_\lambda \circ S_{\mu}$ admits at least a filtration where the graded pieces are isomorphic to Schur functors $S_{\eta}$. Unfortunately, Boffi \cite{MR1128609} and Touzé \cite[Corollary 6.10.]{MR3042627} have found counter-examples to the existence of such filtrations. For example, the plethysm $\Lambda^2 \circ \Lambda^2$ over $\F_2$ does not admit any such filtration. We avoid these counter-examples with a technical restriction on the prime $p$.
\begin{proposition*}[Proposition \ref{th_plethysm}]
Let $\lambda$ and $\mu$ be partitions of size $|\lambda|$ and $|\mu|$. If $p \geq 2|\lambda |-1$, the strict polynomial functor $S_{\lambda} \circ S_{\mu}$ admits a finite filtration
\begin{equation*}
0 =T^n \subsetneq T^{n-1} \subsetneq
 \cdots \subsetneq
 T^0 = S_{\lambda} \circ S_{\mu}
\end{equation*}
by strict polynomial functors of degree $|\lambda| |\mu |$ where the graded pieces are Schur functors.
\end{proposition*}
\subsection{Plethysm and hyperbolicity}
Since $(\varphi,D)$-ampleness is stable under extension, we can use proposition \ref{th_plethysm} to see that $S_{\lambda}\Omega^1_{\Sh^{\tor}}(\log D_{\red})$ is $(\varphi,D)$-ample if the graded pieces $\nabla(\eta)$ that appears in the plethysm $S_{\lambda} \circ \Sym^2$ are $(\varphi,D)$-ample. It is worth pointing out that plethysm computations are really hard and there is no known general combinatorial rule to express the coefficient $c^{\eta}_{\lambda,\mu}$\footnote{At the beginning, we used a computer to find integer partitions $\lambda$ such that each automorphic bundle appearing in the plethysm $S_{\lambda} \circ \Sym^2$ is $(\varphi,D)$-ample (see appendix \ref{appendix2}). Starting with $g=2$, we have found that surfaces in the Siegel theefold are of log-general type. For each $g \in \{2,3,4\}$, the computer was able to find a partition $\lambda$ of height equal to $\dim \Sh - (g-1)$ such that $S_{\lambda} \circ \Sym^2 \Omega$ is $(\varphi,D)$-ample. For $g = 5$, the plethysm computation was too long to conclude and it became clear that we needed a different method.}. Moreover, determining effectively if an automorphic bundle $\nabla(\eta)$ is $(\varphi,D)$-ample with the Theorem 3 is also challenging as it involves the orbitally $p$-closeness condition. It is known since {\cite[Lemma 7]{MR2497590}} that the plethysm $\Lambda^k \circ \Sym^2$ belongs to one of the few cases where a general formula is known. With this formula and an upper bound of the orbitally $p$-closeness condition, we were able to show the following result. 
\begin{customthm}{4}[Theorem \ref{th_hyperbolic}]
Assume that $p \geq g^2+3g+1$. For all $k \geq g(g-1)/2+1$, the bundle $\Lambda^k\Sym^2\Omega = \Omega^k_{\Sh^{\tor}}(\log D_{\red})$ is $(\varphi,D)$-ample.
\end{customthm}
\vspace{0.2cm} 
In the case $g\in \{2,3\}$, we also prove that this bound on $k$ is optimal, which is some evidence that the hypothetical exceptional locus $E \subset \Sh^{\tor}$ has codimension $g$. 
\subsection{Organization of the paper}
In section \ref{sect_rep}, we recall some general results on algebraic representations of reductive groups in characteristic $p$. In section \ref{sect_schur}, we study the plethysm operation for Schur functors in characteristic $p$.  In section \ref{sect_pos}, we introduce and study the main properties of $(\varphi,D)$-ample vector bundles. In particular, we prove many stability properties that are summarized in the table \ref{table_pos}. In section \ref{sect_flag}, we recall the flag bundle construction associated to a general $G$-torsor. In section \ref{sect_push}, we prove that the adjoint bundle of an ample (resp. $(\varphi,D)$-ample) line bundle along a complete flag bundle, is an ample (resp. $(\varphi,D)$-ample) vector bundle. In section \ref{sect_auto}, we apply our result on the positivity of adjoint bundles to the case of automorphic vector bundles over the Siegel variety. In particular, the figure \ref{figure_amp} illustrate our ampleness result in the case $g=2$. In section \ref{sect_hyp}, we finish the proof of our main theorem about the partial hyperbolicity of the Siegel modular variety in characteristic $p$.

\subsection*{Acknowledgements}
I would like to thank Benoit Stroh for his invaluable support over the years and for his careful reading of an early version of this paper. I am very grateful to Yohan Brunebarbe for explaining me his results on the hyperbolicity of the Siegel variety over the complex numbers and to Arnaud Eteve for many helpful conversations about Schur functors. I also thank more broadly Sebastian Bartling, Diego Berger, Thibaut Ménès for their careful feedback. I would also like to acknowledge the anonymous referee for their careful reading of this paper and for noticing that Theorem 1 should hold for non-smooth subvarieties as well.
\section{Representations of algebraic groups}\label{sect_rep}
Recall that $k$ is an algebraically closed field of characteristic $p$. In this section, we recall some well-known results about algebraic representations of reductive groups over $k$ that can be found in \cite{MR2015057}. Let $G$ be a connected split reductive algebraic group over $k$. We choose a Borel pair $(B,T)$ of $G$, i.e. a Borel sugroup $B \subset G$ together with a maximal torus $T \subset G$ defined over $k$. Denote $(X^*,\Phi,X_*,\Phi^{\vee})$ the root datum of $G$ where $X^*$ is the group of characters of $T$, $X_*$ is the group of cocharacters of $T$, $\Phi$ is the set of roots of $G$, $\Phi^{\vee}$ is the set of coroots of $G$ and
\begin{equation*}
\langle \cdot ,\cdot \rangle : X^* \times X_* \rightarrow \mathbb{Z}
\end{equation*}
is the perfect pairing between the characters and the cocharacters of $T$. To any root $\alpha \in \Phi$, there is an associated coroot $\alpha^{\vee}$ such that $\langle \alpha, \alpha^{\vee} \rangle = 2$. This choice of $(B,T)$ determines a set of positive roots $\Phi^{+}$ and a set of simple roots $\Delta \subset \Phi^{+}$. To simplify the statement of the proposition \ref{prop13_sieg}, we follow a non-standard convention for the positive roots by declaring $\alpha \in \Phi$ to be positive if the root group $U_{-\alpha}$ is contained in $B$. A character $\lambda \in X^*$ is said to be $G$-dominant (or simply dominant if there is no ambiguity on the group $G$) if $\langle \lambda, \alpha^{\vee} \rangle \geq 0$ for all $\alpha \in \Phi^{+}$. We denote $\rho$ the half-sum of the positive roots. We denote $W$ the Weyl group of $G$, $l : W \rightarrow \N$ its length function and $w_0$ its longest element. Consider a collection $I \subset \Delta$ of simple roots. We denote $\Phi_I$ (resp. $\Phi_I^+$) the set of roots (resp. positive roots) obtained as $\Z$-linear combination of roots in $I$. We denote $W_I \subset W$ the subgroup generated by the reflections $s_{\alpha}$ where $\alpha \in I$ and ${}^IW \subset W$ the set of minimal length representatives of $W_I\backslash W$. We denote $\varphi : G \rightarrow G^{(p)}$ the relative Frobenius morphism of $G$ where $G^{(p)} = G \times_{k,\sigma} k$ is the pullback along the Frobenius map $\sigma : k \rightarrow k$ of $k$. Since any split reductive group is a base change of a split reductive group over $\mathbb{Z}$, the reductive group $G$ is isomorphic to $G^{(p)}$. For any $G$-module $M$, we define $M^{(p^r)}$ as the same module $M$ with a $G$-action twisted by $\varphi^r$. \\

Denote $\Rep_k(G)$ the category of algebraic representations of $G$ on finite dimensional $k$-vector spaces. We will use interchangeably the term $G$-module to denote any representation $V \in \Rep_k(G)$. It is well-known that this category is not semi-simple but we can still define some interesting highest weight representations. 
\begin{proposition}[{\cite[Part II, sect. 2.4]{MR2015057}}]
For any dominant $T$-character $\lambda$, there is a unique simple $G$-module of highest weight $\lambda$ that we denote $L(\lambda)$. 
\end{proposition}
\begin{definition}[{\cite[Part I, sect. 5.8]{MR2015057}}]
For any character $\lambda$ of $T$, we denote $\Lc_{\lambda}$ the line bundle on the flag variety $G/B$ defined as the $B$-quotient of the vector bundle $G \times_k \mathbb{A}^1 \rightarrow G$ where $B$ acts on $G \times_k \mathbb{A}^1$ by
\begin{equation*}
(g,x)b = (gb^{-1},\lambda(b^{-1})x)
\end{equation*}
and where $\lambda$ is extended by zero on the unipotent radical of $B$.
\end{definition}
\vspace*{0.2cm}
Recall Kempf's vanishing theorem.
\begin{proposition}[{\cite[Part II, sect. 4.5]{MR2015057}}]\label{prop4}
Let $\lambda$ be a dominant character. We have
\begin{equation*}
H^i(G/B,\Lc_{\lambda}) = 0
\end{equation*}
for every integer $i >0$.
\end{proposition}
\vspace*{0.2cm}
\begin{definition}[{\cite[Part II, sect. 2]{MR2015057}}]
Let $\lambda$ be a character of $T$. The costandard $G$-module $\nabla(\lambda)$ of highest weight $\lambda$ is defined as the global section group $H^0(G/B,\Lc_\lambda)$ where $G$ acts through left translation. The standard $G$-module $\Delta(\lambda)$ of highest weight $\lambda$ is defined as $\nabla(-w_0\lambda)^{\vee}$ where $\vee$ denotes the linear dual in $\Rep_k(G)$.
\end{definition}
\begin{proposition}[{\cite[Part II, sect. 2.6]{MR2015057}}]
The $G$-modules $\nabla(\lambda)$ and $\Delta(\lambda)$ are non-zero exactly when $\lambda$ is dominant. Moreover, their highest $T$-weight is $\lambda$.
\end{proposition}
\vspace*{0.2cm}
It follows directly from their definition that $\nabla(\lambda)$ and $\Delta(\lambda)$ have the same weights but they are usually not simple and not isomorphic. We give a condition on the size of the highest weight of a standard/costandard module to be simple.
\begin{proposition}[{\cite[Part II, sect. 5.6]{MR2015057}}]\label{prop_psmall}
If $\lambda$ is a $p$-small character, i.e. 
\begin{equation*}
\forall \alpha \in \Phi^{+} \ \langle \lambda + \rho, \alpha^{\vee} \rangle \leq p,
\end{equation*}
then we have isomorphisms
\begin{equation*}
\nabla(\lambda) = \Delta(\lambda) = L(\lambda).
\end{equation*}
\end{proposition}
\begin{rmrk}
If $\lambda$ is $p$-small and $\mu \leq \lambda$, then $\mu$ is also $p$-small.
\end{rmrk}
\vspace*{0.2cm}
In positive characteristic, there is a very special algebraic representation called the Steinberg representation $\St_r$. The Steinberg representation is a self-dual simple $G$-module whose highest weight is never $p$-small.
\begin{definition}[{\cite[Part II, sect. 3.18]{MR2015057}}]
Assume $p \neq 2$ or $\rho \in X^*(T)$. For each $r\geq 1$, we define the Steinberg module as
\begin{equation*}
\St_r := \nabla((p^r-1)\rho).
\end{equation*}
\end{definition}
\begin{proposition}[{\cite[Part II, sect. 3.19]{MR2015057}}]
We have isomorphisms
\begin{equation*}
\nabla((p^r-1)\rho) = \Delta((p^r-1)\rho) = L((p^r-1)\rho).
\end{equation*}
In particular, $\St_r$ is a simple $G$-module. 
\end{proposition}
\vspace*{0.2cm}
We come to the main proposition that justifies our interest in the Steinberg representation.
\begin{proposition}[{\cite[Part. 2, Chap. 3, Sect. 19]{MR2015057}}]\label{prop1}
Let $\lambda$ be a character and $r \geq 1$ an integer. For all $i \geq 0$, we have an isomorphism of $G$-modules
\begin{equation*}
H^i(G/B,\mathcal{L}_{(p^r-1)\rho} \otimes \Lc_{p^r\lambda}) = \St_r \otimes H^i(G/B,\Lc_{\lambda})^{(p^r)}.
\end{equation*}
\end{proposition}
We recall the notion of $\nabla$-filtration and $\Delta$-filtration.
\begin{definition}\label{def6}
Let $V$ be a $G$-module. A $\nabla$-filtration is a filtration of $V$ where each graded piece is a costandard module. A $\Delta$-filtration is a filtration of $V$ where each graded piece is a standard module.
\end{definition}
\begin{rmrk}
The category $\Rep_k(G)$ has the structure of a highest weight category\footnote{An introduction to this framework is given in \cite{riche:tel-01431526}.}. Within this framework, tilting modules are defined as modules that admits both a $\nabla$- and a $\Delta$-filtration.
\end{rmrk}
\vspace{0.2cm}
The following proposition, due to Mathieu, states the existence of a $\nabla$-filtration for the tensor product $\nabla(\lambda) \otimes \nabla(\mu)$ of costandard modules.
\begin{proposition}[{\cite[Theorem 1]{MR1072820}}]\label{prop8_van}
Consider two dominant characters $\lambda, \mu$ in $X^*(T)$. Then the tensor product $\nabla(\lambda) \otimes \nabla(\mu)$ admits a $\nabla$-filtration $(V^i)_{i \geq 0}$ with graded pieces  
\begin{equation*}
V^i/V^{i+1} \simeq \nabla(\lambda+\mu_i)
\end{equation*}
where $(\mu_i)_{i}$ is a certain subcollection of weights of $\nabla(\mu)$.
\end{proposition}
\begin{rmrk}
\begin{enumerate}
\item Not all the weights $\mu^\prime \leq \mu$ of $\nabla(\mu)$ such that $\lambda + \mu^\prime$ is dominant will appear in the $\nabla$-filtration.
\item By duality, we deduce that tensor products of standard modules $\Delta(\lambda) \otimes \Delta(\mu)$ admit a $\Delta$-filtration.
\end{enumerate}
\end{rmrk}
\begin{corollary}\label{cor2}
Consider two $G$-modules $V$ and $W$. If $V$ and $W$ admit a $\nabla$-filtration, then $V \otimes W$ admits also a $\nabla$-filtration.
\end{corollary}
\vspace*{0.2cm}
The following cohomological criterion is very useful to detect when a $G$-module possesses a $\nabla$-filtration.
\begin{proposition}[Donkyn criterion]\label{propDonkyn}
Consider a $G$-module $V$. The following assertions are equivalent.
\begin{enumerate}
\item $V$ admits a $\nabla$-filtration.
\item For any dominant character $\lambda$ and $i>0$, $\Ext_G^i(\Delta(\lambda),V) = 0$.
\item For any dominant character $\lambda$, $\Ext_G^1(\Delta(\lambda),V) = 0$.
\end{enumerate} 
\end{proposition}
\begin{proof}
See \cite[Part II, sect. 4.16]{MR2015057}.
\end{proof}
\begin{corollary}\label{cor1}
Consider two $G$-modules $V$ and $W$. If $V$ admits a $\nabla$-filtration and $W$ is a direct factor of $V$, then $W$ admits a $\nabla$-filtration.
\end{corollary}
\section{Plethysm for Schur functors in positive characteristic}\label{sect_schur}
\subsection{Over the complex numbers}
Classically, Schur functors are certain endofunctors 
\begin{equation*}
S : \FinVect_{\C} \rightarrow \FinVect_{\C}
\end{equation*}
of the abelian category of finite dimensional complex vector spaces. The first example is given by the $n^{\text{th}}$-symmetric power $\Sym^n$ which sends a vector space $V$ to the space of $\mathfrak{S}_n$-coinvariants $(V^{\otimes n})_{\mathfrak{S}_n}$ where $\mathfrak{S}_n$ acts on $V^{\otimes n}$ by permuting the factors. A second example is given by the $n^{\text{th}}$-exterior power $\Lambda^n$ which sends a vector space $V$ to the space of $\mathfrak{S}_n$-coinvariants $(V^{\otimes n})_{\mathfrak{S}_n}$ where an element $\sigma \in \mathfrak{S}_n$ acts on $V^{\otimes n}$ by antisymmetrization
\begin{equation*}
\sigma(v_1 \otimes \cdots \otimes v_n) = \varepsilon(\sigma)(v_{\sigma(1)} \otimes \cdots \otimes v_{\sigma(n)})\footnote{This definition is not correct over a field of characteric $p$ if $p \geq n$. We should instead consider a quotient by the ideal generated by tensors of the form $x_1 \otimes \cdots \otimes x_n$ where $x_i = x_j$ for some $i\neq j$.}.
\end{equation*}
In general, we consider a finite dimensional representation $\pi$ of the symmetric group $\mathfrak{S}_n$ for some integer $n \geq 1$ and we define the Schur functor 
\begin{equation*}
S_\pi : \FinVect_{\C} \rightarrow \FinVect_{\C}
\end{equation*}
associated to $\pi$ as 
\begin{equation*}
S_\pi (V) = (V^{\otimes n} \otimes \pi)_{\mathfrak{S}_n}
\end{equation*}
where $\mathfrak{S}_n$ acts via permutation on the first factor. It is well-known that irreducible representations $\pi$ of ${\mathfrak{S}_n}$ are in bijection with partitions $\lambda = (\lambda_1 \geq \lambda_2 \geq \cdots \geq \lambda_r \geq 0)$ of $n$. This bijection is made explicit by sending a partition $\lambda$ of $n$ to the Specht module $\Sp_{\lambda}$ of shape $\lambda$. We call $S_{\lambda} = S_{\Sp_{\lambda}}$, the Schur functor of weight $\lambda$. It is well-know that for any two partitions $\lambda$ and $\mu$, we have a direct sum decomposition
\begin{equation}\label{eq_pleth}
S_\lambda \circ S_\mu \simeq \bigoplus_{\eta} S_{\eta}^{\oplus c^{\eta}_{\lambda,\mu}}
\end{equation}
in the category of endofunctors of $\FinVect_{\C}$. The problem of determining the coefficients $c^{\eta}_{\lambda,\mu}$ is called plethysm\footnote{The term "plethysm" was suggested to Littlewood by M. L. Clark after the Greek word plethysmos, or $\pi\lambda\eta\theta\upsilon\sigma\mu o\varsigma$, which means "multiplication" in modern Greek (though apparently the meaning goes back to ancient Greek). The related term plethys in Greek means "a big number" or "a throng", and this in turn comes from the Greek verb plethein, which means "to be full", "to increase", "to fill", etc.}.
\begin{example}\label{ex_plethysm}
There is no known combinatorial rule for computing the coefficients $c^{\eta}_{\lambda,\mu}$. To illustrate how hard is the plethysm problem, we give the following examples.
\begin{enumerate}
\item $S_{(2,1)} \circ S_{(1,1)} = S_{(2, 1, 1, 1, 1)}\oplus S_{(2, 2, 1, 1)} \oplus S_{(3, 2, 1)}$.
\item The composition $S_{(4,2)} \circ S_{(3,1)}$ involves $1,238$ different partitions $\eta$ with a maximum multiplicity $c^{\eta}_{\lambda,\mu}$ of $8,408$. Counted with multiplicity, there are $958,705$ endofunctors in the direct sum.
\item The composition $S_{(3,2,1)} \circ S_{(4,2)}$ involves $11,938$ different partitions $\eta$ with a maximum multiplicity $c^{\eta}_{\lambda,\mu}$ of $9,496,674$. Counted with multiplicity, there are\\ $4,966,079,903$ endofunctors in the direct sum.
\end{enumerate}
\end{example}
\subsection{Strict polynomial functors}
In their founder article \cite{MR1427618}, Friedlander and Suslin introduced the category of strict polynomial functors $\textsf{Pol}$ over $k$. This functor category is well-behaved compared to the category of endofunctors of $\FinVect_k$. In particular, when $n\geq d$, they prove an equivalence of categories
\begin{equation*}
\Pol_d\simeq S(n,d)-\Mod
\end{equation*}
between the category $\textsf{Pol}_d$ of strict polynomial functors homogeneous of degree $d$ and the category of modules over the Schur algebra $S(n,d)$\footnote{Let $A_n = k[\text{Mat}_n]$ denote the Hopf algebra freely generated by the $k$-vector space $\text{Mat}_n$ of $n\times n$ matrices where the non-commutative comultiplication is induced by the matrix multiplication on $\text{Mat}_n$. Let  $A(n,d) \subset A_n$ denote the subalgebra of homogeneous polynomials of degree $d$. The Schur algebra $S(n,d)$ is then defined as the linear dual of $A(n,d)$.}. If $V$, $W$ are finite dimensional vector spaces over $k$, we denote by $\Hom_{\pol}(V,W)$ the abelian group of scheme morphisms over $k$ between $\underline{V}$ and $\underline{W}$. To clarify, we have $\Hom_{\pol}(V,W) = \Sym^*(V^{\vee})\otimes_k W$ and elements of $\Hom_{\pol}(V,W)$ are called polynomial maps between $V$ and $W$.
\begin{definition}[{\cite[Definition 2.1]{MR1427618}}]
A strict polynomial functor $$T : \FinVect_k \rightarrow \FinVect_k$$ is a pair of functions, the first of which assigns to each $V \in \FinVect_k$ a vector space $T(V) \in \FinVect_k$ and the second assigns a polynomial map $$T_{V,W} \in \Hom_{\pol}(\Hom_k(V,W),\Hom_k(T(V),T(W)))$$ to each $V,W$. These two functions should satisfy the usual conditions of the definition of a functor
\begin{enumerate}
\item For any vector space $V \in \FinVect_k$, we have $T_{V,V}(\id_V) = \id_{T(V)}$
\item For any $U,V,W$, the following diagram of polynomial maps commute
\begin{equation*}
\begin{tikzcd}
\Hom_k(V,W) \times \Hom_k(U,V) \arrow[d,"T_{V,W}\times T_{U,V}"] \arrow[r] & \Hom_k(U,W) \arrow[d,"T_{U,W}"] \\
\Hom_k(T(V),T(W)) \times \Hom_k(T(U),T(V)) \arrow[r] & \Hom_k(T(U),T(W))
\end{tikzcd}
\end{equation*}
\end{enumerate}
Let $T : \FinVect_k \rightarrow \FinVect_k$ be a strict polynomial functor. We say that $T$ is homogeneous of degree $d$ if for all vector spaces $V,W$, the polynomial map
$$T_{V,W} \in \Hom_{\pol}(\Hom_k(V,W),\Hom_k(T(V),T(W)))$$
has degree $d$. We denote $\Pol$ the category of strict polynomial functors of finite degree where the morphism are morphism between the underlying functors.
\end{definition}
\begin{proposition}[{\cite[Proposition 2.6]{MR1427618}}]
The category of strict polynomial functors (of finite degree) decomposes
\begin{equation*}
\Pol = \bigoplus_{d \geq 0} \Pol_d
\end{equation*}
where $\Pol_d$ is the full subcategory of $\Pol$ consisting of strict polynomial functors homogeneous of degree $d$. In particular, there are no extension between two strict polynomial functors homogeneous of different degrees.
\end{proposition}
\begin{example}
We give some simple example of strict polynomial functors.
\begin{enumerate}
\item The $n^{\text{th}}$-tensor power $(\cdot)^{\otimes n}$ which sends a $k$-vector space $V$ to $V^{\otimes n}$ is homogeneous of degree $n$.
\item The $n^{\text{th}}$-symmetric power $\Sym^n$ which sends a $k$-vector space $V$ to the space of $\mathfrak{S}_n$-coinvariants $(V^{\otimes n})_{\mathfrak{S}_n}$ where $\mathfrak{S}_n$ acts on $V^{\otimes n}$ by permuting the factors is homogeneous of degree $n$.
\item The $n^{\text{th}}$-divided power $\Gamma^n$ which sends a $k$-vector space $V$ to the space of $\mathfrak{S}_n$-invariants $(V^{\otimes n})^{\mathfrak{S}_n}$ where $\mathfrak{S}_n$ acts on $V^{\otimes n}$ by permuting the factors is homogeneous of degree $n$.
\item The $n^{\text{th}}$-exterior power $\Lambda^n$ which sends a vector space $V$ to the quotient space $V^{\otimes n}/I$ where $I$ is the ideal generated by elements $x_1\otimes \cdots \otimes x_n$ such that $x_i = x_j$ for some $i \neq j$ is homogeneous of degree $n$.
\item If $\text{char}(k)= p$, the Frobenius twist $(\cdot)^{(p)}$ which sends a $k$-vector space $V$ to its pullback $V \otimes_{k,\sigma} k$ by the Frobenius map $\sigma : k \rightarrow k$ is homogeneous of degree $p$.
\end{enumerate}
\end{example}
\begin{rmrk}
The functors $\Sym^n$ and $\Gamma^n$ are isomorphic over a field of characteristic $0$ but not over a field of characteristic $p>0$ when $n\geq p$.
\end{rmrk}
\vspace*{0.2cm}
Following \cite{MR658729}, we now define Schur functors and Weyl functors that are indexed by partitions $\lambda$ as strict polynomial functors. 
\begin{definition}
Given a partition $\lambda = (k_1 \geq k_2 \geq \hdots \geq k_r > 0)$, we write $|\lambda | = \sum_{i=1}^r k_r$ for its size of $\lambda$ and $\text{ht}(\lambda) = r$ for its height.
\end{definition}
\begin{definition}
We represent a partition $\lambda = (k_1 \geq k_2 \geq \hdots \geq k_r > 0)$ with a diagram containing $r$ rows and such that for each $i$, the $i^{\text{th}}$-row contains $k_i$ columns. Such a representation is called a Young diagram.
\end{definition}
\begin{example}
The Young diagram of the partition $\lambda = (4,2,1)$ is
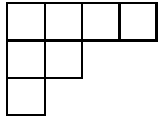
\begin{equation*}
\begin{tikzcd}[ampersand replacement=\&]
    \begin{ytableau}
	\phantom{1} & \phantom{1} & \phantom{1}  & \phantom{1}  \\
		\phantom{1} & \phantom{1} \\
			\phantom{1} 
    \end{ytableau}
\end{tikzcd}
\end{equation*}
Its size is $7$ and its height is $3$.
\end{example}

\begin{definition}
Given a partition $\lambda = (k_1 \geq k_2 \geq \hdots \geq k_r > 0)$, we define its conjugate partition $\lambda^\prime = (k_1^\prime \geq k_2^\prime \geq \hdots \geq k_{s}^\prime > 0)$ as the partition where $k_i^\prime$ is the number of terms of $k_j$ that are greater or equal to $i$. Note that $\lambda$ and $\lambda^\prime$ have the same size. Any integer $l$ between $1$ and $|\lambda|$ determines a unique position $(i,j)$ in the Young diagram of $\lambda$ such that $l = k_1 + \cdots k_{i-1} + j$. Then, we define a permutation on $|\lambda |$-letters $\sigma_\lambda \in \mathfrak{S}_{|\lambda |}$ by setting
\begin{equation*}
\sigma_\lambda(l) = k^\prime_1 + \cdots k^\prime_{i-1} + j.
\end{equation*}
Note that we have $\sigma_{\lambda^\prime} = \sigma_\lambda^{-1}$.
\end{definition}
\begin{example}
The conjugate partition of $\lambda = (8,4,2)$ 
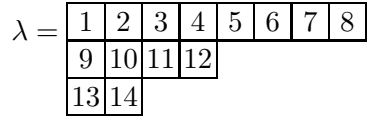
\begin{equation*}
\begin{tikzcd}[ampersand replacement=\&]
   \lambda =  \begin{ytableau}
	{1} & {2} & {3}  & {4} & 	{5} & {6} & {7}  & {8}  \\
		{9} & {10} & {11}  & {12} \\
			{13} & {14} 
    \end{ytableau}
\end{tikzcd}
\end{equation*}
is $\lambda^\prime = (3,3,2,2,1,1,1,1)$,
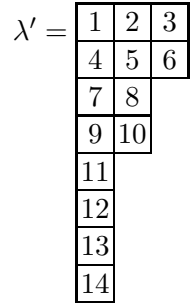
\begin{equation*}
\begin{tikzcd}[ampersand replacement=\&]
    \lambda^\prime = \begin{ytableau}
	{1} & {2} & {3}  \\
		{4} & {5} & {6}  \\
	{7} & {8}  \\
	{9} & {10}  \\
	{11}  \\
	{12}  \\
	{13}  \\
	{14}  \\
    \end{ytableau}
\end{tikzcd}
\end{equation*}
and we have
\begin{equation*}
\sigma_\lambda = 
\begin{pmatrix}
	1 & 2 & 3 & 4 & 5 & 6 & 7 & 8 & 9 & 10 & 11 & 12 & 13 & 14 \\
	1 & 4 & 7 & 9 & 11 & 12 & 13 & 14 & 2 & 5 & 8 & 10 & 3 & 6
\end{pmatrix}.
\end{equation*}
\end{example}
\begin{definition}
Let $\lambda = (k_1 \geq k_2 \geq \hdots \geq k_r > 0)$ denote a partition, $\lambda^\prime = (k^\prime_1 \geq k_2^\prime \geq \hdots \geq k_s^\prime > 0)$ its conjugate partition and $V$ a finite dimensional vector space over $k$. We define $S_\lambda V$ as the image of the map
\begin{equation*}
\begin{tikzcd}
\bigotimes_{1 \leq j \leq s} \Lambda^{k_j^\prime} {V} \arrow[r,"\Delta^{\otimes s}"] & {V}^{\otimes |\lambda |} \arrow[r,"\sigma_{\lambda}"] & {V}^{\otimes |\lambda |} \arrow[r,"\nabla^{\otimes r}"] & \bigotimes_{1 \leq i \leq r} \Sym^{k_i} V
\end{tikzcd}
\end{equation*}
where $\Delta : \Lambda^l{V} \rightarrow {V}^{\otimes l}$ is the comultiplication given by
\begin{equation*}
\Delta(v_1\wedge \cdots \wedge v_l) = \sum_{\sigma \in \mathfrak{S}_l}\epsilon(\sigma)v_{\sigma(1)}\otimes \cdots \otimes v_{\sigma(l)},
\end{equation*}
$\nabla :  {V}^{\otimes l} \rightarrow \Sym^l {V}$ is the multiplication given by 
\begin{equation*}
\nabla(v_1\otimes \cdots \otimes v_l) = v_1 \cdots v_l,
\end{equation*}
and $\sigma_\lambda : {V}^{\otimes |\lambda |} \rightarrow {V}^{\otimes |\lambda |}$ is given by
\begin{equation*}
\sigma_\lambda(v_1\otimes \cdots v_{|\lambda |}) = v_{\sigma_\lambda(1)}\otimes \cdots \otimes v_{\sigma_\lambda(|\lambda |)}.
\end{equation*}
We define $W_\lambda {V}$ as the image of the map
\begin{equation*}
\begin{tikzcd}
\bigotimes_{1 \leq i \leq r} \Gamma^{k_i} {V}  \arrow[r,"\Delta^{\otimes r}"] & {V}^{\otimes |\lambda |} \arrow[r,"\sigma_{\lambda^\prime}"] & {V}^{\otimes |\lambda |} \arrow[r,"\nabla^{\otimes s}"] & \bigotimes_{1 \leq j \leq s} \Lambda^{k^\prime_j} {V}
\end{tikzcd}
\end{equation*}
where $\nabla : {V}^{\otimes l} \rightarrow \Lambda^l{V}$ is the canonical quotient map and $\Delta : \Gamma^{l}{V} \rightarrow {V}^{\otimes l}$ is the canonical inclusion. Note that we consider both the exterior and the symmetric algebras as Hopf algebras. These construction are functorial in $V$ and define strict polynomial functors $S_\lambda$ and $W_\lambda$ that are homogeneous of degree $|\lambda |$.
\end{definition}
\begin{example}
We give the following examples:
\begin{enumerate}
\item If $\lambda = (n)$, then $S_\lambda = \Sym^n$ and $W_\lambda = \Gamma^n$. 
\item If $\lambda = (1,\cdots,1)$ is a partition of $n$, then  $S_\lambda = W_\lambda = \Lambda^n$.
\end{enumerate}
\end{example}
\begin{proposition}
Let $\lambda$ be a partition of $d$ and $V \in \FinVect_k$. We have an isomorphism
\begin{equation*}
{S_{\lambda}(V)}^{\vee} = W_\lambda(V^{\vee})
\end{equation*}
which is functorial in $V$.
\end{proposition}
\begin{proof}
Follows from the fact that $\Sym^n(V)^{\vee}  = \Gamma^n(V^{\vee})$ and $\Lambda^n(V)^{\vee}  = \Lambda^n(V^{\vee})$.
\end{proof}
\vspace*{0.2cm}
We state the main result of this section.
\begin{proposition}[{\cite[Lemma 3.4]{MR1427618}}]\label{prop_eval}
Let $d \geq 1$ be an integer and $V$ a vector space of dimension $n$. If $n\geq d$, the evaluation functor at $V$
\begin{equation*}
\begin{tikzcd}[row sep =small]
\ev_V : \Pol \arrow[r] & \Rep_k(GL(V)) \\
T \arrow[r,mapsto] & T(V)
\end{tikzcd}
\end{equation*}
restricts to an equivalence of category between $\Pol_d$ and the category $\Rep_k(GL(V))^{\Pol}_d$ of polynomial\footnote{Recall that a rational representation $M$ of $GL(V)$ is polynomial if its action of the algebraic group $GL(V)$ extends to an action of the algebraic monoid $\End(V)$.} representations of $\GL(V)$ where $\mathbb{G}_m$ acts by $z \mapsto z^d$. Moreover, through this equivalence, the Schur functor $S_\lambda$ maps to the costandard module $\nabla(\lambda)$ and the Weyl functor $W_\lambda$ maps to the standard module $\Delta(\lambda)$ where we see $\lambda = (k_1, \ldots, k_r)$ as a character of the standard maximal torus of $\GL(V)$
\begin{equation*}
\lambda : \begin{pmatrix}
t_1 & & \\
& \ddots & \\
& & t_n
\end{pmatrix} \mapsto t_1^{k_1} \cdots t_r^{k_r}.
\end{equation*}
\end{proposition}
\begin{rmrk}
All the results of this section are valid if we replace $\FinVect_k$ with the category $\Loc(X)$ of locally free sheaves of finite rank over a $k$-scheme $X$.
\end{rmrk}
\subsection{A plethysm in positive characteristic under additional hypothesis}
Take two Schur functors $S_{\lambda}$ and $S_{\mu}$ as strict polynomial functors over $k$ and consider the composition $S_{\lambda} \circ S_{\mu}$. It is a strict polynomial functor homogeneous of degree $|\lambda| |\mu |$. Since the category of algebraic representation of $\GL_n$ is not semisimple, we have no reason to hope for a decomposition of $S_{\lambda} \circ S_{\mu}$ as a direct sum of Schur functors. One might hope, that there exists at least a filtration of $S_{\lambda} \circ S_{\mu}$ where the graded pieces are Schur functors. Unfortunately, Boffi \cite{MR1128609} and Touzé \cite[Corollary 6.10.]{MR3042627} have found counter-examples to the existence of such filtrations for plethysm of the form $\Sym^k \circ \Sym^d$, $\Lambda^k \circ \Sym^d$ and $\Sym^k \circ \Lambda^d$ with $d \geq 3$ and $p | k$. More precisely, Touzé has found an obstruction to the existence of such filtration that lives in the $p$-torsion of the homology of the Eilenberg-Mac Lane space $K(\Z,d)$. In this section, we prove the following existence result. 
\begin{proposition}\label{th_plethysm}
Let $\lambda$ and $\mu$ be partitions. If $p \geq 2|\lambda |-1$, the strict polynomial functor $S_{\lambda} \circ S_{\mu}$ admits a finite filtration
\begin{equation*}
0 =T^n \subsetneq T^{n-1} \subsetneq
 \cdots \subsetneq
 T^0 = S_{\lambda} \circ S_{\mu}
\end{equation*}
by strict polynomial functors of degree $|\lambda| |\mu |$ where the graded pieces are Schur functors.
\end{proposition}
\vspace*{0.2cm}
We start with the following lemma.
\begin{lemma}\label{lem7}
If $p \geq 2|\lambda |-1$, then $S_{\lambda}$ is a direct summand of $(\cdot)^{\otimes |\lambda |}$ in $\Pol_{|\lambda |}$.
\end{lemma}
\begin{rmrk}
If $S_{\lambda} = \Lambda^n$, then it is enough to ask $p > n$.
\end{rmrk}
\begin{proof}
Write $\lambda = (k_1 \geq \cdots \geq k_r > 0)$. By proposition \ref{prop_eval}, it is enough to prove that $S_{\lambda}V$ is a direct summand of $V^{\otimes |\lambda |}$ in the category of $\GL(V)$-modules for one $k$-vector space of dimension greater than $|\lambda |$. Consider a vector space $V$ of dimension $n$. Notice that the surjection
\begin{equation*}
\nabla^{\otimes r} : V^{\otimes  |\lambda |} \rightarrow \Sym^{\lambda} V := \bigotimes_{1 \leq i \leq r} \Sym^{k_i} V
\end{equation*}
admits a section when $p > \max_i k_i = k_1$. Indeed, we define it as $s = s_1 \otimes \cdots \otimes s_r$ where
\begin{equation*}
s_i(v_1v_2\cdots v_r) = \frac{1}{k_i!}\sum_{\sigma \in \mathfrak{S}_{k_i}}v_{\sigma(1)}\otimes \cdots \otimes v_{\sigma(k_i)}.
\end{equation*}
By definition $S_{\lambda}V$ is a sub-$\GL(V)$-module of $\Sym^{\lambda} V$ and we would like to find a condition on $p$ that guarantees it is also a direct summand. The following exact sequence of $\GL(V)$-modules
\begin{equation*}
\begin{tikzcd}
0 \arrow[r] & S_{\lambda}V \arrow[r] & \Sym^{\lambda} V \arrow[r] & \Sym^{\lambda} V/S_{\lambda}V \arrow[r] & 0
\end{tikzcd}
\end{equation*}
is split if we can show that $\Ext^1_{\GL(V)}(\Sym^{\lambda} V/S_{\lambda}V,S_{\lambda}V)$ vanishes. Let $\tilde{\lambda}$ denote the character $(|\lambda |, 0, \cdots, 0)$ of the standard maximal torus of $\GL(V)$. Since $\Sym^{\lambda}V$ is of highest weight $\tilde{\lambda}$, the $\GL(V)$-modules $S_{\lambda}V$ and $\Sym^{\lambda} V/S_{\lambda}V$ are filtered by simple modules $L(\nu)$'s with $\nu \leq \tilde{\lambda}$. Under the assumption that $\tilde{\lambda}$ is $p$-small, proposition \ref{prop_psmall} implies that we have isomorphisms
\begin{equation*}
L(\nu) = \nabla(\nu) = \Delta(\nu)
\end{equation*}
for all characters $\nu$ satisfying $\nu \leq \tilde{\lambda}$. Since $\Ext^1_{\GL(V)}(\Sym^{\lambda} V/S_{\lambda}V,S_{\lambda}V)$ is the limit of a spectral sequence involving Ext groups
\begin{equation*}
\Ext^1_{\GL(V)}(L(\nu),L(\nu^\prime)) = \Ext^1_{\GL(V)}(\Delta(\nu),\nabla(\nu^\prime)),
\end{equation*}
that vanishes by proposition \ref{propDonkyn}, $\Ext^1_{\GL(V)}(\Sym^{\lambda} V/S_{\lambda}V,S_{\lambda}V)$ must vanish. In conclusion, we have the desired splitting, provided that $\tilde{\lambda}$ is $p$-small, i.e. 
\begin{equation*}
\begin{aligned}
p \geq \max_{\alpha \in \Phi^+} \langle \tilde{\lambda} + \rho, \alpha^{\vee} \rangle &=  \max_{ 1 \leq i < j \leq n} \langle \tilde{\lambda} + \rho, \varepsilon_i -\varepsilon_j \rangle \\
&=  \max_{ 1 \leq i < j \leq n} \langle (|\lambda |+\frac{n-1}{2}, \frac{n-3}{2}, \cdots, -\frac{n-1}{2} ), \varepsilon_i -\varepsilon_j \rangle \\
&= |\lambda |+n-1.
\end{aligned}
\end{equation*}
Since our argument is valid only when $n \geq |\lambda |$, we get the bound $p \geq 2| \lambda |-1$.
\end{proof}
\begin{proof}[Proof of proposition \ref{th_plethysm}]
By proposition \ref{prop_eval}, it is enough to show that $S_{\lambda}\circ S_{\mu} (V)$ admits a $\nabla$-filtration as a $\GL(V)$-module where $V$ is one vector space of dimension greater that $|\lambda ||\mu |$. Consider a vector space $V$ of dimension $n\geq |\lambda ||\mu |$. By lemma \ref{lem7}, $S_{\lambda}(S_{\mu}V)$ is a direct summand of $(S_{\mu}V)^{\otimes |\lambda |}$ as $\GL(S_{\mu}V)$-modules. After restriction to the category of $\GL(V)$-modules through the map $\GL(V) \rightarrow \GL(S_{\mu}V)$ induced by $S_{\mu}$, $S_{\lambda}(S_{\mu}V)$ is again a direct summand of $(S_{\mu}V)^{\otimes |\lambda |}$. By corollary \ref{cor2}, the $\GL(V)$-module $(S_{\mu}V)^{\otimes |\lambda |}$ admits a $\nabla$-filtration. By corollary \ref{cor1}, the $\GL(V)$-module $S_{\lambda}\circ S_{\mu} (V)$ admits a $\nabla$-filtration.
\end{proof}
\begin{rmrk}
Under the assumption of the proposition \ref{th_plethysm}, the partitions (counted with multiplicity) of the Schur functors appearing in the graded pieces of the filtration of $S_{\lambda} \circ S_{\mu}$ are the same as the one appearing in the decomposition \eqref{eq_pleth} over the complex numbers. This is just a consequence of the $\Z$-linearity of the $\nabla(\lambda)$'s in the space ${X^*(T)}^{W}$ of $W$-invariants characters, but we reprove it directly. First note that the weights of the $\GL(V)$-module $S_{\lambda}\circ S_{\mu}(V)$ where $V$ is a vector space of dimension $\geq |\lambda | |\mu |$ do not depend on the characteristic of the base field of $V$. Then, we are left to check that a direct sum $M = \bigoplus_{\lambda} \nabla(\lambda)^{\oplus c_{\lambda}}$ of costandard modules is uniquely determined by its characters $\text{ch}(M) = \sum_{\eta} d_{\eta} \eta$. We prove it with a descending induction on the number of distinct factors of $M$. Consider the highest weight $\eta_0$ appearing in the sum $\text{ch}(M)$. Clearly, $\nabla(\eta_0)$ is a direct factor of the module $M$ because $\eta_0$ cannot appear in the weights of a costandard module $\nabla(\lambda)$ with $\lambda < \eta_0$. Moreover, the multiplicity of $\nabla(\eta_0)$ in $M$ is exactly $d_{\eta_0}$ and we can pursue the induction with $M^\prime =  \bigoplus_{\lambda < \eta_0} \nabla(\lambda)^{\oplus c_{\lambda}}$. 
\end{rmrk}
\section{Positive vector bundles}\label{sect_pos}
In positive characteristic, Hartshorne has defined in \cite{MR193092} two non-equivalent notions of ampleness for vector bundles. The first notion is simply called ampleness, the second, strictly stronger, is called $p$-ampleness. Furthermore, Kleiman has defined in \cite{MR251044} a third notion, again strictly stronger, called cohomological $p$-ampleness. For the convenience of the reader, we recall some well known results about globally generated sheaves and ampleness notions in positive characteristic.\\

In subsection \ref{sect1}, we consider an effective Cartier divisor $D$ and we define a positivity notion for vector bundles called $(\varphi,D)$-ampleness. In the case of line bundles, this notion is equivalent to being nef and big with $D$ as exceptional divisor. Let $X$ be a projective scheme over $k$. We write $\varphi : X \rightarrow X^{(p)}$ for the relative geometric Frobenius of $X$. If $\Fc$ is a sheaf on $X$ and $r \geq 1$ is an integer, we write $\Fc^{(p^r)} := {(\varphi^r)}^*{(\varphi^r)}_*\Fc$. We endow the finite dimensional $\mathbb{R}$-vector space $A_1(X)$ of $1$-cycles on $X$ modulo linear equivalence with a norm $\lVert \cdot \rVert$. If $C$ is a projective curve and $\E$ is a vector bundle on $C$, we denote by $\delta(\E)$ the minimum of the degrees of quotient line bundles of $\E$.

\subsection{Globally generated sheaves}\label{sect_glob}
\begin{definition}
We say that a coherent sheaf $\Fc$ is globally generated at $x \in X$ if the canonical map
\begin{equation*}
H^0(X,\Fc)\otimes_k \Oc_{X} \rightarrow \Fc
\end{equation*}
is surjective at $x \in X$. We say $\Fc$ is globally generated over $U \subset X$ if it is globally generated at $x$ for all $x \in U$. 
\end{definition}
\vspace*{0.2cm}
The following lemma is well-known.
\begin{lemma}\label{prop10}
Let $x$ be a point of $X$. We have the following assertions.
\begin{enumerate}
\item The direct sum of two globally generated sheaves at $x$ is globally generated at $x$.
\item Let $\Fc \rightarrow \Fc^\prime$ be a morphism of coherent sheaves wich is sujective at $x$. If $\Fc$ is globally generated at $x$, then so is $\Fc^\prime$.
\item The tensor product of two globally generated sheaves at $x$ is globally generated at $x$.
\item The pullback of a globally generated sheaf at $x$ is globally generated at $x$.
\end{enumerate}
\end{lemma}
\begin{proof}
Left to the reader.
\end{proof}
\subsection{Ample bundles}\label{sect2}
\begin{definition}
We say that a line bundle $\Lc$ over $X$ is ample if the following equivalent propositions are satisfied.
\begin{enumerate}
\item For all coherent sheaf $\Fc$ on $X$, there is an integer $n_0$ such that $\Fc \otimes \Lc^{\otimes n}$ is globally generated for all $n \geq n_0$.
\item For all coherent sheaf $\Fc$ on $X$, there is an integer $n_0$ such that the cohomology groups $H^i(X,\Fc \otimes \Lc^{\otimes n})$ vanishes for all $i>0$, $n \geq n_0$.
\item For any subvariety $V \subset X$, we have
\begin{equation*}
c_1(\Lc)^{\dim V} \cdot [V] > 0
\end{equation*}
in the Chow ring of $X$.
\end{enumerate} 
\end{definition}
\begin{proof}
For the equivalence of the definitions, see \cite[Prop 1.1/1.2/1.4]{MR193092}.
\end{proof}
From now on, we fix an ample line bundle $\Oc_X(1)$ on $X$ and we write $\Fc(m)$ instead of $\Fc \otimes \Oc_X(1)^{\otimes m}$ for any coherent sheaf $\Fc$ on $X$ and integer $m$. We recall the definition of relative ample line bundles.
\begin{definition}
Let $Y$ be a projective scheme over a base scheme $S$. Write $f : Y \rightarrow S$ for the structure morphism. We say that a line bundle $\mathcal{L}$ on $Y$ is $f$-ample if the following equivalent propositions are satisfied.
\begin{enumerate}
\item For all coherent sheaf $\Fc$ on $Y$, there is an integer $n_0$ such that the adjunction morphism $f^*f_*(\Fc \otimes \Lc^{\otimes n}) \rightarrow \Fc \otimes \Lc^{\otimes n}$ is surjective for all $n \geq n_0$.
\item For all coherent sheaf $\Fc$ on $Y$, there is an integer $n_0$ such that the higher direct image sheaves $R^if_*(\Fc \otimes \Lc^{\otimes n})$ vanishes for all $i>0$, $n \geq n_0$. 
\end{enumerate} 
\end{definition}
\begin{proof}
For the equivalence of the definitions, see \cite[Theorem 1.7.6]{MR2095472} or \cite[\href{https://stacks.math.columbia.edu/tag/02O1}{Lemma 02O1}]{stacks-project}.
\end{proof}
\begin{definition}[\cite{MR193092}]
We say that a vector bundle $\E$ over $X$ is ample if the universal line bundle $\Oc(1)$ is ample on the projective bundle $\Pb(\E)$. Note that the universal line bundle $\Oc(1)$ is equal to the associated line bundle $\Lc_{\lambda}$ with $\lambda = (1,0,\cdots,0)$ for the canonical isomorphism $X^*(T) \simeq \mathbb{Z}^n$ where $T$ is the standard maximal torus of $\GL_n$.
\end{definition}
\begin{proposition}\label{amp_criterion}
Let $\E$ be a vector bundle on $X$. The following assertions are equivalent.
\begin{enumerate}
\item $\E$ is ample on $X$.
\item For all coherent sheaf $\Fc$ on $X$, there is an integer $n_0$ such that $\Fc \otimes \Sym^{n}\E$ is globally generated for all $n \geq n_0$.
\item For all coherent sheaf $\Fc$ on $X$, there is an integer $n_0$ such that the cohomology groups $H^i(X,\Fc \otimes \Sym^{n}\E)$ vanishes for all $i>0$, $n \geq n_0$.
\item There exists a real number $\varepsilon > 0$ such that for all finite morphism $g : C \rightarrow X$ where $C$ is a curve, we have
\begin{equation*}
\delta(g^*\E) \geq \varepsilon \lVert g_*C \rVert.
\end{equation*}
Recall that $\delta(g^*\E)$ is the minimum of the degrees of quotient line bundles of $g^*\E$ and $\lVert \cdot \rVert$ denotes a norm on $A_1(X)$, the $k$-vector space of $1$-cycles modulo linear equivalence.
\end{enumerate}
\end{proposition}
\begin{proof}
See \cite[Proposition 3.2/3.3]{MR193092} for a complete proof of $$(1) \Leftrightarrow (2) \Leftrightarrow (3).$$For $(1) \Leftrightarrow (4)$, this numerical criterion is due to Barton \cite{MR289525}.
\end{proof}

\begin{proposition}\label{prop7}
We have the following assertions
\begin{enumerate}
\item Let $\E$ and $\E^\prime$ be two ample vector bundles on $X$. Then $\E \oplus \E^\prime$ is ample.
\item Consider an extension of vector bundles on $X$
\begin{equation*}
\begin{tikzcd}
0 \arrow[r] & \E_1 \arrow[r] & \E \arrow[r] & \E_2 \arrow[r] & 0
\end{tikzcd}
\end{equation*}
where $\E_1$ and $\E_2$ are ample. Then $\E$ is ample.
\item Let $\E$ and $\E^\prime$ be two vector bundles on $X$ such that $\E$ is ample and $\E^\prime$ is globally generated over $X$. Then the tensor product $\E \otimes \E^\prime$ is an ample vector bundle.
\item Let $\E \rightarrow \E^\prime$ be a surjective morphism of $\Oc_X$-modules between two vector bundles. If $\E$ is ample, then so is $\E^\prime$.
\item The tensor product of ample vector bundles over $X$ is ample. 
\end{enumerate}
\end{proposition}
\begin{proof}
See \cite[Proposition 2.2/Corollary 2.5]{MR193092} for assertions $(1),(3),(4)$, \cite[Corollary 3.4]{MR193092} for assertion $(4)$ and \cite[Theorem. 3.3]{MR289525} for assertion $(5)$.
\end{proof}

\begin{proposition}\label{squareampclassic}
If $\E$ is a vector bundle such that $\E^{\otimes n}$ is ample for some $n\geq 1$, then $\E$ is also ample.
\end{proposition}
\begin{proof}
Assume that $\E^{\otimes n}$ is ample. As a quotient of $\E^{\otimes n}$, $\Sym^n\E$ is ample and we conclude with \cite[Proposition 2.4]{MR193092}.
\end{proof}
\begin{proposition}\label{prop11}
Let $f : Y \rightarrow X$ be a finite morphism of projective schemes and $\E$ be an ample vector bundle on $X$. If $\E$ is ample on $X$, then $f^*\E$ is ample on $Y$. If furthermore $f$ is assumed surjective, then the converse holds.
\end{proposition}
\begin{proof}
See \cite[Proposition 1.2.9]{MR2095471} and  \cite[Corollary 1.2.24]{MR2095471}.
\end{proof}
\begin{corollary}\label{prop8}
Let $\E$ be a vector bundle and $r\geq 1$ an integer. Then $\E$ is ample if and only if $\E^{(p^r)}$ is ample.
\end{corollary}
\begin{proof}
Since the Frobenius map is finite surjective, it follows from the previous proposition.
\end{proof}
\begin{definition}[\cite{MR193092}]\label{def2}
We say that a vector bundle $\E$ on $X$ is $p$-ample if for all coherent sheaf $\Fc$ on $X$, there is an integer $r_0$ such that $\Fc \otimes \E^{(p^r)}$ is globally generated for all $r \geq r_0$.
\end{definition}
\begin{lemma}\label{lem6}
For any coherent sheaf $\Fc$ and $m \geq 0$ large enough, we can write $\Fc$ as a quotient of $\Oc_X(-m)^{\oplus s}$ for a suitable $s \geq 1$.
\end{lemma}
\begin{proof}
Choose $m \geq 0$ large enough such that $\Fc(m)$ is globally generated over $X$. We get a sujective morphism
\begin{equation*}
\Oc_X^{\oplus s} \rightarrow \Fc(m)
\end{equation*}
for some $s \geq 1$ and then we tensor by $\Oc_X(-m)$.
\end{proof}
\begin{proposition}\label{prop3}
In the definition \ref{def2}, we can restrict ourselves to coherent sheaves of the form $\Fc = \Oc_X(-m)$ for all $m \geq 0$ large enough.
\end{proposition}
\begin{proof}
We use the lemma \ref{lem6} to write $\Fc$ as a quotient of $\Oc_X(-m)^{\oplus s}$ for a suitable $s \geq 1$. Take $n$ large enough such that $\Oc_X(-m) \otimes \E^{(p^r)}$ is globally generated. Since the quotient of a globally generated sheaf is globally generated, we get that $\Fc \otimes \E^{(p^r)}$ is globally generated.
\end{proof}
\begin{proposition}\label{pamptoamp}
If $\E$ is $p$-ample on $X$, then $\E$ is ample. 
\end{proposition}
\begin{proof}
Choose $n$ large enough such that $\E^{(p^n)}(-1)$ is globally generated. We deduce that $\E^{(p^n)}$ is quotient of $\Oc_X(1)^{\oplus s}$ for a suitable $s\geq 1$. By the assertion $(3)$ of proposition \ref{prop7}, $\E^{(p^n)}$ is ample and by corollary \ref{prop8}, $\E$ is ample.
\end{proof}
\begin{rmrk}
The converse to the previous proposition is false in general (see \cite{MR296078} for a counter example). However, in the special case where $\E$ is a line bundle or $X$ is curve, it holds by \cite[Proposition 7.3]{MR193092}.
\end{rmrk}
\begin{proposition}\label{prop12}
We have the following assertions:
\begin{enumerate}
\item Let $\E$ and $\E^\prime$ be two $p$-ample vector bundles on $X$. Then $\E \oplus \E^\prime$ is $p$-ample.
\item Let $\E$ and $\E^\prime$ be two vector bundles on $X$ such that $\E$ is $p$-ample and $\E^\prime$ is globally generated over $X$. Then, the tensor product $\E \otimes \E^\prime$ is a $p$-ample vector bundle.
\item Let $\E \rightarrow \E^\prime$ be a surjective morphism of $\Oc_X$-modules between two vector bundles. If $\E$ is $p$-ample, then $\E^\prime$ is also $p$-ample.
\item The tensor product of $p$-ample vector bundles over $X$ is $p$-ample. 
\end{enumerate}
\end{proposition}
\begin{proof}
See \cite[Proposition 6.4/Corollary 6.7]{MR193092} for assertions $(1),(2)$ and $(4)$. Hartshorne does not state assertion $(3)$, so we give a proof. Let $\Fc$ be a coherent sheaf and $r_0 \geq 1$ be an integer such that $\Fc \otimes \E^{(p^r)}$ is globally generated for all $r \geq r_0$. For all $r \geq r_0$, the surjective morphism $\E \rightarrow \E^\prime$ induces a surjective morphism of $\Oc_X$-modules
\begin{equation*}
\Fc \otimes \E^{(p^r)}\rightarrow \Fc \otimes {(\E^\prime)}^{(p^r)}
\end{equation*}
and from assertion $(2)$ of lemma \ref{prop10}, the module $\Fc \otimes {(\E^\prime)}^{(p^r)}$ is globally generated over $X$.
\end{proof}
There is no known cohomological criterion for $p$-ampleness. However, Kleiman has defined in \cite{MR251044} the strictly\footnote{See again \cite{MR296078} for an example of $p$-ample vector bundle that is not cohomologically $p$-ample} stronger notion of cohomological $p$-ampleness. 
\begin{definition}
We say that a vector bundle $\E$ on $X$ is cohomologically $p$-ample if for all coherent sheaves $\Fc$ on $X$, there is an integer $r_0$ such that the cohomology groups $H^i(X,\Fc \otimes \E^{(p^r)})$ vanishes for all $i>0$, $r \geq r_0$.
\end{definition}
\begin{proposition}
If $\E$ is cohomologically $p$-ample on $X$, then $\E$ is $p$-ample.
\end{proposition}
\begin{proof}
See \cite[Proposition 9]{MR251044}.
\end{proof}
To the best of our knowledge, the following statements does not appear in the literature so we state them and provide a proof.  
\begin{lemma}\label{prop13}
A direct sum of cohomologically $p$-ample vector bundle is cohomologically $p$-ample.
\end{lemma}
\begin{proof}
It follows directly from the isomorphism
\begin{equation*}
H^i(X,\Fc\otimes (\E \oplus \E^\prime)^{(p^r)}) = H^i(X,\Fc\otimes \E^{(p^r)}) \oplus H^i(X,\Fc\otimes {\E^{\prime}}^{(p^r)})
\end{equation*}
\end{proof}
\begin{lemma}\label{prop14}
Let $f : Y \rightarrow X$ be a finite morphism of projective schemes and $\E$ be a cohomologically $p$-ample vector bundle on $X$. Then $f^*\E$ is cohomologically $p$-ample on $Y$.
\end{lemma}
\begin{proof}
Let $\Fc$ be a coherent sheaf on $Y$. Since $f$ is finite, the Leray spectral sequence degenerates at page $2$ and we have isomorphisms
\begin{equation*}
H^i(X,f_*(\Fc \otimes f^*\E^{(p^r)})) = H^i(Y,\Fc\otimes f^*\E^{(p^r)})
\end{equation*}
for all $i \geq 0$ and $r \geq 0$. Since $f$ is finite, the pushforward $f_*\Fc$ is a coherent $\Oc_X$-module and the projection formula implies that
\begin{equation*}
f_*(\Fc \otimes f^*\E^{(p^r)}) = f_*\Fc \otimes \E^{(p^r)}.
\end{equation*}
Since $\E$ is cohomologically $p$-ample on $X$, there is an integer $r_0 \geq 1$ such that
\begin{equation*}
H^i(X, f_*\Fc \otimes \E^{(p^r)}) = 0 = H^i(Y,\Fc\otimes (f^*\E)^{(p^r)})
\end{equation*} 
for all $i>0$ and $r\geq r_0$. In particular, $f^*\E$ is cohomologically $p$-ample on $Y$.
\end{proof}
\subsection{$(\varphi,D)$-ample bundles}\label{sect1}
If $D$ is a Cartier divisor, we write $\mathcal{O}_X(D)$ for the associated line bundle. If $\mathcal{F}$ is a coherent sheaf on $X$, then we simply write $\mathcal{F}(D)$ instead of $\mathcal{F}\otimes \mathcal{O}_X(D)$. We consider an \emph{effective} Cartier divisor $D$ on $X$ and we define the notion of $(\varphi,D)$-ampleness for vector bundles over $X$.
\begin{definition}\label{def3}
Let $\E$ be a vector bundle over $X$. We say that $\E$ is $(\varphi,D)$-ample if there is an integer $r_0 \geq 1$ such that for all integer $r \geq r_0$, the vector bundle $\E^{(p^r)}(-D)$ is ample.
\end{definition}
\vspace*{0.2cm}
In the case of line bundles, $(\varphi,D)$-ampleness has the following characterization.
\begin{proposition}\label{D-amp}
Let $\Lc$ be a line bundle over $X$. Then $\Lc$ is $(\varphi,D)$-ample if and only if $\Lc$ is nef and there is an integer $n_0 \geq 1$ such that $\mathcal{L}^{\otimes n_0}(-D)$ is ample.
\end{proposition}
\begin{proof}
Note that $\Lc^{(p^r)} = \Lc^{\otimes p^r}$ for all $r \geq 0$. Assume that $r_0 \geq 1$ is an integer such that $\Lc^{\otimes p^r}(-D)$ is ample for all $r \geq r_0$. If $\mathcal{L}$ was not nef, we could find a subcurve $C \subset X$ such that the intersection product
\begin{equation*}
c_1(\mathcal{L})\cdot [C]
\end{equation*}
is negative. It would imply that the intersection product
\begin{equation*}
c_1(\mathcal{L}^{\otimes p^r}(-D)) \cdot [C] = p^r \underbrace{\left( c_1(\mathcal{L})\cdot [C]\right)}_{<0} - {D}\cdot [C]
\end{equation*}
is negative for some $r \geq r_0$ large enough, which contradicts the ampleness of $\mathcal{L}^{\otimes p^r}(-D)$. Inversely, we assume that $\Lc$ is nef and there exists an integer $n_0 \geq 1$ such that $\mathcal{L}^{\otimes n_0}(-D)$ is ample. Let $r$ be an integer such that
\begin{equation*}
r \geq \log_pn_0
\end{equation*}
and consider
\begin{equation*}
\Lc^{(p^r)}(-D) = \Lc^{\otimes p^r}(-D) = \Lc^{\otimes n_0}(-D) \otimes \Lc^{\otimes p^r-n_0}
\end{equation*}
which is ample as the tensor product of an ample line bundle with a nef line bundle.
\end{proof}
\begin{rmrk}
In the case of line bundles we will drop the $\varphi$ from the notation and simply say that the line bundle is $D$-ample.
\end{rmrk}
\begin{proposition}\label{prop20}
Let $\mathcal{L}$ be a line bundle over $X$. The following proposition are equivalent:
\begin{enumerate}
\item $\mathcal{L}$ is nef and big
\item There exists an effective Cartier divisor $H$ on $X$, such that $\mathcal{L}$ is $H$-ample.
\end{enumerate}
\end{proposition}
\begin{proof}
Assume that there exists an effective Cartier divisor $H$ on $X$ such that $\mathcal{L}$ is $H$-ample. We have seen in proposition \ref{D-amp} that $\Lc$ is nef and there is an integer $n_0 \geq 1$ such that $\Lc^{\otimes n_0}(-H)$ is ample. Moreover, since we can write $\mathcal{L}^{\otimes n_0}$ as a tensor product
\begin{equation*}
\mathcal{L}^{\otimes n_0} = \mathcal{L}^{\otimes n_0}(-H) \otimes \mathcal{O}_X(H)
\end{equation*}
of an ample line bundle with an effective line bundle, $\mathcal{L}$ is big. We are left to show the implication $(i) \Rightarrow (ii)$. Since $\mathcal{L}$ is big, there exists an integer $n_0 \geq 1$ and an ample line bundle $\mathcal{A}$ such that $\mathcal{L}^{\otimes n_0}\otimes \mathcal{A}^{-1} = \mathcal{O}_X(H)$ with $H$ an effective divisor. In particular, the line bundle $\mathcal{L}^{\otimes n_0}(-H)$ is ample. We conclude with proposition \ref{D-amp}.
\end{proof}
We prove some stability properties of $(\varphi,D)$-ample vector bundles. We first prove the following easy lemma.
\begin{lemma}\label{lem_delta}
Let $C$ be a projective curve and $\E$ be a vector bundle on $C$. Recall that $\delta(\E)$ denotes the minimum of degrees of quotient line bundles of $\E$. Then, we have
\begin{enumerate}
\item If $\Lc$ is a line bundle on $C$, then $\delta(\E \otimes \mathcal{L}) = \delta(\E) + \deg \mathcal{L}$.
\item If $f : C^\prime \rightarrow C$ is a finite morphism of degree $d$ with $C^\prime$ a projective curve, then $d\delta(\E) \geq \delta(f^*\E)$.
\end{enumerate}
\end{lemma}
\begin{proof}[Proof of the lemma \ref{lem_delta}]
For $(1)$, take a line bundle $\E \twoheadrightarrow \Lc^{\prime}$ such that $\delta(\E) = \deg \Lc^\prime$. If we tensor it by $\Lc$, we get $\delta(\E \otimes \Lc) \leq \deg \Lc^\prime + \deg \Lc =  \delta(\E) + \deg \mathcal{L}$. The same argument applied to $\E \otimes \Lc^{-1}$ shows the reverse inequality. For $(2)$, take a line bundle $\E \twoheadrightarrow \Lc^{\prime}$ such that $\delta(\E) = \deg \Lc^\prime$. The pullback $f^*$ induces a quotient map $f^*\E \twoheadrightarrow f^*\Lc^{\prime} = {\Lc^{\prime}}^{\otimes d}$ which shows that $\delta(f^*\E) \leq d\deg\Lc^{\prime} = d\delta(\E)$.
\end{proof}
\begin{proposition}\label{nD}
Let $\E$ be a vector bundle on $X$ and $n\geq 1$ an integer. The following assertion are equivalent.
\begin{enumerate}
\item $\E$ is $(\varphi,D)$-ample
\item $\E$ is $(\varphi,nD)$-ample
\end{enumerate}
\end{proposition}
\begin{proof}
Assume that $\E$ is $(\varphi,D)$-ample and consider $r_0 \geq 1$ such that $$\E^{(p^r)}(-D)$$ is ample for all $r \geq r_0$. By Barton's numerical criterion of ampleness recalled in assertion $(4)$ of proposition \ref{amp_criterion}, for all $r \geq r_0$, we have a real number $\varepsilon_r > 0$ such that for all finite morphism $g : C \rightarrow X$ where $C$ is a smooth projective curve over $k$, we have
\begin{equation*}
\delta(g^*\E^{(p^r)}(-D)) \geq \varepsilon_r\lVert g_*C \rVert
\end{equation*} 
which is equivalent to
\begin{equation*}
\delta(g^*\E^{(p^r)}) -D\cdot C \geq \varepsilon_r\lVert g_*C \rVert
\end{equation*}
where $D\cdot C$ is the degree of the line bundle $g^*\Oc_X(D) = \Oc_X(D)_{|C}$ (it is also equal to the intersection number of $D$ with $C$). If $D\cdot C\leq 0$, then 
\begin{equation*}
\begin{aligned}
\delta(g^*\E^{(p^r)}(-nD)) &= \delta(g^*\E^{(p^r)})-D\cdot C - (n-1)D\cdot C \\
&\geq \delta(g^*\E^{(p^r)})-D\cdot C \\
&\geq \varepsilon_r\lVert g_*C \rVert
\end{aligned}
\end{equation*}
for all $r\geq r_0$. If $D\cdot C > 0$, we take $r_1 \geq r_0$ such that $r_1 \geq r_0 + \log_p(n)$ and let $r\geq r_1$ be an integer. Since $\E^{(p^r)}(-D)$ is ample, the bundle 
\begin{equation*}
(\E^{(p^{r_0})}(-D))^{(p^{r-r_0})} = \E^{(p^r)}(-p^{r-r_0}D)
\end{equation*}
is ample and we have
\begin{equation*}
\delta(g^*\E^{(p^r)}(-p^{r-r_0}D)) \geq \varepsilon_r^\prime \lVert g_*C \rVert
\end{equation*}
for some real number $\varepsilon_r^\prime > 0$. Thus,
\begin{equation*}
\begin{aligned}
\delta(g^*\E^{(p^r)}(-nD)) &= \delta(g^*\E^{(p^r)}(-p^{r-r_0}D)) +(p^{r-r_0}-n)D\cdot C \\
&\geq  \delta(g^*\E^{(p^r)}(-p^{r-r_0}D)) \\
&\geq \varepsilon_r^\prime \lVert g_*C \rVert.
\end{aligned}
\end{equation*}
In conclusion, we have
\begin{equation*}
\delta(g^*\E^{(p^r)}(-nD)) \geq \min(\varepsilon_r,\varepsilon_r^\prime) \lVert g_*C \rVert
\end{equation*}
for all $r \geq r_1$ and all $g : C \rightarrow X$, which means that $\E$ is $(\varphi,nD)$-ample. Inversely, consider an integer $r_0 \geq 1$ such that for all $r \geq r_0$, we have a real number $\varepsilon_r > 0$ such that for all finite morphism $g : C \rightarrow X$ where $C$ is a smooth projective curve over $k$, we have
\begin{equation*}
\delta(g^*\E^{(p^r)}) -nD\cdot C \geq \varepsilon_r\lVert g_*C \rVert.
\end{equation*}
If $D\cdot C\geq 0$, we have 
\begin{equation*}
\begin{aligned}
\delta(g^*\E^{(p^r)}(-D)) &= \delta(g^*\E^{(p^r)}(-nD)) +(n-1)D\cdot C \\
&\geq \delta(g^*\E^{(p^r)}(-nD)) \\
&\geq \varepsilon_r\lVert g_*C \rVert
\end{aligned}
\end{equation*}
for all $r\geq r_0$.
Consider an integer $r_1 \geq r_0 + \log_pn$. If $D\cdot C< 0$, we have 
\begin{equation*}
\begin{aligned}
p^{r-r_0}\delta(g^*\E^{(p^{r_0})}(-D)) &\geq \delta(g^*\E^{(p^r)}(-p^{r-r_0}D)) \\
&\geq \delta(g^*\E^{(p^r)}(-nD)) + (n-p^{r-r_0})D\cdot C \\
&\geq \delta(g^*\E^{(p^r)}(-nD)) \\
&\geq \varepsilon_r\lVert g_*C \rVert.
\end{aligned}
\end{equation*}
for all $r \geq r_1$. In conclusion, we have
\begin{equation*}
\delta(g^*\E^{(p^r)}(-D)) \geq \frac{\varepsilon_r}{p^{r-r_0}} \lVert g_*C \rVert
\end{equation*}
for all $r \geq r_1$ and all $g : C \rightarrow X$, which means that $\E$ is $(\varphi,D)$-ample.
\end{proof}
\begin{proposition}\label{prop15}
Let $\E$ and $\E^\prime$ be two $(\varphi,D)$-ample vector bundle on $X$. Then, $\E \oplus \E^\prime$ is $(\varphi,D)$-ample.
\end{proposition}
\begin{proof}
Let $r_0 \geq 1$ be an integer such that for all $r \geq r_0$, the bundles $\E^{(p^r)}(-D)$ and ${(\E^\prime)}^{(p^r)}(-D)$ are ample. For all $r \geq r_0$, we have
\begin{equation*}
(\E \oplus \E^{\prime})^{(p^r)}(-D) = \E^{(p^r)}(-D) \oplus {(\E^\prime)}^{(p^r)}(-D)
\end{equation*}
which is ample by the assertion $(1)$ of proposition \ref{prop7}.
\end{proof}
\begin{proposition}\label{prop19}
Consider an extension of vector bundles on $X$
\begin{equation*}
\begin{tikzcd}
0 \arrow[r] & \E_1 \arrow[r] & \E \arrow[r] & \E_2 \arrow[r] & 0
\end{tikzcd}
\end{equation*}
where $\E_1$ and $\E_2$ are $(\varphi,D)$-ample and assume that $X$ is regular over $k$. Then $\E$ is $(\varphi,D)$-ample.
\end{proposition}
\begin{proof}
On a regular scheme, the Frobenius morphism is flat by \cite{MR252389} or \cite[\href{https://stacks.math.columbia.edu/tag/0EC0}{Lemma 0EC0}]{stacks-project}. As a consequence, we have an integer $r_0\geq 1$ and an exact sequence
\begin{equation*}
\begin{tikzcd}
0 \arrow[r] & (\E_1)^{(p^r)}(-D) \arrow[r] & (\E)^{(p^r)}(-D) \arrow[r] &(\E_2)^{(p^r)}(-D)\arrow[r] & 0
\end{tikzcd}
\end{equation*}
of vector bundles on $X$ where $ (\E_1)^{(p^r)}(-D) $ and $(\E_2)^{(p^r)}(-D)$ are ample for all $r \geq r_0$. We conclude with assertion $(2)$ of proposition \ref{prop7}.
\end{proof}
\begin{proposition}\label{prop16}
Let $\E \rightarrow \E^\prime$ be a surjective morphism of $\Oc_X$-modules between two vector bundles. If $\E$ is $(\varphi,D)$-ample, then $\E^\prime$ is also $(\varphi,D)$-ample.
\end{proposition}
\begin{proof}
Let $r_0 \geq 1$ be an integer such that for all $r \geq r_0$, the bundle $\E^{(p^r)}(-D)$ is ample. For all $r\geq r_0$, the surjective morphism $\E \rightarrow \E^\prime$ induces a surjection
\begin{equation*}
\E^{(p^r)}(-D) \rightarrow {(\E^\prime)}^{(p^r)}(-D)
\end{equation*}
and we conclude with the assertion $(3)$ of proposition \ref{prop7}.
\end{proof}
\begin{proposition}\label{prop17}
The tensor product of $(\varphi,D)$-ample vector bundles is $(\varphi,D)$-ample.
\end{proposition}
\begin{proof}
Let $r_0 \geq 1$ be an integer such that for all $r \geq r_0$, the bundles $\E^{(p^r)}(-D)$ and ${(\E^\prime)}^{(p^r)}(-D)$ are ample. For all $r\geq r_0$, we have
\begin{equation*}
(\E \otimes \E^{\prime})^{(p^r)}(-2D) = \E^{(p^r)}(-D) \otimes {(\E^\prime)}^{(p^r)}(-D)
\end{equation*}
which is ample by assertion $(5)$ of proposition \ref{prop7}. It shows that $\E\otimes\E^{\prime}$ is $(\varphi,2D)$-ample and we conclude with the proposition \ref{nD}.
\end{proof}
\begin{proposition}\label{squarephiD}
If $\E$ is a vector bundle such that $\E^{\otimes n}$ is $(\varphi,D)$-ample for some $n\geq 1$, then $\E$ is also $(\varphi,D)$-ample.
\end{proposition}
\begin{proof}
Assume that $\E^{\otimes n}$ is $(\varphi,D)$-ample. By proposition \ref{nD}, $\E^{\otimes n}$ is also $(\varphi,nD)$-ample. Let $r_0 \geq 1$ be an integer such that for all $r \geq r_0$, the bundle ${(\E^{\otimes n})}^{(p^r)}(-nD)$ is ample. For all $r \geq r_0$, the bundle
\begin{equation*}
{(\E^{(p^r)}(-D))}^{\otimes n} = {(\E^{\otimes n})}^{(p^r)}(-nD)
\end{equation*}
is ample. Thus, the bundle $\E^{(p^r)}(-D)$ is ample for all $r \geq r_0$ by proposition \ref{squareampclassic}.
\end{proof}
\begin{proposition}\label{prop18}
Let $f : Y \rightarrow X$ be a finite morphism of projective schemes such that the pullback $f^{-1}D$ is defined as an effective Cartier divisor of $Y$\footnote{By \cite[\href{https://stacks.math.columbia.edu/tag/02OO}{Lemma 02OO}]{stacks-project}, it is is the case when $f(x) \notin D$ for any weakly associated point $x$ of $X$ or when $f$ is flat.} and $\E$ be a $(\varphi,D)$-ample bundle on $X$. Then $f^*\E$ is $(\varphi,f^{-1}D)$-ample on $Y$. If furthermore, $f$ is assumed surjective, then the converse holds.
\end{proposition}
\begin{proof}
Let $r_0 \geq 1$ be an integer such that for all $r \geq r_0$, the bundle $\E^{(p^r)}(-D)$ is ample. For all $r\geq r_0$, the bundle
\begin{equation*}
f^*(\E^{(p^r)}(-D)) = {(f^*\E)}^{(p^r)}(-f^{-1}D)
\end{equation*}
is ample by proposition \ref{prop11}. If $f$ is assumed surjective, the converse holds by proposition \ref{prop11} again.
\end{proof}
The table \ref{table_pos} summarizes the different stability properties of ampleness, $p$-ampleness, cohomological $p$-ampleness and $(\varphi,D)$-ampleness.
\begin{table}
\begin{center}
\begin{tabular}{|c|c|c|c|c|}
\hline
  & \begin{tabular}{@{}c@{}}cohomologically\\ $p$-ample \end{tabular} & $p$-ample  & Ample  &  $(\varphi,D)$-ample \\
  \hline
  \begin{tabular}{@{}c@{}}
  Stability of \\
  direct sum \end{tabular}  & \ref{prop13} &  \ref{prop12} & \ref{prop7}   & \ref{prop15} \\
  \hline
  \begin{tabular}{@{}c@{}}
  Stability of \\
  extension \end{tabular} & ? & ? & \ref{prop7}  & \ref{prop19} ($X$ regular) \\
  \hline
 \begin{tabular}{@{}c@{}}
  Stability of \\
  quotient \end{tabular}   & ? & \ref{prop12} & \ref{prop7}  & \ref{prop16} \\
  \hline
 \begin{tabular}{@{}c@{}}
  Stability of \\
  tensor product \end{tabular} & ? & \ref{prop12} & \ref{prop7}  & \ref{prop17} \\
  \hline
  \begin{tabular}{@{}c@{}}Stability of\\ tensor roots \end{tabular}  & ? & ? &\ref{squareampclassic} & \ref{squarephiD} \\
\hline
  \begin{tabular}{@{}c@{}}Stability of \\ pullback by \\ finite morphism \end{tabular}   &  \ref{prop14} & ? & \ref{prop11} & \ref{prop18} ($f^{-1}D$ defined)\\
  \hline
\begin{tabular}{@{}c@{}} Descent along \\ finite surjective \\morphism \end{tabular}   &  ?  & ?  & \ref{prop11} & \ref{prop18} \\
\hline
\end{tabular}
\end{center}
\vspace{0.2cm}
\caption{Main properties of the different positivity notions, from the strongest to the weakest one.}
\label{table_pos}
\end{table}
We explain the relationship between $(\varphi,D)$-ampleness and other positivity notion. 
\begin{proposition}
Let $\E$ be a vector bundle on $X$. Then,
\begin{equation*}
\begin{tikzcd}[column sep = small, row sep = small]
&  \E \text{ is } L\text{-big} & \\
\E \text{ is ample } \arrow[r,Rightarrow] & \E \text{ is } (\varphi,D)\text{-ample} \arrow[d,Rightarrow] \arrow[u,Rightarrow]\arrow[r,Rightarrow] & \E \text{ is } \text{nef} \\
& \E^{(p^r)} \text{ is } V\text{-big for some }r \geq 1 &
\end{tikzcd}
\end{equation*}
\end{proposition}
\begin{proof}
The first implication follows directly from \cite[Proposition 3.1]{MR289525}. Now, assume that $\E$ is $(\varphi,D)$-ample and consider the universal line bundle $\Oc(1)$ on the projective bundle $\Pb(\E)$. We have a surjective map $\pi^*\E \rightarrow \Oc(1)$ and since $(\varphi,D)$-ampleness is stable under quotient by proposition \ref{prop16}, $\Oc(1)$ is in particular nef and big by proposition \ref{prop20}. It shows that $\E$ is nef and $L$-big. Take $r \geq 1$ such that $\E^{(p^r)}(-D)$ is ample. We deduce that there is an integer $n \geq 1$ such that 
\begin{equation*}
\Sym^n(\E^{(p^r)}(-D)) \otimes \Oc_X(-1) = \Sym^n(\E^{(p^r)})(-nD) \otimes \Oc_X(-1) 
\end{equation*}
is globally generated. Since $\Sym^n(\E^{(p^r)}) \otimes \Oc_X(-1)$ can be expressed as a tensor product of a globally generated vector bundle with $\Oc_X(nD)$, it is globally generated on the complementary open subset of the support of $D$. It implies that the augmented base locus of $\E^{(p^r)}$ is not equal to $X$, i.e. that $\E^{(p^r)}$ is $V$-big.
\end{proof}

\section{Flag bundle associated to a $G$-torsor}\label{sect_flag}
\subsection{Higher direct image}
In this subsection only, $\pi : Y \rightarrow X$ is a general scheme morphism. We recall some generalities about cohomology and higher direct image.
\begin{proposition}
For any $\Oc_Y$-module $\Fc$, there is a spectral sequence starting at page 2:
\begin{equation*}
E_2^{i,j} = H^i(X,R^i\pi_*(\Fc)) \Rightarrow H^{i+j}(Y,\Fc).
\end{equation*}
\end{proposition}
\begin{proof}
See \cite[Lemma 01F2]{stacks-project}.
\end{proof}
We recall the projection formula:
\begin{proposition}
Let $\Fc$ be a $\Oc_Y$-module, $\E$ a locally free $\Oc_X$-module of finite rank and $i \geq 0$ an integer. The natural map
\begin{equation*}
R^i\pi_*\Fc \otimes_{\Oc_X} \E \rightarrow R^i\pi_*(\Fc \otimes_{\Oc_Y} \pi^*\E)
\end{equation*}
is an isomorphism.
\end{proposition}
\begin{proof}
See \cite[Lemma 01E8]{stacks-project}.
\end{proof}
We recall the following lemma that appears also in \cite{alexandre2022vanishing}.
\begin{lemma}\label{lem1}
Consider two Artin stacks $\mathcal{X}$ and $\mathcal{Y}$ over $k$ and a proper representable morphism $\pi : \mathcal{Y} \rightarrow \mathcal{X}$. Consider a coherent sheaf $\mathcal{F}$ over $\mathcal{Y}$ which is flat over $\mathcal{X}$ and such that for any geometric point $x : \Spec K \rightarrow \mathcal{X}$ fitting in the following cartesian diagram
\begin{equation*}
\begin{tikzcd}
\mathcal{Y}_x := \mathcal{Y}\times_{\mathcal{X},x}\Spec K \arrow[d,"\pi_x"] \arrow[r,"i"] & \mathcal{Y} \arrow[d,"\pi"] \\
\Spec K \arrow[r,"x"] & \mathcal{X}
\end{tikzcd}
\end{equation*}
the complex $R{(\pi_x)}_*\mathcal{F}_{|\mathcal{Y}_x}$ is concentrated in degree $0$. Then, the complex $R\pi_*\mathcal{F}$ is also concentrated in degree $0$.
\end{lemma}
\begin{proof}
See \cite[Lemma 3.19]{alexandre2022vanishing}.
\end{proof}

\subsection{$G$-torsors}
In this subsection, $k$ can be an algebraically closed field of any characteristic, $G$ is a connected split reductive group over $k$, $P \subset G$ is a parabolic subgroup and $X$ is a $k$-scheme. If $Y$ is a scheme, we denote $\Mod(\Oc_Y)$ the abelian category of $\Oc_Y$-module on $Y$ and $\Loc(\Oc_Y) \subset \Mod(\Oc_Y)$ the fully faithful additive subcategory of locally free $\Oc_Y$-module of finite rank.
\begin{definition}\label{def1}
Let $E$ be a $G$-torsor over $X$. We define the flag bundle of type $P$ of $E$ to be the scheme $\Fc_P(E)$ over $X$ that represents the functor whose $S$-points are $P$-reduction of $E \times_X S$ over $S$.
\end{definition}
\begin{definition}
Let $V$ be an algebraic representation of $G$ and $E$ a $G$-torsor over $X$. We define the contracted product of $E$ and $V$ over $G$ to be the representable quotient $X$-scheme 
\begin{equation*}
V\times^G E := \underline{V} \times_k E / G
\end{equation*}
where $\underline{V}$ is the $k$-vector space scheme associated to $V$ and $G$ acts on functorial points by $g(v,e) = (gv,ge)$. Note that the structure of $k$-vector space on $V$ endows $V\times^G E$ with a structure of vector bundle of rank $\dim_k V$ over $X$.
\end{definition}
\begin{definition}\label{def_W}
Let $E$ be a $G$-torsor over $X$. We define a functor
\begin{equation*}
\W : \Rep(G) \rightarrow \Loc(\Oc_X)
\end{equation*}
through the formula 
\begin{equation*}
\W(V) = V \times^G E
\end{equation*}
where $V$ is an algebraic representation of $G$.
\end{definition}
\begin{definition}\label{def_L}
Let $E$ be a $G$-torsor over $X$ and $\pi : \Fc_P(E) \rightarrow X$ the flag bundle of type $P$ of $E$. We define a functor
\begin{equation*}
\Lc : \Rep(P) \rightarrow \Loc(\Oc_{\Fc_P(E)})
\end{equation*}
through the formula 
\begin{equation*}
\Lc(V) = V \times^P H
\end{equation*}
where $V$ is an algebraic representation of $P$ and $H$ is the universal $P$-torsor on $\Fc_P(E)$.
\end{definition}
\begin{rmrk}
If $\lambda \in X^*(P)$ is a character of $P$, we simply write $\Lc_{\lambda}$ for the associated line bundle on $\Fc_P(E)$. We also write $\W_{\lambda}$ for the vector bundle associated to the $G$-representation $H^0(G/P,\Lc_{\lambda}) = \nabla(\lambda)$. We simply denote $\St_r$ the image of the Steinberg module by $\W$.
\end{rmrk}
\begin{proposition}
The functor $\W$ and $\Lc$ are monoidal and exact.
\end{proposition}
\begin{proof}
This is a general result on associated sheaves \cite[Part.\ 1, Chap. \ 5]{MR2015057}.
\end{proof}

\begin{proposition}\label{prop5}
Let $E$ be a $G$-torsor over $X$. Then, the following diagram
\begin{equation*}
\begin{tikzcd}
\Rep(P) \arrow[r,"\Lc"] \arrow[d,"\Ind_P^G"] & \Loc(\Oc_{\Fc_P(E)}) \arrow[d,"\pi_*"] \\
\Rep(G) \arrow[r,"\W"] \arrow[d,"\Res^G_P"] & \Loc(\Oc_X) \arrow[d,"\pi^*"] \\
\Rep(P) \arrow[r,"\Lc"] & \Loc(\Oc_{\Fc_P(E)}) 
\end{tikzcd}
\end{equation*}
commutes where $\Ind_P^G : \Rep_k(P) \rightarrow \Rep_k(G)$ and $\Res^G_P : \Rep_k(G) \rightarrow \Rep_k(P)$ are the induction and restriction functors. Moreover, if $\lambda$ is a dominant character of $P$, then $R\pi_*\mathcal{L}_{\lambda}$ is isomorphic to $\mathcal{W}_{\lambda}$ concentrated in degree 0. 
\end{proposition}
\begin{proof}
The commutativity of the lower square follows directly from the definitions. We focus on the commutativity of the upper square. Consider a representation $V$ of $P$. We have a cartesian diagram
\begin{equation*}
\begin{tikzcd}
\Fc_P(E) \arrow[d,"\pi"] \arrow[r,"\zeta_P"] & \lfloor P \backslash * \rfloor \arrow[d,"\tilde{\pi}"] \\
X \arrow[r,"\zeta"] & \lfloor G \backslash * \rfloor
\end{tikzcd}
\end{equation*}
where the map $\zeta$ is induced by $E$, the map $\zeta_P$ is induced by the universal $P$-reduction of $E$ on $\Fc_P(E)$ and the vertical arrow $\tilde{\pi}$ between the classifying stacks is induced by the inclusion $P \subset G$. Denote $\tilde{\mathcal{L}}(V)$ the vector bundle on the classifying stack of $P$ associated to $V$ and $\tilde{\mathcal{W}}(V)$ the vector bundle on the classifying stack of $G$ associated to the $G$-module $\Ind_P^G (V)$. It follows directly from the definitions that
\begin{equation*}
\left\{
\begin{aligned}
&\tilde{\pi}_* \tilde{\mathcal{L}}(V) = \tilde{\mathcal{W}}(V) \\
&\zeta_{P}^*\tilde{\mathcal{L}}(V) = \mathcal{L}(V) \\
&\zeta^*\tilde{\mathcal{W}}(V) = \mathcal{W}(V)
\end{aligned}
\right.
\end{equation*}
as sheaves on the stack $ \lfloor G \backslash * \rfloor$. Since $\zeta$ is a flat morphism of algebraic stacks, the base change theorem in the derived category of quasi-coherent sheaves over $X$ tells us that the map
\begin{equation}\label{eq1}
\zeta^* \circ {R}\tilde{\pi}_* \tilde{\mathcal{L}}(V)  \DistTo {R}\pi_*\circ \zeta_{P}^*\tilde{\mathcal{L}}(V)
\end{equation}
is an isomorphism. Taking global sections in equation \ref{eq1} yields an isomorphism
\begin{equation*}
\begin{tikzcd}
\W(V) = \zeta^*\tilde{\pi}_* \tilde{\mathcal{L}}(V) \arrow[r,"\simeq"] & \pi_* \mathcal{L}(V)
\end{tikzcd}
\end{equation*}
over $X$. Assume now that $\lambda$ is $P$-dominant. By Kempf's vanishing theorem from proposition \ref{prop4} combined with lemma \ref{lem1}, we deduce
\begin{equation*}
\left\{
\begin{aligned}
&{R}\pi_* \mathcal{L}_{\lambda} = \pi_* \mathcal{L}_{\lambda} \\
&{R}\tilde{\pi}_* \mathcal{L}_{\lambda} = \tilde{\pi}_*  \mathcal{L}_{\lambda}
\end{aligned}
\right.
\end{equation*}
and we get an isomorphism
\begin{equation*}
R\pi_* \mathcal{L}_{\lambda} \simeq \mathcal{W}_{\lambda} [0].
\end{equation*}
\end{proof}

\begin{proposition}
Let $\lambda$ be a character. For all $r\geq 1$, we have isomorphisms
\begin{equation*}
\pi_*(\Lc_{p^r(\lambda+\rho)-\rho}) = \St_r \otimes \mathcal{W}_{\lambda}^{(p^r)}.
\end{equation*}
\end{proposition}
\begin{proof}
This is a direct consequence of proposition \ref{prop1} and \ref{prop5}.
\end{proof}
\begin{proposition}
Let $P_I$ be a standard parabolic subgroup of $G$ of type $I$. Let $E$ be a $G$-torsor over $X$ and $\pi : \Fc_{P_I}(E) \rightarrow X$ the flag bundle ot type $P_I$ of $X$. We have an isomorphism
\begin{equation*}
\Omega^{\text{top}}_{\Fc_{P_I}(E)/X} \simeq \Lc_{-2\rho_I}
\end{equation*}
where top denotes the relative dimension of $\pi$ and $\rho_I = \frac{1}{2} \sum_{\alpha \in \Phi^+_I} \alpha$.
\end{proposition}
\begin{proof}
From the cartesian diagram
\begin{equation*}
\begin{tikzcd}
\Fc_{P_I}(E) \arrow[d,"\pi"] \arrow[r,"\tilde{\zeta}"] & \lfloor P_I \backslash * \rfloor \arrow[d,"\tilde{\pi}"] \\
X \arrow[r,"\zeta"] & \lfloor G \backslash * \rfloor
\end{tikzcd}
\end{equation*}
we deduce an isomorphism $\tilde{\zeta}^*\Omega^1_{\tilde{\pi}} = \Omega^1_{\pi}$. We know that
\begin{equation*}
\Omega^{1}_{\tilde{\pi}} \simeq \Lc(\Lie(G)/\Lie(P_I)^{\vee}),
\end{equation*}
hence
\begin{equation*}
\Omega^{top}_{\tilde{\pi}} \simeq \Lc(\Lambda^{top}\Lie(G)/\Lie(P_I)^{\vee}).
\end{equation*}
The weights of the $T$-action on $\Lie(G)/\Lie(P_I)$ are the roots $-\Phi^+_I$, so $\Lambda^{top}\Lie(G)/\Lie(P_I)$ is a one-dimensional module of weight $-2\rho_I$ and by taking the linear dual, we get an isomorphism $\Omega^{top}_{\pi} =\tilde{\zeta}^*\Omega^{top}_{\tilde{\pi}} = \Lc_{-2\rho_I}$.
\end{proof}

\section{Pushforward of positive line bundles}\label{sect_push}
Recall that $k$ is an algebraically closed field of characteristic $p$. Let $X$ be a projective scheme over $k$ and $D$ an effective Cartier divisor on $X$. Let $E$ be a $G$-torsor and write $\pi : Y \rightarrow X$ for the flag bundle of type $B$ of $E$ as defined in definition \ref{def1}. We write also $D$ for the Cartier divisor $\pi^{-1}(D)$ on $Y$. Recall that we have fixed an ample line bundle $\Oc_X(1)$ on $X$ and that we write $\Fc(m)$ instead of $\Fc \otimes \Oc_X(1)^{\otimes m}$ for any coherent sheaf $\Fc$ on $X$ and integer $m$. We start this section with some preliminary results. 
\begin{lemma}\label{classifying_trick}
Consider a finite surjective morphism $g : G^\prime \rightarrow G$ of algebraic groups with central kernel. Then there exists a projective scheme $X^\prime$ and a finite surjective morphism $f : X^\prime \rightarrow X$ such that the pullback of the $G$-torsor $f^*E$ reduces to a $G^\prime$-torsor on $X^\prime$.
\end{lemma}
\begin{proof}
Let us denote $BG = \lfloor \Spec k / G \rfloor$ and $BG^\prime = \lfloor \Spec k / G^\prime \rfloor$ the classifying stacks of $G$ and $G^\prime$. The $G$-torsor $E$ on $X$ corresponds to a map $b_E : X \rightarrow  BG$ and $g : G^\prime \rightarrow G$ induces a map $b_g : BG^\prime \rightarrow BG$ on the classifying stacks. Let $K$ denote the kernel of $g$. We consider the following cartesian product
\begin{equation*}
\begin{tikzcd}
X^\prime \arrow[rd, dashed, "h"] \arrow[rdd,bend right, dashed, "\alpha \circ h", swap] & & \\
& \mathcal{X} \arrow[d,"\alpha"] \arrow[r,"\beta"] & BG^\prime \arrow[d, "b_g"]\\
& X \arrow[r,"b_E"] & BG
\end{tikzcd}
\end{equation*}
in the category of Artin stacks over $k$ and the objective is now to prove that there exists a scheme $X^\prime$ over $k$ and a morphism $h : X^\prime \rightarrow \mathcal{X}$ such that $\alpha \circ h : X^\prime \rightarrow X$ is finite surjective. The first step is to show that $\alpha$ is quasi-finite, proper\footnote{See \cite[Sect. 10.1]{MR3495343} for a reference on properness for non-representable morphisms of stacks.} and surjective. By base change along $b_E$, it is enough to show it for $b_g$. Since $K$ is central, we have the following cartesian product
\begin{equation*}
\begin{tikzcd}
 BK \arrow[d] \arrow[r] &  BG^\prime \arrow[d, "b_g"]\\
\Spec k \arrow[r,"b_G"] &  BG
\end{tikzcd}
\end{equation*}
where $b_G$ is the classifying map of the trivial $G$-torsor on $\Spec k$. We claim the map $BK \rightarrow \Spec k$ is proper, quasi-finite and surjective. The only non-trivial part is to show that $BK \rightarrow \Spec k$ is separated, i.e. that its diagonal is proper. We have the following cartesian product
\begin{equation*}
\begin{tikzcd}
K \arrow[d] \arrow[r] & \Spec k  \arrow[d]\\
BK \arrow[r] &  BK \times_k  BK
\end{tikzcd}
\end{equation*}
and since $K$ is finite, the diagonal $BK \rightarrow  BK \times_k  BK$ is also finite by faithfully flat descent. By faithfully flat descent along $b_G$, it implies that the map $b_g$ is quasi-finite, proper and surjective. The second step is to find a finite surjective morphism $h : X^\prime \rightarrow \mathcal{X}$ approximating the Artin stack $\mathcal{X}$. By \cite[Theorem B]{MR3272071}, we have to check that the diagonal of $\alpha : \mathcal{X} \rightarrow X$ is quasi-finite and separated. By base change and faithfully flat descent, it follows from the fact that the diagonal of $BK \rightarrow \Spec k$ is finite. Combining the two steps, the composition $\alpha \circ h$ is finite surjective.
\end{proof}
We give a sufficient cohomological condition for a vector bundle to be ample. 
\begin{proposition}\label{prop2}
Let $\E$ be a vector bundle over $X$. Let $\lambda \in X^*$ be a character. If for all coherent sheaf $\Fc$, there is an integer $r_0$ such that
\begin{equation*}
H^i(X,\Fc \otimes \St_r \otimes \W_{\lambda}^{(p^r)}) = 0
\end{equation*}
for all $i>0$ and $r \geq r_0$, then $\W_{\lambda}$ is ample.
\end{proposition}
\begin{proof}
We consider a coherent sheaf $\Fc = \Oc_X(-m)$ with $m\geq 0$ and we write $$\mathcal{G}_r = \St_r\otimes \W_{\lambda}^{(p^r)} \otimes \Oc_X(-m).$$Let $x \in X$ be a closed point. From our hypothesis, there is an integer $r_0$ such that 
\begin{equation*}
H^1(X, \mathcal{G}_{r_0} \otimes \mathcal{I}_x) = 0
\end{equation*}
where $\mathcal{I}_x$ is the ideal sheaf defining the closed point $x$. From the long exact sequence of cohomology associated to the exact sequence 
\begin{equation*}
\begin{tikzcd}[column sep = small]
0 \arrow[r] & \mathcal{G}_{r_0}\otimes \mathcal{I}_x \arrow[r] & \mathcal{G}_{r_0} \arrow[r] & \mathcal{G}_{r_0} \otimes k(x) \arrow[r] & 0
\end{tikzcd}
\end{equation*}
we deduce that the map
\begin{equation*}
H^0(X,\mathcal{G}_{r_0}) \rightarrow H^0(X,\mathcal{G}_{r_0}\otimes k(x))
\end{equation*}
is surjective. In other words, $\mathcal{G}_{r_0}$ is globally generated at $x$. It implies there exists an open $U$ containing $x$ such that $\mathcal{G}_{r_0}$ is globally generated over $U$. Since $\St_{r_0}$ is self dual, there is a canonical surjective map
\begin{equation*}
\St_{r_0}^{\otimes 2} \rightarrow \Oc_X.
\end{equation*}
Since the tensor product of globally generated sheaves over $U$ is again globally generated over $U$, we deduce that 
\begin{equation*}
\mathcal{G}_{r_0}^{\otimes 2} = \St_{r_0}^{\otimes 2} \otimes {(\W_{\lambda}^{\otimes 2})}^{(p^{r_0})} \otimes \Oc_X(-2m)
\end{equation*}
is a globally generated over $U$. Since the quotient of a globally generated sheaf over $U$ is globally generated over $U$, we know that
\begin{equation*}
{(\W_{\lambda}^{\otimes 2})}^{(p^{r_0})}(-2m)
\end{equation*}
is globally generated sheaf over $U$. Now, let $r \geq r_0$ be an integer. From the equality
\begin{equation*}
{(\W_{\lambda}^{\otimes 2})}^{(p^r)}(-2p^{r-r_0}m) = {\left({(\W_{\lambda}^{\otimes 2})}^{(p^{r_0})}(-2m)\right)}^{(p^{r-r_0})}
\end{equation*} 
we deduce that ${(\W_{\lambda}^{\otimes 2})}^{(p^r)}(-2p^{r-r_0}m)$ is globally generated over $U$. Now take $r_1$ large enough to have $\Oc_X((2p^{r_1-r_0}-1)m)$ globally generated. We deduce that
\begin{equation*}
{(\W_{\lambda}^{\otimes 2})}^{(p^r)}(-2p^{r-r_0}m) \otimes \Oc_X((2p^{r-r_0}-1)m) = {(\W_{\lambda}^{\otimes 2})}^{(p^{r})}(-m)
\end{equation*}
is globally generated over $U$ for all $r \geq r_1$. Since $X$ is quasi-compact, we can find an integer $r_2 \geq r_1$ such that ${(\W_{\lambda}^{\otimes 2})}^{(p^{r})}(-m)$ is globally generated over $X$ for all $r\geq r_2$. We use the proposition \ref{prop3} to deduce that ${\W_{\lambda}^{\otimes 2}}$ is $p$-ample. By proposition \ref{pamptoamp} and \ref{squareampclassic}, it implies that $\W_{\lambda}$ is ample.
\end{proof}
We give a sufficient cohomological condition for a vector bundle to be $(\varphi,D)$-ample. 
\begin{proposition}\label{technicalD}
Let $\E$ be a vector bundle over $X$. Let $\lambda \in X^*(T)$ be a character. If there exists an effective Cartier divisor $D$ and an integer $r_ 0 \geq 1$ such that for all $r \geq r_0$ and all coherent sheaf $\Fc$ over $X$, there is an integer $r_1 \geq 1$ such that for all $r^\prime \geq r_1$ and $i>0$, we have
\begin{equation*}
H^i(X,\Fc \otimes \St_{r+r^\prime} \otimes \W_{\lambda}^{(p^{r+r^\prime})}(-p^{r^\prime}D)) = 0,
\end{equation*}
then $\W_{\lambda}$ is $(\varphi,D)$-ample.
\end{proposition}
\begin{proof}
Using proposition \ref{prop3} and proposition \ref{pamptoamp}, it is sufficient to see that there exists integers $n \geq 1$ and $r_ 0 \geq 1$ such that for all $r \geq r_0$, the bundle $\mathcal{W}_{\lambda}^{(p^r)}(-nD)$ is $p$-ample. In other words, it is sufficient to see that there exists integers $n \geq 1$ and $r_ 0 \geq 1$ such that for all $r \geq r_0$ and all $m\geq 1$, there exists $r_1 \geq 1$, such that for all $r^\prime \geq r_1$, the bundle $$\W_{\lambda}^{(p^{r+r^\prime})}(-p^{r^\prime}nD)(-m)$$ is globally generated over $X$. By hypothesis, we have a Cartier divisor $D$ and $r_ 0 \geq 1$ an integer. Consider two integers $r \geq r_0$, $m\geq 0$ and write $$\mathcal{G}_{r^\prime} =  \St_{r+r^\prime} \otimes \W_{\lambda}^{(p^{r+r^\prime})}(-p^{r^\prime}D) \otimes \Oc_X(-m).$$Let $x \in X$ be a closed point. By hypothesis, we have an integer $r_1 \geq 1$ such that 
\begin{equation*}
H^1(X, \mathcal{G}_{r_1} \otimes \mathcal{I}_x) = 0
\end{equation*}
where $\mathcal{I}_x$ is the ideal sheaf defining the closed point $x$. From the long exact sequence of cohomology associated to the exact sequence 
\begin{equation*}
\begin{tikzcd}[column sep = small]
0 \arrow[r] & \mathcal{G}_{r_1}\otimes \mathcal{I}_x \arrow[r] & \mathcal{G}_{r_1} \arrow[r] & \mathcal{G}_{r_1} \otimes k(x) \arrow[r] & 0
\end{tikzcd}
\end{equation*}
we deduce that the map
\begin{equation*}
H^0(X,\mathcal{G}_{r_1}) \rightarrow H^0(X,\mathcal{G}_{r_1}\otimes k(x))
\end{equation*}
is surjective. In other words, $\mathcal{G}_{r_1}$ is globally generated at $x$. It implies there exists an open $U$ containing $x$ such that $\mathcal{G}_{r_1}$ is globally generated over $U$. Since the Steinberg module is self dual, there is a canonical surjective map
\begin{equation*}
\St_{r+r_1}^{\otimes 2} \rightarrow \Oc_X.
\end{equation*}
Since the tensor product of globally generated sheaves over $U$ is again globally generated over $U$, we deduce that 
\begin{equation*}
\mathcal{G}_{r_1}^{\otimes 2} = \St_{r+r_1}^{\otimes 2} \otimes {(\W_{\lambda}^{\otimes 2})}^{(p^{r+r_1})}(-2p^{r_1}D) \otimes \Oc_X(-2m)
\end{equation*}
is globally generated over $U$. Since the quotient of a globally generated sheaf over $U$ is globally generated over $U$, we know that
\begin{equation*}
 {(\W_{\lambda}^{\otimes 2})}^{(p^{r+r_1})}(-2p^{r_1}D) \otimes \Oc_X(-2m)
\end{equation*}
is a globally generated sheaf over $U$. From the equality
\begin{equation*}
{\left({(\mathcal{W}_{\lambda}^{\otimes 2})}^{(p^{r+r_1})}(-2p^{r_1}D)(-2m)\right)}^{(p^{r^\prime-r_1})} = {(\mathcal{W}_{\lambda}^{\otimes 2})}^{(p^{r+r^\prime})}(-2p^{r^\prime}D)(-2p^{r^{\prime}-r_1}m)
\end{equation*} 
we deduce that ${(\mathcal{W}_{\lambda}^{\otimes 2})}^{(p^{r+r^\prime})}(-2p^{r^\prime}D)(-2p^{r^{\prime}-r_1}m)$ is globally generated over $U$. Now take $r_2 \geq r_1$ large enough to have $\Oc_X((2p^{r^\prime-r_1}-1)m)$ globally generated for all $r^\prime \geq r_2$. We deduce that
\begin{multline*}
{(\mathcal{W}_{\lambda}^{\otimes 2})}^{(p^{r+r^\prime})}(-2p^{r^\prime}D)(-2p^{r^{\prime}-r_1}m) \otimes \Oc_X((2p^{r^\prime-r_1}-1)m) \\= {(\mathcal{W}_{\lambda}^{\otimes 2})}^{(p^{r+r^\prime})}(-2p^{r^\prime}D)(-m)
\end{multline*}
is globally generated over $U$ for all $r^\prime \geq r_2$. Since $X$ is quasi-compact, we can find an integer $r_3 \geq r_2$ such that ${(\mathcal{W}_{\lambda}^{\otimes 2})}^{(p^{r+r^\prime})}(-2p^{r^\prime}D)(-m)$ is globally generated over $X$ for all $r^\prime \geq r_3$. In conclusion, we have proven that $\mathcal{W}_{\lambda}^{\otimes 2}$ is $(\varphi,2D)$-ample, which is equivalent to $(\varphi,D)$-ample by proposition \ref{nD}. Then, we use the proposition \ref{squarephiD} to deduce that $\mathcal{W}_{\lambda}$ is $(\varphi,D)$-ample.
\end{proof}
\begin{theorem}\label{th1}
Let $\lambda \in X^*(T)$ be character. If $\mathcal{L}_{2\lambda+2\rho}$ is ample on $Y$, then $\W_{\lambda}$ is ample on $X$.
\end{theorem} 
\begin{proof}
Consider a semi-simple cover of $G$ (the existence is proved in \cite{MR2015057}), i.e. a finite surjective morphism $h : G^\prime \rightarrow G$ of reductive groups with central kernel such that $G^\prime = G_{\text{sc}} \times T_1$ is a product of a semisimple simply-connected group $G_{\text{sc}}$ with a torus $T_1$. Since ampleness can be tested after a pullback by a finite surjective morphism by proposition \ref{prop11}, we can use lemma \ref{classifying_trick} to assume that $\rho$ is a genuine character. Assume that $\Lc_{\lambda+\rho}$ is ample and consider a coherent sheaf $\Fc$ on $X$. We have a Leray spectral sequence starting at second page
\begin{equation*}
E_2^{i,j} = H^i(X,\Fc \otimes R^j\pi_*( \Lc_{\lambda+\rho}^{\otimes p^r} \otimes \Lc_{-\rho} ) ) \Rightarrow H^{i+j}(Y,\pi^*\Fc \otimes \Lc_{\lambda+\rho}^{\otimes p^r} \otimes \Lc_{-\rho} )
\end{equation*}
Since $\Lc_{\lambda+\rho}$ is ample on $Y$, it is also $\pi$-ample and we have 
\begin{equation*}
R^j\pi_*( \Lc_{\lambda+\rho}^{\otimes p^r} \otimes \Lc_{-\rho} ) = 0
\end{equation*}
for all $j>0$ and $r$ large enough. We deduce that the spectral sequence degenerates at page $2$ and we get isomorphisms:
\begin{equation*}
H^i(X,\Fc \otimes \pi_*( \Lc_{\lambda+\rho}^{\otimes p^r} \otimes \Lc_{-\rho} ) ) = H^{i}(Y,\pi^*\Fc \otimes \Lc_{\lambda+\rho}^{\otimes p^r} \otimes \Lc_{-\rho} )
\end{equation*}
for all $i\geq 0$ and $r$ large enough. Moreover, since $\Lc_{\lambda+\rho}$ is ample, the right hand side vanishes for $i>0$ and $r$ large enough. From the proposition \ref{prop1}, we know that 
\begin{equation*}
\pi_*(\Lc_{p^r(\lambda+\rho)-\rho}) = \St_r \otimes \mathcal{W}_{\lambda}^{(p^r)}
\end{equation*} 
and from proposition \ref{prop2}, we deduce that $\mathcal{W}_{\lambda}$ is ample.
\end{proof}
\begin{theorem}\label{th2}
Let $\lambda \in X^*(T)$ be character. If $\mathcal{L}_{2\lambda+2\rho}$ is $D$-ample over $Y$, then $\W_{\lambda}$ is $(\varphi,D)$-ample on $X$.
\end{theorem}
\begin{proof}
Since $(\varphi,D)$-ampleness can be tested after a pullback by a finite surjective morphism by proposition \ref{prop18}, we use the same trick as in theorem \ref{th1} to assume that $\rho$ is a genuine character. Consider $r_0\geq 1$ large enough such that $\Lc_{\lambda+\rho}^{\otimes p^{r}}(-D)$ is ample for all $r\geq r_0$. Let $\Fc$ be a coherent sheaf on $X$ and $r \geq r_0$ integer. For all integer $r^\prime \geq 1$, we have a Leray spectral sequence starting at second page
\begin{multline*}
E_2^{i,j} = H^i(X,\Fc \otimes R^j\pi_*(\Lc_{\lambda+\rho}^{\otimes p^{r+r^\prime}}(-p^{r^\prime}D) \otimes \Lc_{-\rho}))\\ \Rightarrow H^{i+j}(Y,\pi^*\Fc \otimes \Lc_{\lambda+\rho}^{\otimes p^{r+r^\prime}}(-p^{r^\prime}D) \otimes \Lc_{-\rho}).
\end{multline*}
Since $\Lc_{\lambda+\rho}^{\otimes p^{r}}(-D)$ is $\pi$-ample, there is a $r_1 \geq 1$ large enough such that
\begin{equation*}
R^j\pi_*(\Lc_{\lambda+\rho}^{\otimes p^{r+r^\prime}}(-p^{r^\prime}D) \otimes \Lc_{-\rho}) = R^j\pi_*((\Lc_{\lambda+\rho}^{\otimes p^{r}}(-D))^{\otimes p^{r^\prime}} \otimes \Lc_{-\rho}) = 0
\end{equation*}
for all $j>0$ and $r^\prime \geq r_1$. We deduce that the spectral sequence degenerates at page $2$ and we get isomorphisms:
\begin{equation*}
H^i(X,\Fc \otimes \pi_*(\Lc_{\lambda+\rho}^{\otimes p^{r+r^\prime}}(-p^{r^\prime}D) \otimes \Lc_{-\rho})) = H^{i}(Y,\pi^*\Fc \otimes \Lc_{\lambda+\rho}^{\otimes p^{r+r^\prime}}(-p^{r^\prime}D) \otimes \Lc_{-\rho})
\end{equation*}
for all $i\geq 0$ and $r^\prime \geq r_1$. Since $\Lc_{\lambda+\rho}^{\otimes p^{r}}(-D)$ is ample on $Y$, there exists $r_2 \geq r_1$ such that we have 
\begin{equation*}
H^{i}(Y,\pi^*\Fc \otimes (\Lc_{\lambda+\rho}^{\otimes p^{r}}(-D))^{\otimes p^{r^\prime}}\otimes \Lc_{-\rho}) = 0
\end{equation*} 
for all $i>0$ and $r^\prime \geq r_2$. From proposition \ref{prop1}, we know that 
\begin{equation*}
\pi_*(\Lc_{p^{r+r^\prime}(\lambda+\rho)-\rho}(-p^{r^\prime}D)) = \St_{r+r^\prime} \otimes \mathcal{W}_{\lambda}^{(p^{r+r^\prime})}(-p^{r^\prime}D)
\end{equation*}
which implies that 
\begin{equation*}
H^i(X,\Fc \otimes \St_{r+r^\prime} \otimes \mathcal{W}_{\lambda}^{(p^{r+r^\prime})}(-p^{r^\prime}D)) = 0
\end{equation*}
for all $r^\prime \geq r_2$. We deduce with the technical proposition \ref{technicalD} that $\W_{\lambda}$ is $(\varphi,D)$-ample.
\end{proof}
\section{Positivity of automorphic vector bundles on the Siegel variety}\label{sect_auto}
In this section, we prove that certain automorphic bundles on the Siegel modular variety are $(\varphi,D)$-ample for some effective Cartier divisor $D$.
\subsection{Recollection on Siegel modular varieties}
We start by recalling some well-known results from \cite{MR1083353} on Siegel modular varieties and their toroidal compactifications. We denote $\Sch_{R}$ the category of schemes over a ring $R$.
\begin{definition}\label{def7}
Let $V$ be the $\mathbb{Z}$-module $\mathbb{Z}^{2g}$ endowed with the standard non-degenerate symplectic pairing
\begin{equation*}
\begin{tikzcd}[row sep =tiny]
\psi : V \times V \arrow[r]& \mathbb{Z} \\
(x,y) \arrow[r,mapsto] & {}^txJy
\end{tikzcd}
\end{equation*}
where
\begin{equation*}
J = \begin{pmatrix}
0 & I_g \\
-I_g & 0 
\end{pmatrix}.
\end{equation*}
\vspace*{0.0cm}

We denote $\Sp_{2g}$ the algebraic group over $\Z$ of $2g \times 2g$ matrices $M$ that preserve the symplectic pairing $\psi$, i.e. such that
\begin{equation*}
{}^tMJM = J.
\end{equation*}
\end{definition}
\begin{definition}[\cite{MR1083353}]
Let $N$ be a positive integer such that $p \nmid N$. Recall that $k$ is an algebraically closed field of characteristic $p$. Consider the fibered category in groupoids $\mathcal{A}_{g,N}$ on $\Sch_{k}$ whose $S$-points are groupoids with
\begin{itemize}
\item Objects: $(A,\lambda,\psi_N)$ where $A \rightarrow S$ is abelian scheme over $S$ of relative dimension $g$, $\lambda : A \rightarrow A^{\vee}$ is a principal polarization and
\begin{equation*}
\psi_N : A[N] \DistTo \underline{\left(\Z/N\Z\right)}_S^2
\end{equation*}
is a basis over $S$ of the $N$-torsion of $A$.
\item Morphisms: A morphism $(A,\lambda,\psi_N) \rightarrow (A^\prime,\lambda^\prime,\psi^\prime_N)$ is a scheme morphism $\alpha : A \rightarrow A^\prime$ over $S$ such that the diagram
\begin{equation*}
\begin{tikzcd}
A \arrow[r,"\alpha"] \arrow[d,"\lambda"] & A^{\prime} \arrow[d,"\lambda^\prime"] \\
A^{\vee} & {A^\prime}^{\vee} \arrow[l,"\alpha^{\vee}"]
\end{tikzcd}
\end{equation*}
is commutative and the pullback of $\psi_N$ by $\alpha$ is $\psi^\prime_N$.
\end{itemize}
\end{definition}
\begin{proposition}[\cite{MR1083353}]\label{prop1_sieg}
For any integer $N \geq 3$ such that $p \nmid N$, the fibered category in groupoids $\mathcal{A}_{g,N}$ is representable by a smooth integral quasi-projective scheme over $k$.
\end{proposition}
\begin{notation}\label{not1}
We denote $G$ the base change of the algebraic group $\Sp_{2g}$ over $k$. We fix a genus $g \geq 1$ and a level $N \geq 3$ such that $p \nmid N$. We denote simply $\Sh$ the Siegel modular variety $\mathcal{A}_{g,N}$. Let $\mu$ be the following minuscule cocharacter of $G$
\begin{equation*}
\begin{tikzcd}[row sep =tiny]
\mu : \mathbb{G}_m \arrow[r] & G \\
z \arrow[r,mapsto] &\begin{pmatrix}
zI_g & 0 \\ 0 & z^{-1}I_g
\end{pmatrix}.
\end{tikzcd}
\end{equation*}
We denote $P^+ := P_{\mu}$ and $P:=P_{-\mu}$ the associated opposite parabolic subgroups with common Levi subgroup $L = \GL_g$ over $k$. We denote $B \subset P$ the Borel of upper triangular matrices in $G = \Sp_{2g}$ over $k$. We denote $\Phi_L$ (resp. $\Phi_L^+$) the corresponding roots of $L$ (resp. positive roots of $L$).
\end{notation}
\begin{definition}[\cite{MR1083353}]
As a fine moduli space, the Siegel variety $\Sh$ is endowed with a universal principally polarized abelian scheme of relative dimension $g$
\begin{equation*}
\begin{tikzcd}
A \arrow[r, "f"] & \Sh\arrow[l, "e", bend left]
\end{tikzcd}
\end{equation*}
where $e : \Sh \rightarrow A$ is the neutral section. Recall the following associated objects on $\Sh$.
\begin{enumerate}
\item We denote $\mathcal{H}^1_{\dR} := R^1f_*(\Omega^{\bullet}_{A/\Sh})$ the de Rham cohomology vector bundle of rank $2g$ over $\Sh$.
\item We denote $\Omega =e^*\Omega^1_{A/\Sh}$ the Hodge vector bundle of rank $g$ over $\Sh$.
\end{enumerate}
Note that the Weil paring and the principal polarization on the universal abelian scheme $f : A \rightarrow \Sh$ induce a symplectic pairing of the same type as $\psi$ on $\mathcal{H}^1_{\dR}$. In other words, the de Rham cohomology is equivalent to the data of a $G$-torsor on $\Sh$.
\end{definition}
\vspace*{0.2cm}
\begin{proposition}[\cite{MR894379}]
The Hodge-de Rham spectral sequence 
\begin{equation*}
E_1^{i,j} = R^jf_*(\Omega^i_{A/\Sh}) \Rightarrow R^{i+j}f_*(\Omega^{\bullet}_{A/\Sh})
\end{equation*}
degenerates at page $1$ which proves the existence of the Hodge-de Rham filtration
\begin{equation*}
\begin{tikzcd}
0 \arrow[r] & \Omega \arrow[r] & \mathcal{H}^1_{\dR} \arrow[r] & R^1f_*\mathcal{O}_{A} \arrow[r] & 0.
\end{tikzcd}
\end{equation*}
Moreover, the Hodge bundle $\Omega$ is totally isotropic for the symplectic pairing on $\mathcal{H}^1_{\dR}$ which implies that the Hodge-de Rham filtration is equivalent to the data of a $P$-reduction of the $G$-torsor $\mathcal{H}^1_{\dR}$ on the Siegel variety. 
\end{proposition}
\vspace*{0.2cm}
In the next definition, we recall the main properties of toroidal compactifications of Siegel varieties.
\begin{definition}[{\cite[Chapter 4]{MR1083353}\cite[Th. 2.15]{MR2968629}}]\label{prop5_sieg}
Let $C$ denote the cone of all positive semi-definite symmetric bilinear forms on $X^* \otimes_{\mathbb{Z}} \mathbb{R}$ with radicals defined over $\mathbb{Q}$. Following the definitions \cite[Chapter 4, Definition 2.2/2.3]{MR1083353}, we consider a smooth $GL(X^*)$-admissible decomposition $\Sigma = \{ \sigma_{\alpha} \}_{\alpha}$ in polyhedral cones of $C$. Following the definition \cite[Chapter 4, Definition 2.4]{MR1083353}, we assume furthermore that $\Sigma$ admits a $GL(X^*(T))$-equivariant polarization function. The existence of a polyhedral cone decomposition $\Sigma$ satisfying these assumptions is ensured by  \cite{MR2590897} and \cite{MR0335518}. Denote $\Sh^{\tor}$ the toroidal compactification of the Siegel variety associated to $\Sigma$. It follows from the assumptions on $\Sigma$ that $\Sh^{\tor}$ is a smooth projective scheme over $k$ satisfying the following assertions.

\begin{enumerate}
\item The boundary $D_{\red} = \Sh^{\tor} - \Sh$ with its reduced structure is an effective Cartier divisor with normal crossings.
\item The universal abelian scheme $f : A \rightarrow \Sh$ extends to a semi-abelian scheme $f^{\tor} : A^{\tor} \rightarrow \Sh^{\tor}$.
\item The sheaf $\Omega^{\tor} := e^*\Omega^1_{A^{\tor}/\Sh^{\tor,\Sigma}}$ is a vector bundle of rank $g$ that extends the Hodge bundle $\Omega$ to $\Sh^{\tor,\Sigma}$.
\item By \cite[Chapter 4]{MR1083353} or \cite[Th. 2.15, (2)]{MR2968629} there exists a log-smooth projective compactification $\bar{f}^{\tor} : \bar{A}^{\tor} \rightarrow \Sh^{\tor}$ of the semi-abelian scheme $f^{\tor} : A^{\tor} \rightarrow \Sh^{\tor}$ and we denote again $D_{\red}$ the divisor with normal crossings $\bar{A}^{\tor} - A$.
\item By \cite[Chapter 4]{MR1083353} or \cite[Th. 2.15, (3)]{MR2968629}, the log-de Rham cohomology
\begin{equation*}
\mathcal{H}^1_{\log-\dR} := R^1{(\bar{f}^{\tor})}_*\bar{\Omega}^{\bullet}_{\bar{A}^{\tor}/\Sh^{\tor}}
\end{equation*}
where $\bar{\Omega}^{\bullet}_{\bar{A}^{\tor}/\Sh^{\tor}}$ is the complex of log-differentials
\begin{equation*}
\begin{aligned}
\bar{\Omega}^i_{\bar{A}^{\tor}/\Sh^{\tor}} &= \Lambda^i\bar{\Omega}^1_{\bar{A}^{\tor}/\Sh^{\tor}} \\
&=\Lambda^i\Omega^1_{\bar{A}^{\tor}}(\log D_{\red}) /  {(\bar{f}^{\tor})}^{*}  \Omega^1_{\Sh^{\tor}}(\log D_{\red})
\end{aligned}
\end{equation*}
is a $\Sp_{2g}$-torsor that extends the de Rham cohomology $\mathcal{H}^1_{\dR}$ to $\Sh^{\tor}$.
\item The logarithmic Hodge-de Rham spectral
sequence 
\begin{equation*}
E^{i,j}_1 = R^j{(\bar{f}^{\tor})}_*\bar{\Omega}^{i}_{\bar{A}^{\tor}/\Sh^{\tor}} \Rightarrow \mathcal{H}^i_{\log-\dR} := R^i{(\bar{f}^{\tor})}_*\bar{\Omega}^{\bullet}_{\bar{A}^{\tor}/\Sh^{\tor}}
\end{equation*}
degenerates at page $1$, which proves the existence of a $P$-reduction of the $\Sp_{2g}$-torsor $\mathcal{H}^1_{\log-\dR}$ extending the Hodge-de Rham filtration to $\Sh^{\tor}$.
\end{enumerate}
\end{definition}
The Hodge line bundle $\omega = \det \Omega^{\tor}$ is usually not ample on the Siegel variety $\Sh^{\tor}$ but it satisfies a weaker positivity result we will explain. We recall the definition of the minimal compactification of the Siegel variety.
\begin{definition}[{\cite[Chap. V]{MR1083353}}]\label{def4_sieg}
The minimal compactification $\Sh^{\min}$  of the Siegel variety $\Sh$ is defined as the scheme 
\begin{equation*}
\Proj(\oplus_{n\geq 0} H^0(\Sh^{\tor},\omega^{\otimes n})),
\end{equation*}
where $\omega = \det \Omega^{\tor}$ is the Hodge line bundle. 
\end{definition}
\begin{proposition}[{\cite[Chap. IX, Theorem 2.1, p. 208]{MR797982}}]\label{prop_sa}
The Hodge line bundle $\omega$ is semi-ample on $\Sh^{\tor}$, i.e. there exists an integer $m \geq 1$ such that $w^{\otimes m}$ is globally generated over $\Sh^{\tor}$. In particular, the Hodge line bundle descends to an ample line bundle on the minimal compactification.
\end{proposition}
\begin{proposition}[{\cite[Chap. V, Theorem 5.8]{MR1083353}}]
The toroidal compactification $\Sh^{\tor}$ is the normalization of the blow-up of $\Sh^{\min}$ 
\begin{equation*}
\nu : \Sh^{\tor} \rightarrow \Sh^{\min}
\end{equation*}
along a coherent sheaf of ideals $\mathcal{I}$ of $\mathcal{O}_{\Sh^{\min}}$.
\end{proposition}
\vspace*{0.2cm}
In particular, the pullback $\nu^*\mathcal{I}$ is of the form $\mathcal{O}_{\Sh^{\tor}}(\shortminus D)$ where $D$ is an effective Cartier divisor whose associated reduced Cartier divisor is the boundary $D_{\red}$. It follows from the ampleness of $\omega$ on $\Sh^{\min}$ and the $\nu$-ampleness of $\mathcal{O}_{\Sh^{\tor}}(\shortminus D)$ that there exists $\eta_0 > 0$ such that $\omega^{\otimes \eta}(\shortminus D)$ is ample for every $\eta \geq \eta_0$. In other words, we have
\begin{corollary}
The Hodge line bundle $\omega = \det \Omega^{\tor}$ is $D$-ample on the toroidal compactification $\Sh^{\tor}$.
\end{corollary}
\begin{rmrk}
The effective Cartier divisor $D$ appearing in the corollary obviously depends on the choice of the $GL(X^*)$-equivariant polarization function on the decomposition in polyhedral cones $\Sigma$.
\end{rmrk}

\subsection{Automorphic vector bundles}
We define the automorphic vector bundles over the Siegel variety. We choose an intermediary parabolic subgroup $P_0 \subset P$ of type $I_0 \subset I \subset \Delta$ and we denote $P_{0,L} := P_0 \cap L \subset L$ the parabolic subgroup of $L$. 
\begin{definition}
We define the flag bundle $\pi : Y^{\tor}_{I_0} \rightarrow \Sh^{\tor}$ of type $I_0$ as the flag bundle $\mathcal{F}_{P_{0,L}}(\Omega^{\tor})$ (as in definition \ref{def1}) of type $P_{0,L}$ of the $L$-torsor $\Omega^{\tor}$.
\end{definition}
\begin{definition}
From definitions \ref{def_W} and \ref{def_L}, we have functors 
\begin{equation*}
\mathcal{W} : \Rep(L) \rightarrow \Loc(\Oc_{\Sh^{\tor}}),
\end{equation*}
\begin{equation*}
\Lc : \Rep(P_{0,L}) \rightarrow \Loc(\Oc_{Y_{I_0}^{\tor}})
\end{equation*}
and we call automorphic bundle any vector bundle in the essential image of these functors. Moreover, if $\lambda$ is a character of $P_0$, we denote $\nabla(\lambda)$ the automorphic vector bundle $\mathcal{W}(\Ind_{P_{0,L}}^L \lambda)$ on $\Sh^{\tor}$ and $\Lc_{\lambda}$ the automorphic line bundle $\Lc(\lambda)$ on $Y_{I_0}^{\tor}$. With our conventions the module $\Ind_{P_{0,L}}^L \lambda$ is isomorphic to the costandard representation of highest weight $w_0w_{0,L}\lambda$.
\end{definition}
\begin{corollary}\label{prop13_sieg}
Let $\lambda$ be a dominant character of $P_0$. We have an isomorphism of vector bundles
\begin{equation*}
R\pi_* \Lc_{\lambda} = \nabla(\lambda)[0].
\end{equation*}
\end{corollary}
\begin{proof}
It is a direct consequence of proposition \ref{prop5}.
\end{proof}
\begin{example}
We have the following special cases.
\begin{enumerate}
\item If $\lambda = (0, \ldots, 0, \shortminus 1)$, then $\nabla(\lambda) = \Omega^{\tor}$
\item If $\lambda = (0, \ldots, 0, \shortminus n)$ with $n \geq 1$, then $\nabla(\lambda) = \Sym^n \Omega^{\tor}$
\item if $\lambda = (\shortminus 1, \ldots, \shortminus 1)$, then $\nabla(\lambda) = \Lambda^g\Omega^{\tor} = \omega$
\end{enumerate}
\end{example}
We recall the Kodaira-Spencer isomorphism. 
\begin{proposition}[{\cite[Chap. 3, sect. 9]{MR1083353}}]\label{prop19_sieg}
The Kodaira-Spencer map on the toroidal compactification of the Siegel variety
\begin{equation*}
\rho_{\text{KS}} : \Sym^2 \Omega^{\tor} \DistTo \Omega^1_{\Sh^{\tor}}(\log D)
\end{equation*}
is an isomorphism between the automorphic bundle $\nabla(0, \cdots, 0,\shortminus 2)$ and the sheaf of logarithmic $1$-differentials $\Omega^1_{\Sh^{\tor}}(\log D)$. Taking the determinant yields an isomorphism of line bundles
\begin{equation*}
\Omega^d_{\Sh^{\tor}}(\log D) \simeq \nabla({-2\rho^L})
\end{equation*} 
where $d$ is the dimension of $\Sh^{\tor}$ and
\begin{equation*}
\rho^L = \frac{1}{2} \sum_{\alpha \in \Phi^{+} \backslash \Phi_L^{+}} \alpha.
\end{equation*}
\end{proposition}
We recall a result on the $D$-ampleness of automorphic line bundles $\Lc_{\lambda}$ that admits generalized Hasse invariants.
\begin{definition}
Let $\lambda$ be a character of $T$. For every coroot such that $\langle \lambda,\alpha^{\vee}\rangle \neq 0$, we set:
\begin{equation*}
\text{Orb}(\lambda,\alpha^{\vee}) = \left\{ \frac{ | \langle \lambda,w\alpha^{\vee}\rangle |}{|\langle \lambda,\alpha^{\vee} \rangle | } \ | \ w \in W \right\}
\end{equation*}
and we say that $\lambda$ is 
\begin{enumerate}
\item orbitally $p$-close if $\max_{\alpha \in \Phi} \text{Orb}(\lambda,\alpha^{\vee}) \leq p-1$
\item $\mathcal{Z}_{\emptyset}$-ample if $\langle \lambda, \alpha^{\vee}\rangle >0$ for all $\alpha \in I$ and $\langle \lambda, \alpha^{\vee}\rangle <0$ for all $\alpha \in \Phi^+ \backslash \Phi^+_L$.
\end{enumerate}
\end{definition}
\vspace*{0.2cm}
The following result is due to \cite{StrohPrep}.
\begin{proposition}[{\cite[Theorem 5.11]{alexandre2022vanishing}}]\label{prop_automorphic_line}
Let $\lambda$ be a character of $T$. If $\lambda$ is orbitally $p$-close and $\mathcal{Z}_B$-ample, then $\mathcal{L}_{\lambda}$ is $D$-ample on $Y^{\tor}$.
\end{proposition}
\vspace{0.2cm}
We can now state and prove one of our main results.
\begin{theorem}\label{th_automorphic_bundle}
Let $\lambda$ be a dominant character of $T$. 
\begin{enumerate}
\item If $\lambda$ is a positive parallel weight, i.e. $\lambda = k(1,\cdots,1)$ with $k<0$, or
\item if $2\lambda+2\rho_L$ is orbitally $p$-close and $\mathcal{Z}_{\emptyset}$-ample
\end{enumerate}
then the automorphic vector bundle $\nabla(\lambda)$ is $(\varphi,D)$-ample on $\Sh^{\tor}$. 
\end{theorem}
\begin{proof}
This a direct consequence from theorem \ref{th2} and proposition \ref{prop_automorphic_line}.
\end{proof}
To illustrate our result when $g = 2$, we represent the weights $\lambda = ( k_1, k_2)$ such that the automorphic bundle $\nabla(\lambda)$ is $(\varphi,D)$-ample on the Siegel threefold for different values of $p$ in the figure \ref{figure_amp}.
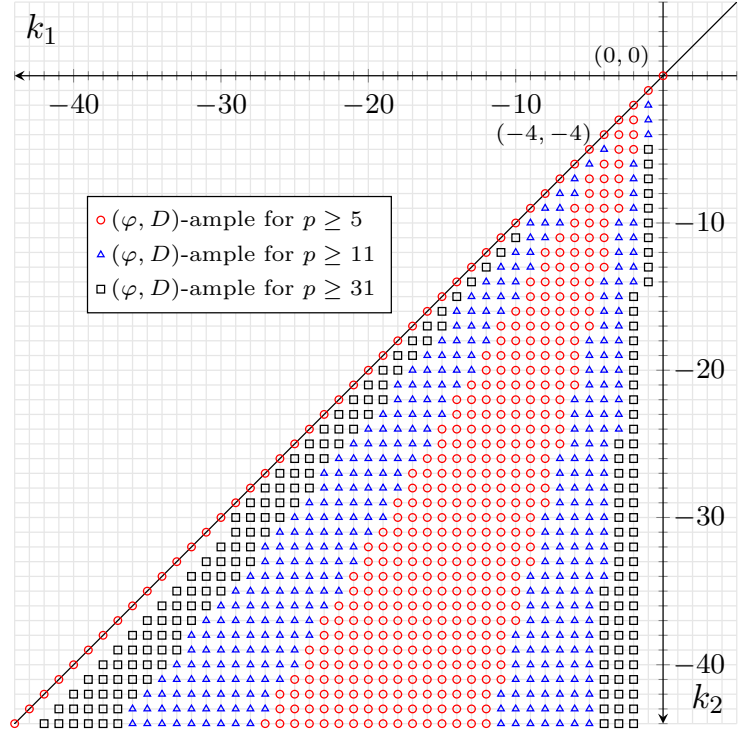
\begin{figure}[h]
\centering
\begin{tikzpicture}[scale = 1.2]
\begin{axis}[
 axis lines=middle,
	grid=both,
	grid style={black!10},
	xmin=-44,
	xmax=5,
	ymin=-44,
	ymax=5,
	legend style={at={(0.1,0.65)},anchor=west, font = \tiny},
 	xlabel=$k_2$,
	ylabel=$k_1$,
	minor tick num=9,
	x axis line style = {stealth-},
	y axis line style = {stealth-},
	xticklabel shift={0.0cm},
	xlabel style={yshift=-7.2cm},
	yticklabel shift={-0.9cm},
	ylabel style={xshift=-7.2cm},
]

\addplot [only marks,
	color=red,
	mark=o,
	mark options={scale=0.6, fill=white}]
	table{results_ample_auto/g=2p=5.txt};
	\addlegendentry[align = left]{$(\varphi,D)$-ample for $p \geq 5$}

\addplot [only marks,
	color=blue,
	mark=triangle,
	mark options={scale=0.6, fill=white}]
	table{results_ample_auto/g=2p=11.txt};
	\addlegendentry[align = left]{$(\varphi,D)$-ample for $p \geq 11$}
	
	\addplot [only marks,
	color=black,
	mark=square,
	mark options={scale=0.6, fill=white}]
	table{results_ample_auto/g=2p=31.txt};
	\addlegendentry[align = left]{$(\varphi,D)$-ample for $p \geq 31$}
	
\draw[scale=0.5, domain=-70:70, smooth, variable=\x, black] plot ({\x}, {\x});

\node at (-2.8,1.3) {\tiny $(0,0)$};
\node at (-8.1,-4.0) {\tiny $(-4,-4)$};

\end{axis}

\end{tikzpicture}
\caption{$(\varphi,D)$-ampleness of automorphic bundles $\nabla(\lambda)$ when $g = 2$.}
\label{figure_amp}
\end{figure}
\section{Hyperbolicity of the Siegel variety}\label{sect_hyp}
\subsection{The supersingular pencil of Moret-Bailly}
Recall that $k$ is an algebraically closed field of characteristic $p$. Denote $\Sh_g$ the Siegel variety of genus $g$ and full level $N \geq 3$ (with $p \nmid N$) over $k$ and $\Sh_g^{\tor}$ a smooth toroidal compactification with boundary a normal crossing divisor $D_{\red}$. Recall that $D$ denotes the effective divisor supported on the boundary that appears as the exceptional divisor of the blow-up from $\Sh_g^{\tor}$ to the minimal compactification of $\Sh_g$. In \cite{MR3618576}, Moret-Bailly constructs a non-isotrivial family $A \rightarrow \Pb^1_{k}$ of principally polarized supersingular abelian surfaces over the projective line with a full level $N$-structure. This family yields a closed immersion $\iota_2 : \Pb_k^1\hookrightarrow \Sh_2$ whose image belongs to the supersingular locus of the Siegel threefold. In particular, we already know that $\Sh_g^{\tor}$ is not hyperbolic when $g = 2$. This family can be used to contradict the hyperbolicity of the Siegel variety for all $g\geq 2$: Take an abelian variety $A_0$ of dimension $g-2$ over $k$ and consider the closed immersion $\iota := \iota_{A_0} \circ \iota_2$
\begin{equation*}
\begin{tikzcd}
\Pb^1_{k} \arrow[rd,"\iota"] \arrow[r,"\iota_2"] & \Sh_{2}^{\tor} \arrow[d,"\iota_{A_0}"] \\
& \Sh_{g}^{\tor} 
\end{tikzcd}
\end{equation*}
where $\iota_{A_0}$ sends an abelian surface $A$ to the fibre product $A \times_{k} A_0$. It also shows that the logarithmic cotangent bundle $\Omega^1_{\Sh_g^{\tor}}(\log D_{\red})$ cannot be nef. Indeed, $\iota$ induces a surjective morphism
\begin{equation*}
\iota^*\Omega^1_{\Sh_g^{\tor}}(\log D_{\red}) \rightarrow \Omega^1_{\Pb^1}
\end{equation*}
and if $\Omega^1_{\Sh^{\tor}}(\log D_{\red})$ was nef, it would imply that $\Omega^1_{\Pb^1} = \Oc_{\Pb^1}(-1)$ is nef. In the rest of this subsection, we will study more closely the non-positivity of certain automorphic bundles. Our goal is to show the following.
\begin{proposition}\label{prop_not_nef}
Assume that $g \in \{2,3\}$. Any automorphic bundle $\nabla(k_1,\cdots,k_g)$ on $\Sh^{\tor}$ where $k_1 = 0$ is not nef.
\end{proposition}
\begin{rmrk}
In particular, we recover that the bundle $\Omega^1_{\Sh^{\tor}}(\log D_{\red}) = \nabla(0,\cdots, 0,\shortminus 2)$ is not nef. We believe that this result generalizes to every $g \geq 2$.
\end{rmrk}
\vspace{0.2cm}
\begin{proof}
Consider a dominant character $\lambda$ of $T$ and write $I_0 \subset I$ for the set of simple roots such that $\langle \lambda,\alpha^{\vee} \rangle = 0$. As a consequence, the line bundle $\Lc_{\lambda}$ on $Y_{I_0}^{\tor}$ is relatively $\pi$-ample which implies that we have a surjective map for some $n \geq 1$ large enough
\begin{equation*}
\pi^*\pi_*\Lc^{\otimes n}_{\lambda} = \pi^*\nabla(n\lambda) \rightarrow \Lc^{\otimes n}_{\lambda}.
\end{equation*}
In particular, if $\nabla(\lambda)$ was nef, it would imply that $\nabla(\lambda)^{\otimes n}$, hence $\nabla(n\lambda)$ and $\Lc_{\lambda}$ would be nef. We are reduced to show the non-nefness of $\Lc_{\lambda}$ on $Y_{I_0}^{\tor}$, which can be tested on $Y_{\emptyset}^{\tor}$. We claim that we can always find a EO stratum $Y_{I_0,w}^{\tor}$ such that the following intersection product is negative
\begin{equation*}
c_1(\Lc_{\lambda})^{l(w)} \cdot [\overline{Y_{I_0}^{\tor}}] < 0.
\end{equation*}
These intersection computations are done in the appendix \ref{appendix1}.
\end{proof}
\subsection{Understanding the failure of hyperbolicity in positive characteristic}
We have seen that $\Omega^1_{\Sh^{\tor}}(\log D_{\red})$ cannot be nef as we can always see $\mathbb{P}^1$ as a closed curve in $\Sh^{\tor}$. Consider a partition $\lambda$ with height $\text{ht}(\lambda) \leq \dim \Sh^{\tor}$ and denote $S_{\lambda}$ the corresponding Schur functor 
\begin{equation*}
S_{\lambda} : \Loc(\Oc_{\Sh^{\tor}}) \rightarrow \Loc(\Oc_{\Sh^{\tor}})
\end{equation*}
as a strict polynomial functor on the category of locally free modules of finite rank over $\Sh^{\tor}$. We start with the following lemma.
\begin{lemma}\label{lem_log_general}
If $S_{\lambda} \Omega^1_{\Sh^{\tor}}(\log D_{\red})$ is $(\varphi,D)$-ample and $\iota : V \hookrightarrow \Sh^{\tor}$ is any subvariety such that
\begin{enumerate}
\item $V$ is smooth,
\item $\iota^{-1}D_{\red}$ is a normal crossing divisor,
\item $\dim V \geq \text{ht}(\lambda)$,
\end{enumerate}
then the logarithmic canonical bundle $\omega_{V}(\iota^{-1}D_{\red})$ is $(\varphi,\iota^{-1}D)$-ample. In particular, it is nef and big with exceptional locus contained in the boundary and $V$ is of log general type with respect to $D$. 
\end{lemma}
\begin{proof}
The surjective morphism
\begin{equation*}
\iota^*\Omega^1_{\Sh^{\tor}}(\log D_{\red}) \rightarrow \Omega^1_{V}(\log \iota^{-1}D_{\red})
\end{equation*}
induces a surjective morphism
\begin{equation*}
\iota^*S_{\lambda}\Omega^1_{\Sh^{\tor}}(\log D_{\red}) \rightarrow S_{\lambda}\Omega^1_{V}(\log \iota^{-1}D_{\red})
\end{equation*}
and by proposition \ref{prop18} and \ref{prop16}, we deduce that $S_{\lambda}\Omega^1_{V}(\log \iota^{-1}D_{\red})$ is $(\varphi,\iota^{-1}D)$-ample. Since $\text{ht}(\lambda) \leq \dim V$, the bundle
\begin{equation*}
\det S_{\lambda}\Omega^1_{V}(\log \iota^{-1}D_{\red}) = \left(\omega_{V}(\iota^{-1}D_{\red})\right)^{\otimes \frac{|\lambda|\dim\nabla(\lambda)}{g}}
\end{equation*}
is non-zero and $(\varphi,\iota^{-1}D)$-ample. We conclude with proposition \ref{squarephiD}.
\end{proof} 
With this fundamental lemma in mind, the aim is to find partitions $\lambda$ that ensure the $(\varphi,D)$-ampleness of $S_{\lambda} \Omega^1_{\Sh^{\tor}}(\log D_{\red})$, which is isomorphic to $S_{\lambda} \Sym^2 \Omega^{\tor}$ by the Kodaira-Spencer isomorphism (proposition \ref{prop19_sieg}). Recall that under the assumption $p \geq 2|\lambda |-1$, the plethysm $S_{\lambda} \circ \Sym^2$ is filtered by Schur functors $S_{\eta}$ by proposition \ref{th_plethysm}. This allows us to state the following lemma.
\begin{lemma}\label{lem_filter_phiD}
Let $\lambda$ be a partition and assume that $p \geq 2|\lambda |-1$. If $S_{\lambda} \circ \Sym^2$ is filtered by Schur functors $S_{\eta}$ such that $S_{\eta}\Omega^{\tor}$ is $(\varphi,D)$-ample or zero, then $S_{\lambda} \Omega^1_{\Sh^{\tor}}(\log D_{\red})$ is $(\varphi,D)$-ample.
\end{lemma}
\begin{proof}
Since $\Sh^{\tor}$ is smooth, this is a direct consequence of proposition \ref{prop19}.
\end{proof}
\begin{rmrk}
If $\eta$ has more than $g$ parts, then $S_{\eta}\Omega^{\tor} = 0$. Otherwise, $S_{\eta}\Omega^{\tor} = \nabla(w_{0}w_{0,L}\eta)$.
\end{rmrk}
\vspace{0.2cm}
By theorem \ref{th_automorphic_bundle}, we are reduced to find a partition $\lambda$ such that all the partition $\eta$ with at most $g$ parts appearing in the plethysm $S_{\lambda} \circ \Sym^2$ are such that $2\rho_L+2w_{0}w_{0,L}\eta$ is orbitally $p$-close and $\mathcal{Z}_{\emptyset}$-ample. See the appendix \ref{appendix2} for some explicit plethysm computations in the cases $g = 2, 3, 4$ which helped us to build some intuition about the general case.

\subsubsection{The general case}
Consider the Siegel variety $\Sh^{\tor}$ of genus $g$ over $k$. In this section, we prove the following result.
\begin{theorem}\label{th_hyperbolic}
Assume that $p \geq g^2+3g+1$. For all $k \geq g(g-1)/2+1$, the bundle $\Omega^k_{\Sh^{\tor}}(\log D_{\red})$ is $(\varphi,D)$-ample.
\end{theorem}
\begin{corollary}\label{cor_hyperbolic}
Assume that $p \geq g^2+3g+1$. Any subvariety $\iota : V \hookrightarrow \Sh^{\tor}$ of codimension $\leq g-1$ satisfying
\begin{enumerate}
\item $V$ is smooth,
\item $\iota^{-1}D_{\red}$ is a normal crossing divisor,
\end{enumerate}
is of log general type with respect to $D$.
\end{corollary}
\vspace*{0.2cm}
Recall the following conjecture.
\begin{conjecture}[Green-Griffiths-Lang]\label{GGL}
Let X be an irreducible projective complex variety. Denote $\text{Exc}(X)$ the Zariski closure of the union of the images of all non-constant holomorphic maps $\C \rightarrow X$. Then $X$ is of general type if and only if $\text{Exc}(X) \neq  X$.
\end{conjecture}
\begin{rmrk}
The Green-Griffiths-Lang conjecture fails in positive characteristic. Specifically, in characteristic $p > 0$, there exist unirational surfaces of general type. These surfaces are dominated by the projective plane $\mathbb{P}^2$ via rational maps, yet they possess a big canonical bundle, classifying them as surfaces of general type. This phenomenon contradicts the expectation from the conjecture that varieties of general type should exhibit hyperbolic behavior and not admit non-constant rational curves.
\end{rmrk}
\vspace*{0.2cm}
Motivated by the Green-Griffiths-Lang conjecture, we can formulate the following one.
\begin{conjecture}
For $p$ large enough, there is a closed subscheme $E \subset \Sh^{\tor}$ such that for any subvariety $\iota : V \rightarrow \Sh^{\tor}$ satisfying
\begin{enumerate}
\item $V$ is smooth,
\item $\iota^{-1}D_{\red}$ is a normal crossing divisor,
\end{enumerate}
$V$ is of log general type if and only if $V \nsubseteq E$.
\end{conjecture}
\vspace{0.2cm}
The theorem \ref{th_hyperbolic} indicates that such an exceptional locus $E \subset \Sh^{\tor}$ should have codimension $>g-1$. We believe it has exactly codimension $g$.
\begin{proof}[Proof of theorem \ref{th_hyperbolic}]
The strategy is to study a $\nabla$-filtration of the $k^{\text{th}}$-exterior power of the bundle $\Omega^1_{\Sh^{\tor}}(\log D_{\red})$ and check that all the graded pieces are $(\varphi,D)$-ample automorphic vector bundles when $p \geq g^2+3g+1$ and $k \geq g(g+1)/2-(g-1)$. By the Kodaira-Spencer isomorphism of proposition \ref{prop19_sieg}, the bundle $\Lambda^k\Omega^1_{\Sh^{\tor}}(\log D_{\red})$ is isomorphic to $\mathcal{W}(\Lambda^k\Sym^2 \std_{\GL_g})$. By proposition \ref{th_plethysm}, the $\GL_g$-module $\Lambda^k\Sym^2 \std_{\GL_g}$ has a $\nabla$-filtration when $p > k$ and it implies that $\Lambda^k\Omega^1_{\Sh^{\tor}}(\log D_{\red})$ is filtered by automorphic bundles $\nabla(w_0w_{0,L}\lambda)$'s where the $\lambda$'s are the highest weights of the $\nabla$-filtration of $\Lambda^k\Sym^2 \std_{\GL_g}$. As explained in the example \ref{ex_plethysm}, determining the Schur functors appearing in a plethysm $S_{\lambda}\circ S_{\mu}$ is often a hard task, however the plethysm $\Lambda^k \circ \Sym^2$ belongs to the one of the few cases where a general formula is known. We start with a notation.
\begin{notation}
Let $k$ be a positive integer and $\lambda$ a partition of $k$ in $r$ distinct parts. We denote $2[\lambda]$ the partition of $2k$ whose main-diagonal hook lengths are $2\lambda_1, \cdots, 2\lambda_r$, and whose $i^{\text{th}}$-part has length $\lambda_i+i$. For example, we have
\begin{equation*}
\begin{tikzcd}[ampersand replacement=\&]
    2[(5,3,1)] = \begin{ytableau}
	  *(green) {10} & {} & {} & {} & {} & {} \\
	{} & *(green) {6}  & *(gray) {\rightarrow} & *(gray) & *(gray)   \\
	{} & *(gray) \downarrow & *(green)  {2} & {} \\
	{} & *(gray)  \\
	{}
    \end{ytableau} = (6,5,4,2,1)
\end{tikzcd}
\end{equation*}
where the diagonal hook have lengths $10,6,2$. 
\end{notation}
\begin{lemma}[{\cite[Lemma 7]{MR2497590}}]
Assume that $p > k$. Then the polynomial functor $\Lambda^k \circ \Sym^2$ has a filtration where the graded pieces are the Schur functors $S_{2[\lambda]}$ where $\lambda$ range over the set of partitions of $k$ in distinct parts.
\end{lemma}
\begin{example}
Consider the case $k = 5$. The partitions of $5$ in distinct parts are $(5), (4,1)$ and $(3,2)$. The plethysm $\Lambda^5 \circ \Sym^2$ is then filtered by the Schur functors $S_{2[(5)]} = S_{(6,1^4)}$, $S_{2[(4,1)]} = S_{(5,3,1^2)}$ and $S_{2[(3,2)]} = S_{(4,4,2)}$.
\end{example}
\vspace{0.2cm}
Since we evaluate this plethysm at the Hodge bundle $\Omega^{\tor}$ which has rank $g$, we can discard the partitions $2[\lambda]$ of height strictly greater than $g$ (for such partitions, the evaluation vanishes). Since the height of $2[\lambda]$ is $\lambda_1$, we want to study the $(\varphi,D)$-ampleness of the automorphic bundles
\begin{equation*}
S_{2[\lambda]}\Omega^{\tor} = \mathcal{W}(\nabla(2[\lambda])) = \nabla(w_0w_{0,L}2[\lambda])
\end{equation*}
where $\lambda$ is partition of $k$ in distinct parts with $\lambda_1 \leq g$. By theorem \ref{th_automorphic_bundle}, we know it is the case when $2w_0w_{0,L}2[\lambda]+ 2\rho_L$ is $\mathcal{Z}_{\emptyset}$-ample and orbitally $p$-close. Even if the second condition is always satisfied for $p$ large enough, the first condition may not be satisfied as explained in the appendix \ref{ex_pleth_g3}. In proposition \ref{prop_not_nef}, we have seen that automorphic bundles of the form $\nabla(\eta)$ where $\eta = (\eta_1 \geq \cdots \geq \eta_g)$ is a dominant character such that $\eta _1= 0$ are not nef, hence not $(\varphi,D)$-ample. Conversely, we will see that any automorphic bundle $\nabla(\eta)$, where $\eta$ is a dominant character such that $\eta_1 \leq -1$, is $(\varphi,D)$-ample if $p$ is greater than a specific bound which depends on $\eta$. We start with the following lemma.
\begin{lemma}\label{lem8}
Consider two $\GL_g$-dominant character $\lambda = (\lambda_1 \geq \cdots \geq \lambda_g \geq 0)$ and $\mu = (\mu \geq \cdots \geq \mu \geq 0)$. The $\GL_g$-module $\nabla(\lambda)\otimes \nabla(\mu)$ is filtered by costandard modules $\nabla(\eta)$ such that $\eta_g \geq \lambda_g + \mu_g$ and  $\eta_1 \leq \lambda_1 + \mu_1$.
\end{lemma}
\begin{proof}[Proof of the lemma]
See proposition \ref{prop8_van} for the existence of the $\nabla$-filtration. The tensor product of two polynomial representation of $\GL_g$ is still a polynomial representation. Apply it to $\nabla(\lambda-(\lambda_g^g)) \otimes \nabla(\mu-(\mu_g^g))$ where $(\lambda_g^g) = (\lambda_g, \cdots, \lambda_g)$ and $(\mu_g^g) = (\mu_g, \cdots, \mu_g)$ to get the first inequality. The second inequality follows from the fact that $\lambda + \mu$ is the highest weight of $\nabla(\lambda)\otimes \nabla(\mu)$.
\end{proof}
\begin{proposition}\label{prop_amp_bound}
Let $\eta = (\eta_1 \geq \cdots \geq \eta_g)$ be a dominant character such that $\eta_1 \leq -1$. Then the automorphic bundle $\nabla(\eta)$ is $(\varphi,D)$-ample if $p \geq (g+1)|\eta_g| + g$.
\end{proposition}
\begin{proof}[Proof of the proposition]
By proposition \ref{squarephiD}, it is enough to show that $\nabla(\eta)^{\otimes n}$ is $(\varphi,D)$-ample for some $n\geq 1$. By lemma \ref{lem8}, the bundle $\nabla(\eta)^{\otimes n}$ is filtered by automorphic bundles of the form $\nabla(\delta)$ where $\delta_1 \leq n\eta_1$ and $\delta_g \geq n\eta_g$. To apply theorem \ref{th_automorphic_bundle}, we need to see that each $2\delta + 2\rho_L$ is $\mathcal{Z}_{\emptyset}$-ample and orbitally $p$-close. We first focus on the $\mathcal{Z}_{\emptyset}$-ampleness of $\gamma := 2\delta + 2\rho_L$. In other words, we need to check that
\begin{equation*}
\begin{aligned}
\gamma = 2\delta + 2\rho_L &= (2\delta_1, \cdots, 2\delta_g) + (g-1,g-3,\cdots, -(g-1)) \\
&= (2\delta_1+g-1, \cdots, 2\delta_g-g+1)
\end{aligned}
\end{equation*}
is such that $0 > 2\delta_1+g-1 > 2\delta_2+g-3 > \cdots > 2\delta_g -g+1$. The first inequality being the only one non-trivial, it is enough to have $n > (g-1)/2$ as it implies
\begin{equation*}
\begin{aligned}
2\delta_1+g-1 &\leq 2n\eta_1+g-1  \\
&\leq -2n+g-1 \\
&< 0.
\end{aligned}
\end{equation*}
For the orbitally $p$-closeness of $\gamma = 2\delta + 2\rho_L$, we have the following bound
\begin{equation*}
\begin{aligned}
\max_{\alpha \in \Phi, w \in W, \langle \gamma,\alpha^{\vee} \rangle \neq 0} | \frac{\langle \gamma,w\alpha^{\vee}\rangle}{\langle \gamma,\alpha^{\vee} \rangle }| &\leq \max_{1\leq i\leq j \leq g}\frac{|\gamma_j|+|\gamma_i|}{2} \\
&\leq \frac{2|\gamma_g|}{2} \\
&\leq |2\delta_g-(g-1)| \\
&\leq 2|\delta_g| + (g-1) \\
&\leq 2n|\eta_g| + (g-1) \text{ by lemma \ref{lem8}}
\end{aligned}
\end{equation*}
and we deduce that it is enough to have $2n|\eta_g| + g \leq p$. Combining it with the restriction $n = \lfloor (g-1)/2 \rfloor +1 \leq (g+1)/2$ which ensure the $\mathcal{Z}_{\emptyset}$-ampleness of $\gamma$, we get
\begin{equation*}
p \geq (g+1)|\eta_g| + g.
\end{equation*}
\end{proof}
With the proposition \ref{prop_amp_bound} in mind, recall that we want to prove that the bundle
\begin{equation*}
\nabla(w_0w_{0,L}2[\lambda])
\end{equation*}
is $(\varphi,D)$-ample when $\lambda$ is a partition of $k$ in distinct parts such that $\text{ht}(2[\lambda]) = \lambda_1 \leq g$. If there exists such a partition $\lambda$ with $\lambda_1 \leq g-1$, we will not be able to apply the proposition \ref{prop_amp_bound} to $w_0w_{0,L}2[\lambda]$ as the first term will be $0$. To avoid these partitions, we prove the following lemma.
\begin{lemma}
Assume that $p> k$. All the automorphic bundles $\nabla(\eta)$ appearing as graded pieces of the $\nabla$-filtration of $\Lambda^k\Sym^2 \Omega^{\tor}$ satisfy $\eta_1 \leq -1$ if and only if $k \geq g(g-1)/2+1$.
\end{lemma}
\begin{proof}[Proof of the lemma]
Assume that $k \geq g(g-1)/2+1$. We need to check that there exists no partition $\lambda$ of $k$ in distinct parts such that $\text{ht}(2[\lambda]) = \lambda_1 \leq g-1$. Consider a partition $\lambda$ of $k$ in $r$-distinct parts. We have
\begin{equation*}
\frac{g(g-1)}{2}+1 \leq k = \lambda_1 + \lambda_2 + \cdots + \lambda_r \leq \frac{\lambda_1(\lambda_1+1)}{2}
\end{equation*}
which is possible only if $\lambda_1 \geq g$. Conversely, if $k \leq g(g-1)/2$, it is not hard to find a partition $\lambda$ of $k$ in distinct parts such that $\lambda_1 \leq g-1$.
\end{proof}
Since $(2[\lambda])_1 = \lambda_1+1$, we conclude with the proposition \ref{prop_amp_bound} which says that each automorphic bundle $\nabla(w_0w_{0,L}2[\lambda])$ where $\lambda_1 = g$ is $(\varphi,D)$-ample when
\begin{equation*}
p \geq g^2 + 3g +1 = (g+1)\underbrace{|(w_0w_{0,L}2[\lambda])_g|}_{= g+1} + \ g
\end{equation*}
\end{proof}
\appendix
\section{Intersection computations on Ekedahl-Oort strata}\label{appendix1}
We will use the results of \cite{wedhorn2018tautological} to do some computation on the Chow $\Q$-algebra of the partial flag bundle $\Pb(\Omega^{\tor})$. Recall that $I$ denotes the type of the parabolic subgroup $P = P_{-\mu}$ of $\Sp_{2g}$. Let $I_0$ denotes a subset of $I$ and consider the morphisms
\begin{equation*}
\zeta : \Sh^{\tor} \rightarrow \Sp_{2g}\Zip^{\mu}
\end{equation*}
and 
\begin{equation*}
\zeta_{I_0} : Y_{I_0}^{\tor} \rightarrow \Sp_{2g}\ZipFlag^{\mu,I_0}
\end{equation*}
as defined in \cite{MR3989256} and \cite{MR3973104}. For all $w \in {}^IW$, we denote $\Sh^{\tor}_{w} := \zeta^{-1}([w])$  the EO stratum of the Siegel variety where $[w] \subset \Sp_{2g}\Zip^{\mu}$ is the corresponding substack. More generally\footnote{If $I_0 = I$, then $Y^{\tor}_{I_0} = \Sh^{\tor}$}, for all $w \in {}^{I_0}W$, we denote $Y^{\tor}_{I_0,w} := \zeta_{I_0}^{-1}([w])$ the EO stratum of the partial flag bundle of type $I_0 \subset I$ where $[w] \subset \Sp_{2g}\ZipFlag^{\mu,I_0}$ is the corresponding substack. The morphism $\zeta_{I_0}$ induces a pullback map on the corresponding Chow $\Q$-algebra
\begin{equation*}
\begin{tikzcd}
A^{\bullet}(\Sp_{2g}\ZipFlag^{\mu,I_0}) \arrow[r,"\zeta_{I_0}^*"] & A^{\bullet}(Y_{I_0}^{\tor})
\end{tikzcd}
\end{equation*}
and we call the image of $\zeta_{I_0}^*$, the tautological ring $\mathcal{T}_{I_0}$ of $Y^{\tor}_{I_0}$. Clearly, the Chow $\Q$-algebra of $\Sp_{2g}\ZipFlag^{\mu,I_0}$ is generated by the cycle classes of the EO strata $\overline{[w]}$ for $w \in {}^{I_0}W$ but we would like another description relying on chern classes of automorphic bundles. We have a morphism of $\Q$-vector spaces
\begin{equation*}
\begin{tikzcd}[row sep = tiny]
c_1 : X^*(T) \arrow[r] & A^1(\Sp_{2g}\ZipFlag^{\mu,\emptyset}) \\
\lambda \arrow[r,mapsto] & c_1(\Lc_{\lambda})
\end{tikzcd}
\end{equation*}
which induces a morphism of $\Q$-algebras $S \rightarrow A^{\bullet}(\Sp_{2g}\ZipFlag^{\mu,\emptyset})$ where $S = \Sym X^*(T)$ is the symmetric algebra of the characters of $T$. By \cite[Theorem 3]{wedhorn2018tautological}, this map is surjective with kernel generated by the $W$-invariant elements of degree $>0$. We deduce a description of the Chow $\Q$-algebra of $\Sp_{2g}\ZipFlag^{\mu,\emptyset}$ as 
\begin{equation*}
\begin{tikzcd}[row sep = small]
S \arrow[r] \arrow[d] & A^{\bullet}(\Sp_{2g}\ZipFlag^{\mu,\emptyset}) \\
S/\mathcal{I}S \arrow[ru]
\end{tikzcd}
\end{equation*}
where $\mathcal{I}$ is the augmentation ideal of the $W$-invariant elements of $S$. This ideal admits an explicit description as the augmentation ideal of a polynomial algebra
\begin{equation*}
\mathcal{I} = \Q[f_1,\cdots,f_g]_{\geq 1} \ \ f_i = x_1^{2i}+\cdots+x_g^{2i}.
\end{equation*}
In particular, the tautological ring $\mathcal{T}_{\emptyset}$ is generated as a $\Q$-algebra by the cycle classes of the closed EO strata $\overline{Y^{\tor}_{\emptyset,w}}$ and by the chern classes of the automorphic line bundles $c_1(\Lc_{\lambda})$. The goal is now to express $[\overline{Y^{\tor}_{\emptyset,w}}]$ as an element of $S/\mathcal{I}S$ and to compute intersection products of the form
\begin{equation*}
c_1(\Lc_{\lambda})^{l(w)}\cdot [\overline{Y^{\tor}_{\emptyset,w}}].
\end{equation*}
Following the strategy of \cite{wedhorn2018tautological}, we have implemented on Sage an algorithm which computes $[\overline{Y^{\tor}_{\emptyset,w}}]$ as an element of $S/\mathcal{I}S$. In order to be more explicit, we choose a system of positive roots in a way to obtain:
\begin{equation*}
I = \{e_i-e_{i+1} \ | \ i = 1, \cdots g-1 \}  \subset \Delta = \{e_i-e_{i+1} \ | \ i = 1, \cdots g-1 \} \cup \{2e_g\}.
\end{equation*}
The Weyl group $W = S_g \ltimes ({\Z}/{2\Z})^{g}$ contains $2^gg!$ elements we can write as a product of the simple reflections $s_1, s_2, \cdots, s_g$ associated to $e_1-e_{2},\cdots, e_{g-1}-e_{g},2e_g$.
\subsection{The case $g = 2$}
We represent the Weyl group of $\Sp_4$ with the following diagram
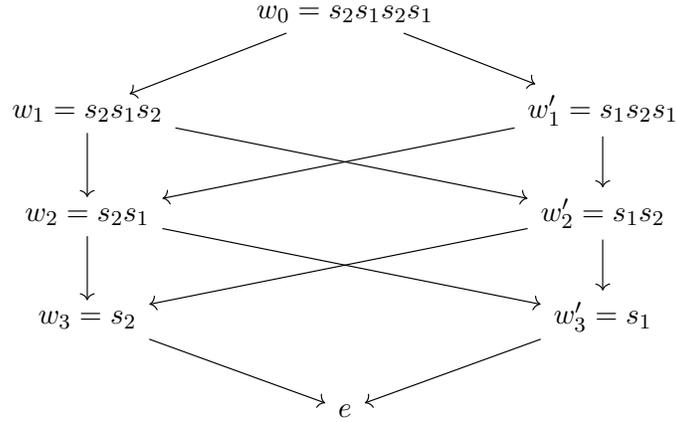
\begin{equation*}
\begin{tikzcd}
& w_0 = s_2s_1s_2s_1 \arrow[ld] \arrow[rd] & \\
w_1 = s_2s_1s_2 \arrow[d] \arrow[rrd] & & w_1^{\prime} = s_1s_2s_1 \arrow[d] \arrow[lld]\\
w_2 =  s_2s_1 \arrow[d] \arrow[rrd] & & w_2^{\prime} = s_1s_2 \arrow[d] \arrow[lld]\\
w_3 = s_2 \arrow[rd] & & w_3^{\prime} = s_1 \arrow[ld] \\
& e &
\end{tikzcd}
\end{equation*}
where an arrow is drawn from $w$ to $w^{\prime}$ if $w^{\prime} \leq w$ and $l(w^\prime) = l(w)-1$. Consider the line bundle $\Lc_{\lambda}$ on $Y^{\tor}_{\emptyset} = \Pb(\Omega^{\tor})$ where $\lambda = (k_1,k_2)$ and recall that $\Lc_{\lambda_{\Omega}} = \Lc_{(0,-1)} = \mathcal{O}(1)$. In the graded algebra $$\mathcal{T}_{\emptyset} = \Q[x_1,x_2]/(x_1^2+x_2^2,x_1^2x_2^2),$$ we have the following formulas
\begin{equation*}
\left\{
\begin{aligned}
&[\overline{Y^{\tor}_{\emptyset, w_0}}] = 1, \\
&[\overline{Y^{\tor}_{\emptyset, w_1}}] = x_1-px_2, \\
&[\overline{Y^{\tor}_{\emptyset,w_1^\prime}}] = -(p-1)(x_1+x_2), \\
&[\overline{Y^{\tor}_{\emptyset,w_2}}] = -(p-1)(px_1+x_2)x_1, \\
&[\overline{Y^{\tor}_{\emptyset,w_2^\prime}}] = (p-1)(px_2-x_1)x_1, \\
&[\overline{Y^{\tor}_{\emptyset,w_3}}] = (p^2-1)(px_2-x_1)x_1^2, \\
&[\overline{Y^{\tor}_{\emptyset,w_3^\prime}}] = (p^2+1)(p-1)(x_1^3+x_2^3),\\
&[\overline{Y^{\tor}_{\emptyset,e}}] = (p^4-1)x_1x_2^3.
\end{aligned}
\right.
\end{equation*}
Since the cycle $x_1x_2^3$ has positive degree and since we are only concerned with the sign of the intersection products, we make the identification $x_1x_2^3 = 1$ and we get the following intersection products.
\begin{equation*}
\left\{
\begin{aligned}
&c_1(\Lc_{\lambda})^{4} \cdot [\overline{Y^{\tor}_{\emptyset, w_0}}] = (k_1x_1+k_2x_2)^4 = 4(k_1k_2^3-k_1^3k_2), \\
&c_1(\Lc_{\lambda})^{3} \cdot [\overline{Y^{\tor}_{\emptyset, w_1}}]  = (pk_1^3-3pk_1k_2^2-3k_1^2k_2+k_2^3), \\
&c_1(\Lc_{\lambda})^{3} \cdot [\overline{Y^{\tor}_{\emptyset, w_1^\prime}}]  = (p-1)(k_1^3+3k_1^2k_2-3k_1k_2^2-k_2^3), \\
&c_1(\Lc_{\lambda})^{2} \cdot [\overline{Y^{\tor}_{\emptyset, w_2}}]  = (p-1)(k_1^2-k_2^2+2pk_1k_2), \\
&c_1(\Lc_{\lambda})^{2} \cdot [\overline{Y^{\tor}_{\emptyset, w_2^\prime}}]  =(p-1)(p(k_2^2-k_1^2)+2k_1k_2), \\
&c_1(\Lc_{\lambda}) \cdot [\overline{Y^{\tor}_{\emptyset, w_3}}]  = (p^2-1)(k_2-pk_1), \\
&c_1(\Lc_{\lambda}) \cdot [\overline{Y^{\tor}_{\emptyset, w_3^\prime}}]  = (p-1)(p^2+1)(k_1-k_2).
\end{aligned}
\right.
\end{equation*}
If $k_1 = 0$ and $k_2 < 0$, then $c_1(\Lc_{\lambda}) \cdot [\overline{Y^{\tor}_{\emptyset, w_3}}]  = (p^2-1)k_2 < 0$, so $\Lc_{\lambda}$ is not nef.
\subsection{The case $g = 3$}
The degree $9$ part of the graded algebra $\mathcal{T}_{\emptyset} =\Q[x_1,x_2,x_3]/\mathcal{I}$ is a $\Q$-vector space of dimension $1$ generated by $x_1^5x_2^3x_3$. We have
\begin{equation*}
\overline{[Y_{\emptyset,e}^{\tor}]} = \underbrace{(p^9-p^8+p^7+2p^4-p^3+p^2-p+1)}_{>0} x_1^5x_2^3x_3
\end{equation*}
and since this polynomial in $p$ is always positive, we may identify $x_1^5x_2^3x_3$ with $1$. We have then 
\begin{equation*}
c_1(\Lc_{\lambda_{\Omega}}) \cdot \overline{[Y_{\emptyset,s_3}^{\tor}]} = -p\left( p^5(p-1)-1\right) < 0,
\end{equation*}
which shows that $\Omega^{\tor}$, hence $\Omega^1_{\Sh^{\tor}}(\log D_{\red}) = \Sym^2\Omega^{\tor}$, is not nef.
\section{Plethysm computations}\label{appendix2}
The plethysm computations are accessible at
\begin{itemize}
\item \href{https://github.com/ThibaultAlexandre/positivity-of-automorphic-bundles}{github.com/ThibaultAlexandre/positivity-of-automorphic-bundles}.
\end{itemize}
\subsection{The case $g = 2$}
The Hodge bundle $\Omega^{\tor}$ is locally free of rank $2$ and $\Omega^1_{\Sh^{\tor}}(\log D_{\red}) = \Sym^2 \Omega^{\tor}$ is locally free of rank $3$. Recall there is no need to assume that $p \geq 2|\lambda |-1$ when taking the highest exterior power. Under the assumption $p \geq 2\times 2 -1 = 3$, we have
\begin{equation*}
\left\{
\begin{aligned}
&\Lambda^{3}\Sym^2 \Omega^{\tor} = \nabla(\shortminus3,\shortminus3) \\
&\Lambda^{2}\Sym^2 \Omega^{\tor} = \nabla(\shortminus1,\shortminus3) \\
&\Sym^2 \Omega^{\tor}  = \nabla(0,\shortminus2).
\end{aligned}
\right.
\end{equation*}
Clearly, the line bundle $\nabla(\shortminus3,\shortminus3)$ is $D$-ample for any $p>0$ and $\nabla(0,\shortminus2)$ is never nef (hence never $(\varphi,D)$-ample) by proposition \ref{prop_not_nef}. For $\nabla(\shortminus1,\shortminus3)$, we need to check wether $(\shortminus1,\shortminus7)$ is orbitally $p$-close and $\mathcal{Z}_{\emptyset}$-ample. This condition is satisfied as soon as $p\geq 11$. By lemma \ref{lem_log_general}, this shows that any (good) subsurface of the Siegel threefold is of log general type when $p\geq 11$. Putting some extra effort, one can show that this result holds with $p = 7$ as well. When $p \geq 2\times 4-1 = 7$, the bundle
\begin{equation*}
S_{(2,2)} \circ \Sym^2 \Omega^{\tor} 
\end{equation*}
is filtered by the automorphic vector bundles $\nabla(\shortminus2,\shortminus6)$ and $\nabla(\shortminus4,\shortminus4)$ which are $(\varphi,D)$-ample when $p \geq 7$. We get the following result:
\begin{proposition}
Assume that $g = 2$. 
\begin{enumerate}
\item If $p\geq 11$, then $\Lambda^{2}\Omega^1_{\Sh^{\tor}}(\log D_{\red})$ is $(\varphi,D)$-ample. 
\item If $p = 7$, then $S_{(2,2)}\Omega^1_{\Sh^{\tor}}(\log D_{\red})$ is $(\varphi,D)$-ample. 
\end{enumerate}
\end{proposition}
\begin{corollary}
Assume that $g = 2$ and $p \geq 7$. If $\iota : S \hookrightarrow \Sh^{\tor}$ is a subvariety of dimension $\geq 2$ such that
\begin{enumerate}
\item $S$ is smooth,
\item $\iota^{-1}D_{\red}$ is a normal crossing divisor,
\end{enumerate}
then $S$ is of log general type with respect to $D$.
\end{corollary}
\subsection{The case $g = 3$}\label{ex_pleth_g3}
The Hodge bundle $\Omega^{\tor}$ is locally free of rank $3$ and $\Omega^1_{\Sh^{\tor}}(\log D_{\red}) = \Sym^2 \Omega^{\tor}$ is locally free of rank $6$. Under the assumption $p \geq 2 \times 5-1 = 9$, we have
\begin{equation*}
\left\{
\begin{aligned}
&\Lambda^{6}\Sym^2 \Omega^{\tor} = \nabla(\shortminus4,\shortminus4,\shortminus4) \\
&\Lambda^{5}\Sym^2 \Omega^{\tor} = \nabla(\shortminus2,\shortminus4,\shortminus4) \\
&\Lambda^{4}\Sym^2 \Omega^{\tor}  = \nabla(\shortminus1,\shortminus3,\shortminus4) \\
&\Lambda^{3}\Sym^2 \Omega^{\tor} \text{ is filtered by } \nabla(\shortminus1,\shortminus1,\shortminus4), \nabla(0,\shortminus3,\shortminus3) \\
&\Lambda^{2}\Sym^2 \Omega^{\tor}  = \nabla(0,\shortminus1,\shortminus3) \\
&\Sym^2 \Omega^{\tor}  = \nabla(0,0,\shortminus2) .
\end{aligned}
\right.
\end{equation*}
We deduce that $\Lambda^{i}\Sym^2 \Omega^{\tor}$ is $(\varphi,D)$-ample for $i = 5,6$ when $p \geq 11$. For $i = 4$, we do not know if the bundle $\nabla(\shortminus1,\shortminus3,\shortminus4)$ is $(\varphi,D)$-ample since
\begin{equation*}
2(\shortminus1,\shortminus3,\shortminus4) + 2\rho = (0,\shortminus6,\shortminus10)
\end{equation*}
is not $\mathcal{Z}_{\emptyset}$-ample. This incites us to consider the plethysm
\begin{equation*}
S_{(2,2,2,2)} \circ \Sym^2 \Omega^{\tor}
\end{equation*}
which is filtered by $\nabla(\shortminus2,\shortminus6,\shortminus8)$, $\nabla(\shortminus3,\shortminus6,\shortminus7)$, $\nabla(\shortminus4,\shortminus4,\shortminus8)$ and $\nabla(\shortminus4,\shortminus6,\shortminus6)$ when $p \geq 2\times 8 -1 = 15$. These automorphic bundles are $(\varphi,D)$-ample when $p \geq 17$ by theorem \ref{th_automorphic_bundle}. It implies that $S_{(2,2,2,2)} \circ \Sym^2 \Omega^{\tor}$ is $(\varphi,D)$-ample when $p \geq 17$. We get the following result:
\begin{proposition}
Assume that $g = 3$. 
\begin{enumerate}
\item If $p\geq 11$, then $\Lambda^{5}\Omega^1_{\Sh^{\tor}}(\log D_{\red})$ is $(\varphi,D)$-ample. 
\item If $p \geq 17$, then $S_{(2,2,2,2)}\Omega^1_{\Sh^{\tor}}(\log D_{\red})$ is $(\varphi,D)$-ample. 
\end{enumerate}
\end{proposition}
\begin{corollary}
Assume that $g = 3$ and $p \geq 17$. If $\iota : V \hookrightarrow \Sh^{\tor}$ is a subvariety of dimension $\geq 4$ such that
\begin{enumerate}
\item $V$ is smooth,
\item $\iota^{-1}D_{\red}$ is a normal crossing divisor,
\end{enumerate}
then $V$ is of log general type with respect to $D$.
\end{corollary}
\subsection{The case $g= 4$}
The Hodge bundle $\Omega^{\tor}$ is locally free of rank $4$ and $\Omega^1_{\Sh^{\tor}}(\log D_{\red}) = \Sym^2 \Omega^{\tor}$ is locally free of rank $10$. Under the assumption $p \geq 2 \times 9-1 = 17$, we have
\begin{equation*}
\left\{
\begin{aligned}
&\Lambda^{10}\Sym^2 \Omega^{\tor} = \nabla(\shortminus5,\shortminus5,\shortminus5,\shortminus 5) \\
&\Lambda^{9}\Sym^2 \Omega^{\tor} = \nabla(\shortminus 3,\shortminus5,\shortminus5,\shortminus 5) \\
&\Lambda^{8}\Sym^2 \Omega^{\tor} = \nabla(\shortminus2,\shortminus4,\shortminus5,\shortminus 5) \\
&\Lambda^{7}\Sym^2 \Omega^{\tor} \text{ is filtered by } \nabla(\shortminus1,\shortminus4,\shortminus4, \shortminus 5), \nabla(2,\shortminus 2,\shortminus 5, \shortminus 5) \\
&\Lambda^{6}\Sym^2 \Omega^{\tor}  \text{ is filtered by } \nabla(\shortminus1,\shortminus2,\shortminus4, \shortminus 5), \nabla(0,\shortminus 4,\shortminus 4, \shortminus 4) \\
&\Lambda^{5}\Sym^2 \Omega^{\tor}  \text{ is filtered by } \nabla(\shortminus1,\shortminus1,\shortminus3, \shortminus 5), \nabla(0,\shortminus 2,\shortminus 4, \shortminus 4) \\
&\Lambda^{4}\Sym^2 \Omega^{\tor}  \text{ is filtered by } \nabla(\shortminus1,\shortminus1,\shortminus1, \shortminus 5), \nabla(0,\shortminus 1,\shortminus 3, \shortminus 4) \\
&\Lambda^{3}\Sym^2 \Omega^{\tor}  \text{ is filtered by } \nabla(\shortminus0,\shortminus 1,\shortminus 1, \shortminus 4), \nabla(0,\shortminus 0,\shortminus 3, \shortminus 3) \\
&\Lambda^{2}\Sym^2 \Omega^{\tor} = \nabla(\shortminus0,\shortminus 0,\shortminus1, \shortminus 3) \\
&\Sym^2 \Omega^{\tor}  =  \nabla(\shortminus0,\shortminus 0,\shortminus0, \shortminus 2).
\end{aligned}
\right.
\end{equation*}
We deduce that $\Lambda^{i}\Sym^2 \Omega^{\tor}$ is $(\varphi,D)$-ample for $i = 8,9,10$ when $p \geq 17$. It does not work for $i = 7$ since the character
\begin{equation*}
2(\shortminus1,\shortminus4,\shortminus4,\shortminus 5) + 2\rho = (1,\shortminus7,\shortminus 9, \shortminus 13)
\end{equation*}
is not $\mathcal{Z}_{\emptyset}$-ample. Under the assumption $p \geq 2\times 14-1 = 27$, the plethysm
\begin{equation*}
S_{(2^7)} \circ \Sym^2 \Omega^{\tor}
\end{equation*}
is filtered by the following list of automorphic bundles
\begin{equation*}
\left\{
\begin{aligned}
&\nabla(\shortminus 2,\shortminus 8,\shortminus8, \shortminus 10) \\
&\nabla(\shortminus 3,\shortminus 6,\shortminus 9, \shortminus 10) \\
&\nabla(\shortminus 3,\shortminus 7 ,\shortminus9, \shortminus 9) \\
&\nabla(\shortminus 3,\shortminus 8,\shortminus 8, \shortminus 9) \\
&\nabla(\shortminus 4,\shortminus 4 ,\shortminus 10, \shortminus 10) \\
&\nabla(\shortminus 4,\shortminus 6,\shortminus 8, \shortminus 10) \\
&\nabla(\shortminus 4,\shortminus 7,\shortminus 8, \shortminus 9) \\
&\nabla(\shortminus4,\shortminus 8,\shortminus 8, \shortminus 8) \\
&\nabla(\shortminus 5,\shortminus 5,\shortminus 9, \shortminus 9) \\
&\nabla(\shortminus 5,\shortminus 6,\shortminus 8, \shortminus 9) \\
&\nabla(\shortminus 5,\shortminus 7,\shortminus 7, \shortminus 9) \\
&\nabla(\shortminus 6,\shortminus 6,\shortminus 6, \shortminus 10) \\
&\nabla(\shortminus 6,\shortminus 6,\shortminus 8, \shortminus 8) \\
&\nabla(\shortminus 7,\shortminus 7,\shortminus 7, \shortminus 7).
\end{aligned}
\right.
\end{equation*}
which are all $(\varphi,D)$-ample when $p \geq 31$.
\begin{proposition}
Assume that $g = 4$. 
\begin{enumerate}
\item If $p\geq 17$, then $\Lambda^{i}\Omega^1_{\Sh^{\tor}}(\log D_{\red})$ is $(\varphi,D)$-ample for $i \geq 8$ 
\item If $p \geq 31$, then $S_{(2^7)}\Omega^1_{\Sh^{\tor}}(\log D_{\red})$ is $(\varphi,D)$-ample. 
\end{enumerate}
\end{proposition}
\begin{corollary}
Assume that $g = 4$ and $p \geq 31$. If $\iota : V \hookrightarrow \Sh^{\tor}$ is a subvariety of dimension $\geq 7$ such that
\begin{enumerate}
\item $V$ is smooth,
\item $\iota^{-1}D_{\red}$ is a normal crossing divisor,
\end{enumerate}
then $V$ is of log general type with respect to $D$.
\end{corollary}

\bibliography{mybib}{}
\bibliographystyle{alpha}
\end{document}